\documentclass[reqno,11pt]{amsart}
\usepackage{setspace,tikz,xcolor,mathrsfs,listings,multicol,amssymb,amsfonts}
\usepackage{pgfplots}
\usepackage{rotating}
\usepackage[vcentermath]{youngtab}
\usepackage{tkz-base}
\usepackage{tkz-euclide}
\usepackage[margin=1in,includefoot,footskip=30pt]{geometry}
\usepackage{enumerate}
\usepackage{booktabs,subcaption}
\usepackage{gensymb}
\usepackage[centertableaux]{ytableau}
\usepackage[all,cmtip]{xy}
\usetikzlibrary{arrows,matrix,math}
\tikzset{tab/.style={matrix of math nodes,column sep=-.35, row sep=-.35,text height=7pt,text width=7pt,align=center,inner sep=2,font=\footnotesize}}

\newif\iftikz
\tikztrue

\newif\iflimitshapes
\limitshapestrue  

\usepackage[colorlinks=true, pdfstartview=FitV, linkcolor=blue, citecolor=blue, urlcolor=blue]{hyperref}

\newcommand{\imi}{\mathbf{i}} 

\newcommand{\ZZsh}{\mathbb{S}}  
\newcommand{\inner}[2]{\langle #1, #2 \rangle}
\newcommand{\iso}{\cong}

\newcommand{\qbinom}[3]{\genfrac{[}{]}{0pt}{}{#1}{#2}_{#3}}
\newcommand{\abs}[1]{\lvert #1 \rvert}
\newcommand{\Abs}[1]{\lVert #1 \rVert}

\newcommand{\bra}[1]{\langle #1 \rvert}
\newcommand{\ket}[1]{\lvert #1 \rangle}
\newcommand{\braket}[2]{\langle #1 | #2 \rangle}

\newcommand{\bigO}{\mathcal{O}} 
\newcommand{\zcrit}{z_{\mathsf{c}}}  
\newcommand{\EV}{\mathbb{E}}  
\newcommand{\dfbinom}[2]{\left(\!\!\!\binom{#1}{#2}\!\!\!\right)}  

\newcommand{\ds}{\mathrm{ds}}
\newcommand{\dt}{\mathrm{dt}}
\newcommand{\dw}{\mathrm{dw}}
\newcommand{\dx}{\mathrm{dx}}
\newcommand{\dy}{\mathrm{dy}}
\newcommand{\dz}{\mathrm{dz}}
\newcommand{\dzeta}{\mathrm{d}\zeta}
\newcommand{\dnu}{\mathrm{d}\nu}

\DeclareMathOperator{\supp}{supp} 
\DeclareMathOperator{\shape}{sh} 
\DeclareMathOperator{\He}{He} 
\DeclareMathOperator{\Li}{Li} 
\DeclareMathOperator{\Var}{Var} 
\DeclareMathOperator{\tr}{tr} 
\renewcommand{\Re}{\operatorname{Re}}
\renewcommand{\Im}{\operatorname{Im}}
\newcommand{\Pin}{\operatorname{Pin}} 
\newcommand{\Sp}{Sp}
\newcommand{\SO}{SO}
\newcommand{\Or}{O}  

\newcommand{\GL}{GL}

\newcommand{\mcC}{\mathcal{C}}

\newcommand{\mcK}{\mathcal{K}}

\newcommand{\ta}{\mathsf{a}}
\newcommand{\tb}{\mathsf{b}}
\newcommand{\tc}{\mathsf{c}}

\newcommand{\NN}{\mathbb{N}}
\newcommand{\PP}{\mathbb{P}}
\newcommand{\ZZ}{\mathbb{Z}}

\newcommand{\RR}{\mathbb{R}}
\newcommand{\CC}{\mathbb{C}}

\definecolor{darkred}{rgb}{0.7,0,0} 
\newcommand{\defn}[1]{{\color{darkred}\emph{#1}}} 

\definecolor{UQgold}{RGB}{196, 158, 54} 
\definecolor{UQpurple}{RGB}{73, 7, 94} 
\definecolor{HUgreen}{RGB}{33,98,26} 

\usepackage{listings}
\lstdefinelanguage{Sage}[]{Python}
{morekeywords={False,sage,True},sensitive=true}
\definecolor{dblackcolor}{rgb}{0.0,0.0,0.0}
\definecolor{dbluecolor}{rgb}{0.01,0.02,0.7}
\definecolor{dgreencolor}{rgb}{0.2,0.4,0.0}
\definecolor{dgraycolor}{rgb}{0.30,0.3,0.30}

\theoremstyle{plain}
\newtheorem{thm}{Theorem}[section]
\newtheorem{lemma}[thm]{Lemma}
\newtheorem{conj}[thm]{Conjecture}
\newtheorem{prop}[thm]{Proposition}
\newtheorem{cor}[thm]{Corollary}
\theoremstyle{definition}

\newtheorem{remark}[thm]{Remark}

\numberwithin{equation}{section}

\usepackage[colorinlistoftodos]{todonotes}

\begin{document}

\title{Limit shapes and fluctuations for $(\GL_{n},\GL_{k})$ skew Howe duality}

\author[D.~Betea]{Dan Betea}
\address[D.~Betea]{Université d’Angers, CNRS, LAREMA, SFR MATHSTIC, Angers, F-49045, France}
\email{dan.betea@gmail.com}
\urladdr{https://sites.google.com/view/danbetea}

\author[A.~Nazarov]{Anton Nazarov}
\address[A.~Nazarov]{Department of High Energy and Elementary Particle Physics, St.\ Petersburg State University, University Embankment, 7/9, St.\ Petersburg, Russia, 199034 and Beijing Institute of Mathematical Sciences and Applications (BIMSA),
Bejing 101408, People’s Republic of China}
\email{antonnaz@gmail.com}
\urladdr{http://hep.spbu.ru/index.php/en/1-nazarov}

\author[P.~Nikitin]{Pavel Nikitin}
\address[P.~Nikitin]{
  Beijing Institute of Mathematical Sciences and Applications (BIMSA),
Bejing 101408, People’s Republic of China}
\email{pnikitin0103@yahoo.co.uk}
\author[T.~Scrimshaw]{Travis Scrimshaw}
\address[T.~Scrimshaw]{Department of Mathematics, Hokkaido University, 5 Ch\=ome Kita 8 J\=onishi, Kita Ward, Sapporo, Hokkaid\=o 060-0808}
\email{tcscrims@gmail.com}
\urladdr{https://tscrim.github.io/}

\keywords{limit shape, skew Howe duality, lozenge tiling, Aztec diamond, $q$-Krawtchouk ensemble}
\subjclass[2010]{05A19, 60C05, 60G55}

\begin{abstract}
  We consider the probability measures on Young diagrams in the $n\times k$ rectangle obtained by piecewise-continuously differentiable specializations of Schur polynomials in the dual Cauchy identity.
  We use a free fermionic representation of the correlation kernel to study its asymptotic behavior and derive the uniform convergence to a limit shape of Young diagrams in the limit $n,k\to\infty$.
  More specifically, we show the bulk is the discrete sine kernel with boundary fluctuations generically given by the Tracy--Widom distribution with the Airy kernel.
  When our limit shape touches the boundary corner of the rectangle, the fluctuations with a second order correction are given by the discrete Hermite kernel, and we recover the discrete distribution of Gravner--Tracy--Widom (2001) restricting to the leading order.
  Finally, we demonstrate our limit shapes can have sections with no or full density of particles, where the Pearcey kernel appears when such a section is infinitely small.
\end{abstract}

\maketitle

\tableofcontents

\section{Introduction and main results}
\label{sec:introduction}


Skew Howe duality~\cite{howe1989remarks} for $(\GL_{n},\GL_{k})$ states that the natural action of the Lie group $\GL_{n} \times \GL_{k}$ on the exterior algebra $\bigwedge(\CC^{n} \boxtimes \CC^{k})$ has a multiplicity-free decomposition into irreducible representations
\begin{equation}
  \label{eq:exterior-algebra-decomposition}
  \bigwedge\left(\CC^{n}\boxtimes \CC^k \right) \iso
  \bigoplus_{\lambda} V_{\GL_{n}}(\lambda) \boxtimes V_{\GL_{k}}(\lambda'),
\end{equation}
where the sum is taken over all partitions $\lambda$ contained inside of an $n \times k$ rectangle.
By taking characters, we obtain the classical dual Cauchy identity
\begin{equation}
  \label{eq:dual-Cauchy}
  \sum_{\lambda \subseteq k^n} s_{\lambda}(x_1, \dotsc, x_n) s_{\lambda'}(y_1, \dotsc, y_k) = \prod_{i=1}^n \prod_{j=1}^k (1 + x_i y_j).
\end{equation}
The dual Cauchy identity then yields a probability measure on Young diagrams $\lambda$ inside of the $n\times k$ rectangle, with $\lambda'$ denoting a conjugate diagram, depending on the non-negative character specialization parameters $X = (x_1, \dotsc, x_n)$ and $Y = (y_1, \dotsc, y_k)$:
\begin{equation}
  \label{eq:GL-probability-measure}
  \mu_{n,k}(\lambda|\{x_{i}\}_{i=1}^{n},\{y_{j}\}_{j=1}^{k})=\dfrac{s_{\lambda}(x_1,
    \dotsc, x_n) s_{\lambda'}(y_1, \dotsc, y_k)}{\displaystyle \prod_{i=1}^n
    \prod_{j=1}^k (1 + x_i y_j)}. 
\end{equation}
The goal of this manuscript is to study the measure~\eqref{eq:GL-probability-measure} in the limit as $n, k \to \infty$.

The character $s_{\lambda}(x_{1},\dotsc,x_{n})$ is the Schur polynomial, which can be written as a sum over all semi-standard Young tableaux of the shape $\lambda$ with entries at most $n$:
\begin{equation}
  \label{eq:character}
  s_{\lambda}(x_{1},\dotsc,x_{n})=\sum_{T\in SSYT(\lambda|n)}\prod_{i=1}^{n}x_{i}^{T_{i}},
\end{equation}
where $T_{i}$ is number of boxes with the value $i$ in the tableau $T$.
An alternative way to prove the dual Cauchy identity~\eqref{eq:dual-Cauchy} is by using the dual Robinson--Schensted--Knuth (RSK) algorithm~\cite{knuth1970permutations} (in the nomenclature of~\cite[Ch.~7]{ECII}) to show pairs of semistandard Young tableaux of shape $\lambda$ with entries $\leq n$ and $\lambda'$ with entries $\leq k$ are in one-to-one correspondence with the $n \times k$ matrices of zeros and ones.
Thus, the dual RSK bijection can be used to sample random diagrams from the distribution~\eqref{eq:GL-probability-measure} (see Section~\ref{sec:skew-howe-duality} for precise details). 
Moreover, this leads to another interpretation as induced from a certain measure on lozenge tilings of the ``skew'' hexagon glued from two trapezoids or on domino tilings of the Aztec diamond glued from two rectangular parts.

Let us describe this relationship in more detail.
Each Young tableau $T$ corresponds to a tiling of a trapezoid or half hexagon in the following way.
There are $\ell$ horizontal lozenges on vertical line number $\ell$, their positions are encoded by the number of boxes in rows of the tableau $T$ with values not larger than $\ell$.
In another combinatorial language, the positions are given by the Gelfand--Tsetlin pattern corresponding to $T$.
Row lengths of the diagram $\lambda$ correspond to positions on the rightmost vertical line.
The positions on the conjugate diagrams are complementary, therefore we need to cut horizontal lozenges in triangles (on the rightmost vertical line) and glue two half hexagons together to obtain the measure~\eqref{eq:GL-probability-measure} as depicted in left panel of Fig.~\ref{fig:lozenge-aztec-tilings}.
Alternatively, we can obtain the measure~\eqref{eq:GL-probability-measure} by considering the tiling of the $(n+k)\times (n+k)$ Aztec diamond that is glued from the rectangular parts of sizes $n\times (n+k)$ and $k\times (n+k)$, depicted on the right panel of Fig.~\ref{fig:lozenge-aztec-tilings}.
Here we require that only three of four\footnote{While there are only horizontal and vertical dominoes, we consider the Aztec diamond to be a checkerboard, yielding the four types of dominoes once we take into account the checkerboard coloring.} possible types of dominoes are used to tile these rectangular parts.
Reformulating this, we study the distribution on the gluing line in this paper.
In contrast to ordinary tilings that are connected to Howe duality (corresponding to the action on $\operatorname{Sym}(\CC^n \boxtimes \CC^k)$ and yielding the Cauchy identity), these skew tilings appear to be less studied.

\begin{figure}
\includegraphics[width=0.46\linewidth]{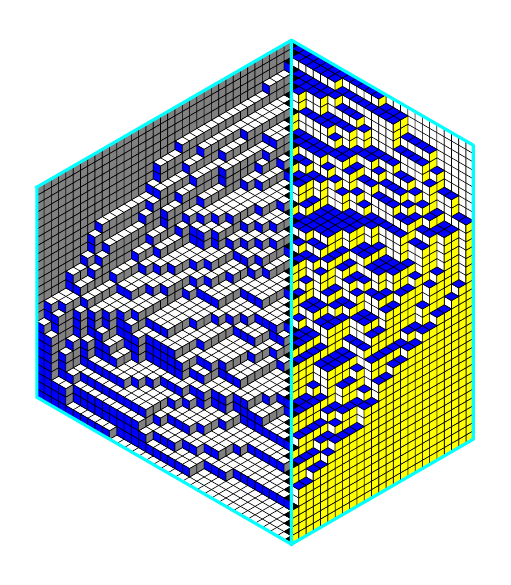}
\includegraphics[width=0.52\linewidth]{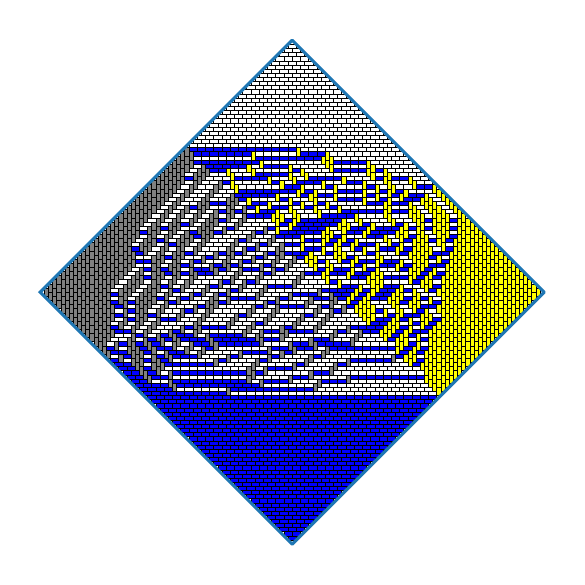}
  \caption{Lozenge tiling of a skew hexagon and corresponding domino
    tiling of Aztec diamond glued from two rectangles. Positions of
    three types of lozenges correspond to positions of
    three types of dominoes in the Aztec diamond. Only three
    types out of four are used in each part of domino tiling. }
  \label{fig:lozenge-aztec-tilings}
\end{figure}

This construction is a special case of the more general Schur process~\cite{Okounkov01}, and as such, it is known to be a determinantal point process~\cite{OR03}.
Contrary to the usual definition of Schur measure as $\mu(\lambda|X,Y)=s_{\lambda}(X)s_{\lambda}(Y)\prod_{i,j=1}^{\infty}(1-x_{i}y_{j})$ that is defined starting with Cauchy identity for Schur polynomials, we do not need to require $x_{i}, y_{j}$ to be less than $1$ for all $i,j$.
Moreover the transition to the Miwa variables $p_{1},p_{2},\ldots$, where $p_{\ell} = p_{\ell}(X) = \frac{1}{\ell} \sum_{i=1}^{\infty} x_i^{\ell}$ is the (rescaled) powersum symmetric function, is not needed in the skew case that we consider.

Our first result is integral representation for the correlation kernel that we obtain using free fermions (see, \textit{e.g.},~\cite[Ch.~14]{kac90} or~\cite{Okounkov01}) following the approach in~\cite{Okounkov01,OR03}.
In this paper, we will always use $\imi = \sqrt{-1}$ and $i$ as a variable (typically an indexing variable).

\begin{thm}[Correlation kernel]
  \label{thm:determinantal-ensemble-corr-kernel}
  Let $a_i := \lambda_i - i + \frac{1}{2}$.
  The measure
  $\mu_{n,k}( \lambda |\{x_{i}\}_{i=1}^{n},\{y_{j}\}_{j=1}^{k})$
  is a determinantal ensemble
  \begin{equation}
    \label{eq:measure-determinantal}
    \mu_{n,k}( \lambda |\{x_{i}\}_{i=1}^{n},\{y_{j}\}_{j=1}^{k})=
    \det\bigl[ \mcK(a_{i},a_{j}) \bigr]_{i,j=1}^{n},  
  \end{equation}
  where the correlation kernel $\mathcal{K}(m,m')$ has an integral representation
  \begin{equation}
    \label{eq:correlation-kernel-integral-representation}
    \mcK(m,m')=\oint\oint_{\abs{w}<\abs{z}}\frac{dz}{2\pi \imi z}
    \frac{dw}{2\pi \imi w} \frac{K(z)}{K(w)} z^{-m}w^{m'}\frac{\sqrt{zw}}{z-w},
  \end{equation}
  with
  \begin{equation}
    \label{eq:Fz-expression}
    K(z)=\prod_{i=1}^{n}\frac{1}{1-x_{i}z}\prod_{j=1}^{k}\frac{1}{1+y_{j}/z},
  \end{equation}
 and the contour for $z$ contains $-y_j$ for all $j$ and does not contain $x_i^{-1}$ for all $i$, and the contour for $w$ encircles zero. 
\end{thm}

We are interested in the limit $n,k\to\infty$ such that $\lim\frac{k}{n}=c$.
We need a way to specify the parameters $x_{i},y_{j}$ in a meaningful way for the limiting procedure to work.
Since Schur polynomials are symmetric, we can always assume the sequences $\{x_{i}\}_{i=1}^{n}$ and $\{y_{j}\}_{j=1}^{k}$ to be non-decreasing.
We can also assume $x_{i},y_{j}\neq 0$, since $s_{\lambda}(x_{1},\dotsc,x_{n},0,\dotsc,0) = s_{\lambda}(x_{1},\dotsc, x_{n})$. 
(Note that lozenge tilings of a skew hexagon and domino tilings of Aztec diamond glued from two rectangles depend upon the order of specialization parameters as each tiling correspond to a monomial in Schur polynomial; contrast Fig.~\ref{fig:gl-pearcey-lozenge-tiling} with Fig.~\ref{fig:gl-pearcey-lozenge-tiling-different-order} below.)
We derive the bulk asymptotics of the correlation kernel and demonstrate it convergence to the discrete sine kernel.

\begin{thm}[Bulk asymptotics]
  \label{thm:correlation-kernel-bulk}
  Assume that $f,g \colon [0, 1] \to \RR_{\geq 0}$ are piecewise $\mcC^1$ functions and $f(s) > 0, g(s) \geq 0$. 
  Assume the equation 
  \begin{equation}
    \label{eq:two-root-requirement-for-f-and-g}
    \int_{0}^{1} \ds \frac{f(s)}{(1-f(s)z)^2}
    - c \int_{0}^{1} \ds \frac{g(s) }{(z+g(s))^{2}} = 0
  \end{equation}
  has real roots $z^{(i)} \in \RR \sqcup \{\infty\}$ for $i = 1, \dotsc, m$.
  Set
  \begin{equation}
    \label{eq:x-plus-x-minus}
    t^{(i)} = \int_{0}^{1} \ds \frac{f(s)z^{(i)}}{1 - f(s)z^{(i)}} +
    c \int_{0}^{1} \ds \frac{g(s)}{z^{(i)}+g(s)},
  \end{equation}
  and without loss of generality, we take $-1 =: t^{(0)} \leq t^{(1)} \leq t^{(2)} \leq \cdots \leq t^{(m)} \leq t^{(m+1)} := c$.
  Then for all $a = 0, \dotsc, m$ the equation
  \begin{equation}
    \label{eq:zdz-S-eq-zero-Th}
    \int_{0}^{1} \ds \frac{f(s)z}{1 - f(s)z} + c \int_{0}^{1} \ds \frac{g(s)}{z+g(s)} - t = 0
  \end{equation}
  either has exactly two complex conjugate roots $z_{1}(t),z_{2}(t)= \overline{z_{1}(t)}$ for all $t\in (t^{(a)}, t^{(a+1)})$ or only real roots for all $t \in (t^{(a)}, t^{(a+1)})$.

  Assume that $\lim\frac{k}{n} = c$ as $n,k\to\infty$.
  Assume $a$ is such that we have two complex conjugate roots in $(t^{(a)}, t^{(a+1)})$, which we call a \defn{support interval}.
  Then for $t \in [t^{(a)},t^{(a+1)}] \subseteq [-1,c]$, integers $l,l'$, and specializing the parameters $x_{i} = f\bigl(\frac{i}{n}\bigr)$, $y_{j} = g\bigl(\frac{j}{k}\bigr)$,  we have
  \begin{equation}
    \label{eq:correlation-kernel-bulk-limit}
    \lim_{n\to\infty}\mathcal{K}(nt+l,nt+l') = \begin{cases}
        \dfrac{\sin\bigl( \pi \rho(t) \cdot (l-l') \bigr)}{\pi(l-l')} & \text{if } l \neq l', \\
        \rho(t) & \text{if } l = l',
      \end{cases}
  \end{equation}
  where $\rho(t) = \frac{1}{\pi} \arg z_{1}(t)$ is given by the
  argument of the solution $z_1(t)$ of Equation~\eqref{eq:zdz-S-eq-zero-Th} as a function of $t$.
\end{thm}

Since we can rearrange the specialization parameters to be non-decreasing, we could also assume our functions $f, g$ are non-decreasing.
We will generally assume that the intervals $[-1, t^{(1)}]$ and $[t^{(m)}, c]$ will only contian real roots.
We note there is a canonical way to extend the limit density $\rho(t)$ as either $0$ or $1$ for these intervals that only has real roots, which we will make precise in Section~\ref{sec:bulk-asymptotics}.
We will mostly concentrate on the case where there is a single support interval as the analysis for each (open) support interval is the same.
For convenience in this case, we will denote the support interval $(t_-, t_+)$ and hence $z_- := z^{(1)}$, $z_+ := z^{(2)}$ be the corresponding real roots of~\eqref{eq:two-root-requirement-for-f-and-g}.

Note that if conditions on $f, g$ are weakened, 
Equation~\eqref{eq:zdz-S-eq-zero-Th} can have several pairs of complex conjugate roots or a pair of complex conjugate roots (and several real roots).
In these cases, careful analysis of the integration contours in~\eqref{eq:correlation-kernel-integral-representation} is required.
Hence, our results are generic but not very explicit, but these conditions are satisfied in many natural examples of specializations.
Indeed, in examples we present in this paper in Sections~\ref{sec:examples} and~\ref{sec:princ-spec-q-krawtchouk}, we only have one pair of roots (subsequently, in each continuous section for the examples in Section~\ref{sec:piec-const-spec}).
Furthermore, it will be possible to deform the integration contours in such a way that only a single pair of complex conjugate roots contribute to the computations and the convergence of the correlation kernel to the discrete sine kernel is established.


As a corollary, we demonstrate that random Young diagrams under the measure~\eqref{eq:GL-probability-measure} converge in probability (for fixed $f$ and $g$) to a limit shape that is described by a solution of a certain equation.
We use the assumptions and notation of Theorem~\ref{thm:correlation-kernel-bulk}.

\begin{thm}[Limit shape]
  \label{thm:limit-shape}
  The upper boundary $F_n$ of a rotated and scaled Young diagram $\lambda$
  converges pointwise in probability with respect to the probability measure~\eqref{eq:GL-probability-measure} to the limit shape on each support interval, as defined in Equation~\eqref{eq:x-plus-x-minus}, given by the formula
  \begin{equation}
    \label{eq:limit-shape-as-integral}
    \Omega(u) = 1 + \int_{-1}^{u} \dt \; \bigl( 1-2 \rho(t) \bigr),
  \end{equation}
  where the limit density is $\rho(t)$ given above.
\end{thm}

While Theorem~\ref{thm:correlation-kernel-bulk} proves weak convergence of the probability measures, one can consider the upper boundaries of the (rotated) random diagrams as random piecewise linear functions.
It is natural to ask if this random functions converge to the limit shape in Theorem~\ref{thm:limit-shape} uniformly in probability.
From the pointwise convergence to the limit shape in probability it is possible to deduce the uniform convergence by a general argument presented in~\cite[Lemma~11]{betea2023multicritical} (see Section~\ref{sec:bulk-asymptotics}).
This leads to the next corollary. 

\begin{cor}[Uniform convergence]
  \label{cor:uniform-convergence}
  Let $F_{n}$ denote the upper boundary of a Young diagram $\lambda$ rotated and scaled by $\frac{1}{n}$ and regarded as a function $F_n(u)$ of $u\in [-1, c]$.
  Then the functions $\{F_n\}_{n=1}^{\infty}$ converge in probability with respect to the probability measure~\eqref{eq:GL-probability-measure} in the supremum norm $\Abs{\cdot}_{\infty}$ to the limiting shape $\Omega(u)$ given by the formula~\eqref{eq:limit-shape-as-integral}; that is, for any $\varepsilon>0$
  \begin{equation}
    \label{eq:uniform-convergence}
    \lim_{n\to\infty}\PP[\mathrm{sup}_{u}|F_{n}(u)-\Omega(u)|>\varepsilon]= 0.
  \end{equation}
\end{cor}

We also consider boundary asymptotics of the correlation kernel and demonstrate that in general it is described by the Airy kernel $\mathcal{K}_{\mathrm{Airy}}$; hence is is distributed according to the Tracy--Widom GUE distribution~\cite{TW94}, which we denote by $F_{\rm GUE}$.
In particular, we have an analogue of the Baik--Deift--Johannson asymptotics~\cite{BDJ99} for the fluctuations around $t_{+}$ of the random diagram in the generic case.
The realization of these fluctuations depend on the limit shape solutions, which is either $\lambda_1$ or the number of parts that are equal to $k$, both of which occur when $t_{+} < c$.

\begin{thm}[Generic boundary asymptotics]
\label{thm:boundary-asymptotics}  
With setup as in Theorem~\ref{thm:limit-shape}, consider the case when $t_{+} < c$, and the asymptotic regime $m\approx t_{+}n+\xi n^{\frac{1}{3}} \sigma^{-1}$, $m'\approx t_{+}n+\eta n^{\frac{1}{3}} \sigma^{-1}$ for some explicit constant $\sigma$ as $n\to\infty$.  Denote by $\zcrit$ the real root of Equation~\eqref{eq:zdz-S-eq-zero-Th} that corresponds to $t_{+}$,  then 
\begin{equation}
  \label{eq:airy-kernel-intro}
  \lim_{n\to\infty} \sigma^{-1}n^{\frac{1}{3}}\zcrit^{m-m'}\mathcal{K}(m,m')=\mathcal{K}_{\mathrm{Airy}}(\xi,\eta) = \iint \frac{\dzeta \dnu}{(2\pi i)^{2}} \frac{\exp\left(\frac{\zeta^{3}}{3}-\zeta\xi\right)}{\exp\left(\frac{\nu^{3}}{3}-\nu\eta\right)}\frac{1}{\zeta-\nu}.
\end{equation}
  Let $L = \begin{cases} \lambda_1, &\text{if $\Omega$ convex around $t_{+}$}, \\ n-|\{i \mid \lambda_i = k\}|, &\text{if $\Omega$ concave around $t_{+}$}. \end{cases}$
  Then we have
  \begin{equation}
 \label{eq:tracy-widom-observable}    
  \lim_{n \to \infty} \PP \left( \frac{L - t_{+} n}{\sigma^{-1} n^{1/3}} \right) = F_{\text{GUE}} (s).
  \end{equation}
Similar result holds for left boundary $t_{-}$ when $t_{-}> -1$. 
\end{thm}
Note that the term $\zcrit^{m-m'}$ does not contribute to correlation functions and gap probability as it is canceled in determinant computations.
This applies to Pearcey and Hermite kernels discussed below as well. 

Our next result is a method to produce an the appearance of the Pearcey kernel.
To do so, we need to consider the case where the limit density $\rho(t)$ has multi-interval support of $(t_-, t_+)$ and $(t'_-, t'_+)$ with $-1 < t_i < t_+ \leq t_d \leq t'_- < t'_+ \leq c$ for some $t_d$.
Note that on the intervals $(t_+, t_d)$ and $(t_d, t'_-)$ will have a density of $0$ or $1$, which should be the same, depending on the behaviors in each part.
By choosing our $f, g$ such that $t_+ = t_d = t'_-$, we produce the following behavior.
The intersection point $t_d$ corresponds to the higher order root of the action $S(z,t)=\lim_{n,k\to\infty}\frac{1}{n}\ln K(z)-t\ln z$ and local fluctuations near it are no longer described by the Airy kernel, but by the Pearcey kernel $\mcK_{\mathrm{Pearcey}}$~\cite{brezin1998universal,tracy2006pearcey,okounkov2007random}.

\begin{thm}[Generic Pearcey asymptotics]
\label{thm:Pearcey_kernel}
Assume we have two support intervals $(t_-, t_d)$ and $(t_d, t_+)$. Denote by $\zcrit$ the real root of Equation~\eqref{eq:zdz-S-eq-zero-Th} that corresponds to $t_{d}$. 
Consider the asymptotic regime $m \approx t_d n+\xi n^{\frac{1}{4}}\sigma^{-1}$, $m' \approx t_d n+\eta n^{\frac{1}{4}}\sigma^{-1}$ for some constant $\sigma$ as $n\to \infty$.
Then
\begin{equation}
  \label{eq:pearcey-kernel}
  \lim_{n,k\to\infty}\sigma^{-1}n^{\frac{1}{4}}\zcrit^{m-m'} \mathcal{K}(m,m') = \mathcal{K}_{\mathrm{Pearcey}}(\xi,\eta) = \iint \frac{\dzeta \dnu}{(2\pi i)^{2}} \frac{\exp\left(\frac{\zeta^{4}}{4}-\zeta\xi\right)}{\exp\left(\frac{\nu^{4}}{4}-\nu\eta\right)}\frac{1}{\zeta-\nu}.
\end{equation}
\end{thm}

Therefore, we have the Pearcey process any time two support intervals share a common boundary point.
In addition, we conjecture that looking at the three-dimensional picture of the entire lozenge tiling (or Aztec diamond) at this point is described by the extended Pearcey kernel of~\cite{okounkov2007random,tracy2006pearcey}.

There is a special case when the boundary asymptotics change.
This occurs when the limit shape touches the corner of the diagram (up to a second order approximation), where it has (very) different behavior given by the discrete Hermite kernel
\begin{equation}
 \mcK^{\He}_{s}(l,l') := \frac{1}{\sqrt{\pi(l-1)! (l'-1)!}}\int_{s}^{\infty} \dt \; e^{-\frac{t^{2}}{2}}\He_{l-1}(t)\He_{l'-1}(t)
\end{equation}
from~\cite{borodin2007asymptotics}, where $\He_l(t)$ is the (probabilist's) Hermite polynomial.

\begin{thm}[Boundary asymptotics near the corner]
  \label{thm:near-corner-asymptotics}
Consider the asymptotic regime $n\to\infty$ with $k = cn+\frac{\widetilde{s}}{\tau}\sqrt{n} + o(\sqrt{n})$, where $\widetilde{s}$ is a parameter and $\tau=\frac{1}{\sqrt{\Abs{f}_{2}^{2}+c \Abs{g^{-1}}_{2}^{2}}}$ is a normalization constant. 
Take $m = \lfloor cn+\frac{\widetilde{s}}{\tau}\sqrt{n}\rfloor-l+\frac{1}{2}$ and $m' = \lfloor cn+\frac{\widetilde{s}}{\tau}\sqrt{n}\rfloor-l'+\frac{1}{2}$ so that $l,l'\in\ZZ$.
We assume
\begin{equation}
  \label{corner-condition}
  \int_{0}^{1} \dt \, f(t) = c\int_{0}^{1}\frac{\dt}{g(t)}.
\end{equation}
Alternatively, these conditions can be written as
\begin{equation}
\lim_{n,k\to\infty}\frac{k}{n}=c,
\qquad\qquad
\sum_{i=0}^{n-1}f(i/n)=c\sum_{j=0}^{k-1}g(j/k)^{-1}+\frac{\widetilde{s}}{\tau}\sqrt{n}.
\end{equation}
With this scaling and conditions we have full support case of $t_{+} = c$ and the correlation kernel converges to the discrete Hermite kernel with parameter $s := \widetilde{s}\int_{0}^{1}\frac{\dt}{g(t)}$:
\begin{equation}
  \lim_{n\to\infty}n^{\frac{l-l'}{2}}\mcK(m,m')=\tau^{l-l'} \sqrt{\frac{(l'-1)!}{(l-1)!}} \mcK^{\He}_{s}(l,l').
\end{equation}
For $\Delta \in \NN$,  the probability distribution of the length of the first row of the diagram is given by the determinant
  \begin{equation}
  \label{eq:discrete_distribution}
  \lim_{n \to \infty} \PP (\lambda_1 - n c \leq -\Delta) = \det_{0 \leq l, l' \leq \Delta-1} [ \delta_{l, l'} -  \mcK^{\He}_{s}(l,l') ]
  \end{equation}
Similarly, if we have
\begin{equation}
  \label{eq:left-corner-condition}
  \int_{0}^{1}\frac{\dt}{f(t)}=c\int_{0}^{1} \dt \; g(t),
\end{equation}
then $t_{-}=-1$ and we have fluctuations described by the discrete Hermite kernel in the left corner of the rectangle. 
\end{thm}

We also have the following conjecture characterizing the critical case.
We phrase it only in terms of $t_{+}$, but due to the symmetry in the system, we would have the analogous statement for $t_{-}$ for the boundary $-1$.

\begin{conj}
\label{conj:critical_classification}
The following are equivalent:
\begin{enumerate}
\item $t_{+} = c$ (that is, the limit shape has a support interval ending at the right boundary);
\item $z_+ = 0$;
\item $\Omega'(t_{+}) = 0$;
\item $\displaystyle \int_{0}^{1} \ds \, f(s) = c\int_{0}^{1}\frac{\ds}{g(s)}$.
\end{enumerate}
\end{conj}

As evidence for Conjecture~\ref{conj:critical_classification}, when $\Omega$ is convex (resp.\ concave), all of the examples computed show that $z_+ \geq 0$ (resp.\ $z_+ \leq 0$), and we believe this is an equivalence that holds in general.
Consequently, the critical case would correspond to when $z_+ = 0$, which is clearly a sufficient condition.

Let us consider the boundary asymptotitcs in the corner with only a first order approximation, which amounts to taking $s = \widetilde{s} = 0$.
In this case, we show that discrete Hermite kernel $\mcK_0^{\He}$ becomes the discrete kernel from the critical case of Gravner--Tracy--Widom~\cite{GTW01}.
In fact, we show a stronger statement, that the matrices defining the kernels are equal up to a overall simple factor (that depends on the diagonal).

\begin{thm}[Critical boundary asymptotics in the corner]
  \label{thm:corner-asymptotics}
  With the scaling of Theorem~\ref{thm:limit-shape}, in the critical support case of $t_{+} = c$ and for $\Delta \in \NN$, we have
  \begin{equation}
  \label{eq:discrete_distribution_s0}
  \lim_{n \to \infty} \PP (\lambda_1 - n c \leq -\Delta) = \det_{0 \leq i, j \leq \Delta-1} [\delta_{i, j} - K_{\rm crit}(i,j)]
  \end{equation}
  with the matrix
  \begin{equation}
  K_{\rm crit} (i, j) = \sum\limits_{\ell = 0}^{(\Delta - j - 1)/2} \begin{cases} \frac{1}{2 \pi}  \frac{1}{\ell!} \sin \frac{\pi (j-i)}{2} \Gamma(\ell + \frac{j-i}{2}), & \text{if } \ell + \frac{j-i}{2} \notin \ZZ_{\leq 0}, \\ \frac{1}{2}  \frac{(-1)^\ell}{\ell! (\frac{i-j}{2} - \ell)!}, & \text{if } \ell + \frac{j-i}{2} \in \ZZ_{\leq 0}. \end{cases}
  \end{equation}
\end{thm}

We remark that we encounter the same phase transition that was observed in~\cite{GTW01}.
Indeed, one side of their transition is a deterministic regime as~\cite{GTW01} only considers the behavior of $\lambda_1$, but this corresponds to $\lambda_1 = k$ fixed.
On the other hand, the fluctuations for the number of rows of length $k$ are described by the Tracy--Widom distribution.

Let us discuss how our results relate with the literature.
If we take $x_i = \alpha$ and $y_j = 1$ for some positive constant $\alpha$, then this case has been considered previously in~\cite{borodin2017asep,Johansson01,GTW01,GTW02,GTW02II,nazarov2022skew}.
We perform an explicit analysis of this example in Section~\ref{sec:const-funct-krawtch}, which we then extend to the piecewise constant case in Section~\ref{sec:piec-const-spec}.
By using the relationship with Aztec diamonds~\cite{Johansson2002non}, the limit shape correlation kernel we derive has appeared in~\cite{BBCCR17}.
When specializing $x_i = y_i = 1$, we obtain the $\GL_n \times \GL_k$ results in~\cite{nazarov2021skew}, which were also previously known and studied (sometimes under different guises such as the Krawtchouk ensemble) in~\cite{BO05,BO05gamma,BO05II,BO06,borodin2007asymptotics,GTW01,ismail1998strong,Johansson01,Johansson2002non,pittel2007limit,panova2018skew,Sepp98}, although this list is likely not exhaustive.
For more precise details, see, \textit{e.g.},~\cite[Sec.~5.6.1]{nazarov2021skew}.

For the principal specializations $x_{i} = q^{i-1},\; y_{j}=q^{j-1}$ and $x_{i}=q^{i-1},\; y_{j}=q^{-j+1}$, which correspond to exponential functions, we recover the $q$-Krawtchouk polynomial ensemble.
We also take the limit $q\to 1$ in such a way that $\lim n(q-1) = \gamma$, the diagrams converge to the corresponding limit shape.
In these cases the equations can be solved explicitly, so as a consequence of Theorem~\ref{thm:limit-shape}, we obtain another proof of the limit shapes from~\cite{nazarov2022skew}, where they were derived using the $q$-difference equations for the $q$-Krawtchouk polynomials in~\cite{nazarov2022skew}.

For the usual Schur measure, the simplest way is to take one non-zero value in the sequence of Miwa variables  $p_{\ell}(X) = \frac{1}{\ell} \sum_{i=1}^{\infty} x_i^{\ell}$ with $p_{1}=\xi,p_{2}=p_{3}=\dots=0$ to obtain the poissonization of the Plancherel measure \cite{borodin2000asymptotics}.
This can be seen as letting $x_{i}=\frac{\xi}{n}, y_{j}=\frac{\xi}{k}$ and taking the limit $n,k\to\infty$, which we can also undertake for the skew case.
However, it produces the same results as the diagram does not ``feel'' the $n\times k$ rectangle it is confined in.

We expect the discrete Hermite kernel to be universal for the fluctuations at the corner.
As per Conjecture~\ref{conj:critical_classification}, we believe for our class of functions $f,g$ the limit shape must take a flat approach in the corner; that is, $\Omega'(c) = 0$.
However, if we allow $f,g$ to grow to infinity, we can obtain non-flat approach to the corner, which we demonstrate in Section~\ref{sec:constant-q} by taking a generalized principal specialization with constant $q$.
In this case, the limit shape is a linear function, and we conjecture that corner fluctuations are described by a $q$-analogue of the discrete Hermite kernel.


The paper is organized as follows.
In Section~\ref{sec:skew-howe-duality} we discuss skew Howe duality, the dual Cauchy identity for $\GL_{n} \times \GL_{k}$ characters, and sampling of random Young diagrams with respect to the probability measure~\eqref{eq:GL-probability-measure}. 
We also explain a combinatorial non-intersecting lattice path realization of the measure~\eqref{eq:GL-probability-measure} and its two graphic representations as domino tilings of the Aztec diamond with gluing condition and as lozenge tilings of the hexagon with gluing condition along the diagonal.
In Section~\ref{sec:free-ferm-repr} we explain free fermionic represenation for the ensemble and prove Theorem~\ref{thm:determinantal-ensemble-corr-kernel}.
Next, we discuss the asymptotics of our measure by splitting into three parts.
The first part is Section~\ref{sec:bulk-asymptotics}, where we study bulk asymptotics of the correlation kernel and prove Theorem~\ref{thm:correlation-kernel-bulk} with Theorem~\ref{thm:limit-shape} and Corollary~\ref{cor:uniform-convergence}.
Then in Sections~\ref{sec:boundary-asymptotics},~\ref{sec:pearcey-asymptotics},~\ref{sec:asymptotics-near-the-corner},~\ref{sec:asymptotics-at-the-corner}, we discuss the asymptotics of the correlation kernel near and at the boundary and prove Theorems~\ref{thm:boundary-asymptotics},~\ref{thm:Pearcey_kernel},~\ref{thm:near-corner-asymptotics},~\ref{thm:corner-asymptotics}, respectively.
In Section~\ref{sec:examples}, we give a number of examples of our results for certain specializations:
\begin{itemize}
\item constant (Section~\ref{sec:const-funct-krawtch});
\item piecewise constant (Section~\ref{sec:piec-const-spec});
\item general monomials (Section~\ref{sec:monomial-functions});
\item alternating parameters and the relation to symplectic Young diagrams (Section~\ref{sec:gl=2sp}).
\end{itemize}
The relation of principal specialization of the $\GL_{n} \times \GL_{k}$-characters to the $q$-Krawtchouk ensemble is discussed in Section~\ref{sec:princ-spec-q-krawtchouk}.
In particular, in Section~\ref{sec:uniform-convergence}, we provide an alternative proof of Corollary~\ref{cor:uniform-convergence} on the uniform convergence of Young diagrams to the limit shape for the principal specialization by a direct computation.
In the limit for the $q$-Krawtchouk ensemble, we also require to take the limit $q \to 1$ as $n,k \to \infty$, but we also consider a limiting case of this behavior with $q$ being a constant in Section~\ref{sec:constant-q}. Note, that $q=\mathrm{const}$ does not satisfy our general assumptions as there are no functions $f(s)$ and $g(s)$ corresponding to this specialization. Therefore Theorems \ref{thm:correlation-kernel-bulk}, \ref{thm:limit-shape}, \ref{thm:boundary-asymptotics}, \ref{thm:near-corner-asymptotics}, \ref{thm:corner-asymptotics} are not applicable, but we still describe the limit shape and conjecture the fluctuations in the corner. 
We conclude by presenting some open problems related to the results of this paper in Section~\ref{sec:conclusion-outlook}.

\section*{Acknowledgements}

The authors thank Daniil Sarafannikov for deriving formula \eqref{eq:sp_half_limit_density}.
The authors thank J\'er\'emie Bouttier, Janko Gravner, Arno Kuijlaars, Nicolai Reshetikhin, Walter van Assche, and Anatoly Vershik for useful conversations.
The authors thank Cesar Cuenca and Matteo Mucciconi for useful comments on an earlier draft.

Dan Betea was supported by ERC grant COMBINEPIC No.~759702.
Anton Nazarov was supported by the Russian Science Foundation under grant No.~21-11-00141.
Travis Scrimshaw was partially supported by Grant-in-Aid for JSPS Fellows 21F51028 and for Scientific Research for Early-Career Scientists 23K12983.

\section{Skew Howe duality and dual RSK algorithm}
\label{sec:skew-howe-duality}

A \defn{partition} $\lambda = (\lambda_1, \lambda_2, \lambda_3, \ldots)$ is a weakly decreasing sequence of nonnegative integers with only finitely many nonzero entries.
We will consider our partitions to be given to their Young diagrams, which we draw using English convention.
We let $\lambda'$ denote the \defn{conjugate} shape, given by reflecting over the $y = -x$ line.
A \defn{tableau} $T$ is a filling of the Young diagram of $\lambda$ by nonnegative integers, and it is \defn{semistandard} if the rows weakly increase from left-to-right and columns strictly increase from top-to-bottom.
Define $\shape(T) = \lambda$ to be the \defn{shape} of $T$.

We can encode a basis element of $\bigwedge\left(\CC^{n}\otimes \CC^{k}\right)$ of the form
\begin{equation}
(e_{i_1}\otimes e_{j_1}) \wedge (e_{i_2}\otimes e_{j_2}) \wedge \cdots \wedge (e_{i_{\ell}}\otimes e_{j_{\ell}})
\end{equation}
with $(i_k, j_k) < (i_{k+1}, j_{k+1})$ in lexicographic order for all $1 \leq k \leq \ell$, as a $\{0,1\}$-matrix $M$ by $M_{i_k,j_k} = 1$ for all $k$ and $0$ otherwise.
Subsequently, we can represent $M$ as the corresponding pair of sequences of row numbers $(i_k)_{k=1}^{\ell}$ and column numbers $(j_k)_{k=1}^{\ell}$, or written as a biword:
\begin{equation}
\begin{bmatrix}
i_1 & i_2 & i_3 & \cdots & i_{\ell} \\
j_1 & j_2 & j_3 & \cdots & j_{\ell}
\end{bmatrix}.
\end{equation}

Next, we describe the \defn{dual RSK bijection}~\cite{knuth1970permutations} (see also~\cite[Ch.~7]{ECII}) that can be considered as a combinatorial realization the decomposition~\eqref{eq:exterior-algebra-decomposition} (see, \textit{e.g.},~\cite{BS17,Lothaire02}).
We start with a pair of empty semistandard tableau $(P_0, Q_0)$, and proceed inductively on $m = 0, 1, \dotsc, \ell$ as follows.
Consider $(P_m, Q_m)$, and we perform the following modified Schensted insertion algorithm~\cite{knuth1970permutations} starting with $j_m$ starting from the top row.
To row $r$, we try to insert letter $x$ as follows:
\begin{itemize}
\item If $x$ is strictly larger than every other letter in $r$ (including if $r$ is empty), we add $x$ to the end of $r$ and terminate.
\item Otherwise, find the smallest $y \geq x$, replace the leftmost such occurrence in $r$, and insert $y$ into to the row below $r$.
\end{itemize}
Let $P_{m+1}$ be the resulting semistandard tableau, and define $Q_{m+1}$ as $Q_m$ with adding a box with entry $i_m$ to $Q$ such that $\shape(Q_{m+1})' = \shape(P_{m+1})$.
This gives a bijection between the $m \times n$ $\{0,1\}$-matrices that index basis elements of $\bigwedge\left(\CC^{n}\otimes \CC^{k}\right)$ and pairs of semistandard tableau $(P, Q)$ such that $\shape(P)' = \shape(Q)$ with the entries in $P$ (resp.\ $Q$) being at most $n$ (resp.\ $m$).
The tableau $P$ (resp.\ $Q$) is known as the \defn{insertion tableau} (resp.\ \defn{recording tableau}).

Using this, we can obtain a sampling algorithm for the random diagram with distribution~\eqref{eq:GL-probability-measure} with parameters $x_1, \dotsc, x_n$, $y_1, \dotsc, y_k$ given as follows.
We form the random $n\times k$ matrix $M$ so that probability to have 0 at the position $(i,j)$ is $(1+x_{i}y_{j})^{-1}$ (thus it is $1$ with probability $\frac{x_i y_j}{1 + x_i y_j}$).
We then apply dual RSK to $M$ and taking the shape $\shape(P)$ of the insertion tableau $P$ (equivalently, of the recording tableau) gives the sampling algorithm.
All of our random Young diagrams will be obtained using this sampling algorithm.

\section{Free fermionic representation}
\label{sec:free-ferm-repr}

Another way to represent a partition $\lambda = (\lambda_1, \lambda_2, \dotsc)$ is by using the coordinates $a_i = \lambda_i - i + \frac{1}{2}$ on the shifted integer lattice $\ZZsh := \ZZ + \frac{1}{2}$.
Furthermore, the sequences $(a_i \in \ZZsh)_{i=1}^{\infty}$ such that $a_i > a_{i+1}$ for all $i$ and there exists an $\ell$ such that $a_i = -i + \frac{1}{2}$ for all $i > \ell$ are in bijection with partitions; in particular, we have $\ell(\lambda) \leq \ell$.
The set $(a_i)_{i=1}^{\lambda}$ corresponding to a partition $\lambda$ is known as the Maya diagram of $\lambda$.
This can be visually constructed by putting a dot on each horizontal edge of the \emph{conjugate} Young diagram (including all the infinite number of edges at the top), then rotating it so that the corner points down touching at $0 \in \RR$ (known as Russian convention), and then projecting onto the line.
See Fig.~\ref{fig:maya_diagram} for an example.

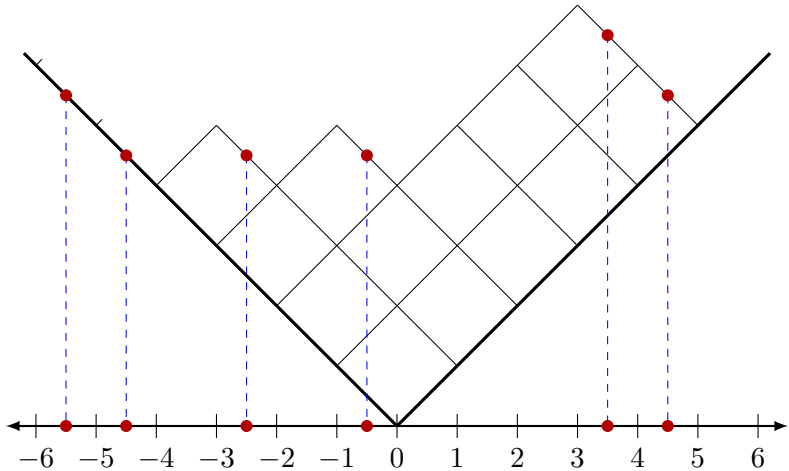
\begin{figure}
\[
\begin{tikzpicture}[scale=.8,>=latex]
\draw[-,very thick] (0,0) -- ++(6.2,6.2);
\draw[-,very thick] (0,0) -- ++(-6.2,6.2);
\draw[<->,thick] (-6.5,0) -- (6.5,0);
\foreach \la/\i in {4/1,3/2,2/3,2/4,2/5}
  \draw[-] (\i,\i) -- ++(-\la,\la);
\foreach \la/\i in {5/1,5/2,2/3,1/4} {
  \draw[-] (-\i,\i) -- ++(\la,\la);
  \draw[color=blue,dashed] (-\i,\i) + (\la+.5,\la-.5) -- (\la-\i+.5,0);
  \fill[color=darkred] (-\i,\i) + (\la+.5,\la-.5) circle (0.1);
  \fill[color=darkred] (\la-\i+.5,0) circle (0.1);
}
\foreach \i in {5,6} {
  \draw[-] (-\i,\i) -- ++(.1,.1);
  \draw[color=blue,dashed] (-\i+.5,\i-.5) -- (-\i+.5,0);
  \fill[color=darkred] (-\i+.5,\i-.5) circle (0.1);
  \fill[color=darkred] (-\i+.5,0) circle (0.1);
}
\foreach \i in {-6,-5,...,6} {
  \draw[-] (\i,0.2) -- (\i,-0.2) node[anchor=north] {$\i$};
}
\end{tikzpicture}
\]
\caption{The construction of the Maya diagram of the partition $(5,5,2,1)$.}
\label{fig:maya_diagram}
\end{figure}

We want to consider the elements in the Maya diagram as fermions, particles where no two occupy the same position.
Thus we encode the positions of particles as basis vectors in the infinite wedge space $\bigwedge \CC[\ZZsh]$.
This has a subspace $V$ called \defn{fermionic Fock space} defined as the span of
\begin{equation}
\{ v_{a_1} \wedge v_{a_2} \wedge \cdots \mid a_1 > a_2 > \cdots \text{ and there exists $\ell, C \in \ZZ$ such that } a_i = -i + C + \frac{1}{2} \text{ for all } i > \ell \}.
\end{equation}
In particular, this gives us the classical correspondence between Maya diagrams $(a_i)_{i=1}^{\infty}$ and the subspace of vectors in $V_0 \subseteq V$ spanned by those with $C = 0$ above, and for a partition $\lambda$, we write the corresponding basis element as $\ket{\lambda}$.
This space $V$ is an irreducible representation of the infinite dimensional Clifford algebra generated by $\{\psi_i, \psi_i^{\dagger}\}_{i \in \ZZsh}$ that satisfy the canonical commutation relations:
\begin{equation}
\psi_i \psi_j + \psi_j \psi_i = \psi^{\dagger}_i \psi^{\dagger}_j + \psi^{\dagger}_j \psi^{\dagger}_i = 0,
\qquad\qquad
\psi_i \psi^{\dagger}_j + \psi^{\dagger}_j \psi_i = \delta_{ij},
\end{equation}
where the action on $V$ is given by
\begin{subequations}
\begin{align}
\psi_i \cdot (v_{i_1} \wedge v_{i_2} \wedge \cdots) & = v_i \wedge v_{i_1} \wedge v_{i_2} \wedge \cdots,
\\
\psi_i^{\dagger} \cdot (v_{i_1} \wedge v_{i_2} \wedge \cdots) & = \begin{cases} (-1)^{j-1} v_{i_1} \wedge \cdots \wedge \widehat{v}_{i_j} \wedge \cdots & \text{if } i_j = i, \\ 0 & \text{otherwise,} \end{cases}
\end{align}
\end{subequations}
where $\widehat{v}_{i_j}$ denotes that the vector is not present.
Because of this action, the operators $\psi_i$ and $\psi_i^{\dagger}$ are known as elementary (fermionic) \defn{creation} and \defn{annihilation} operators, respectively.
There are also the \defn{current operators} $\{\alpha_{\ell}\}_{\ell \in \ZZ}$ defined by\footnote{Special care is needed for the case $\ell = 0$, where we need a normal ordering on the product. The result is that $\alpha_0$ returns the value $C$ of any basis vector of $V$. However, we do not use $\alpha_0$ here; so we omit the normal ordering.}
\begin{equation}
  \label{eq:boson-fermion}
  \alpha_{\ell}=\sum_{j \in \ZZsh}\psi_{j-\ell}\psi^{\dagger}_j,
\end{equation}
and these satisfy commutation relations of Heisenberg algebra
\begin{equation}
  \label{eq:heisenberg-algebra}
  [\alpha_{m}, \alpha_{\ell}] = m\delta_{m+\ell, 0}.
\end{equation}
Equation~\eqref{eq:boson-fermion} is one half of the \defn{boson-fermion correspondence}, and $V_0$ becomes a representation of the Heisenberg algebra.
For more information, see, \textit{e.g.},~\cite[Ch.~14]{kac90} or~\cite{Okounkov01}. 

There is an anti-involution on the Clifford algebra defined by $\psi_i \longleftrightarrow \psi_i^{\dagger}$, which also  sends $\alpha_{\ell} \longleftrightarrow \alpha_{-\ell}$, that allows us to define a dual representation $V^{\dagger}$.
We will denote the dual basis element of $\ket{\lambda}$ as $\bra{\lambda}$, and we denote the natural pairing by $\braket{\mu}{\lambda} = \delta_{\lambda\mu}$.
Furthermore, this pairing has the property that for any operator $\Psi$, the notation
\begin{equation}
\bra{\mu} \Psi \ket{\lambda} = (\bra{\mu} \Psi) \ket{\lambda} = \bra{\mu} (\Psi \ket{\lambda})
\end{equation}
is unambiguous.
This can be extended to all of $V$, but we will only be concerned with $V_0$.

Next, we define generating series that are formal \defn{fermion fields}
\begin{equation}
\psi(z) = \sum_{j \in \ZZsh} \psi_j z^j,
\qquad\qquad
\psi^{\dagger}(w) = \sum_{j \in \ZZsh} \psi^{\dagger}_j w^{-j}.
\end{equation}
Let $X = (x_1, x_2, \dotsc)$ and $Y = (y_1, y_2, \dotsc)$.
We also define \defn{half-vertex operators} as
\begin{equation}
  \label{eq:gamma-definition}
  \Gamma_{\pm}(X) = \exp\left(\sum_{\ell=1}^{\infty}p_{\ell}(X)\alpha_{\pm \ell}\right),
\end{equation}
where $p_{\ell}(X) = \frac{1}{\ell} \sum_{i=1}^{\infty} x_i^{\ell}$ are the Miwa variables, which are a rescaled version of the powersum symmetric functions.
Note that the Clifford algebra anti-involution sends $\Gamma_+(X) \longleftrightarrow \Gamma_-(X)$.
From the definition and Heisenberg relations~\eqref{eq:heisenberg-algebra}, we have
\begin{subequations}
\begin{gather}
[\Gamma_+(X), \Gamma_+(Y)] = [\Gamma_-(X), \Gamma_-(Y)] = 0,
\\
\Gamma_+(X) \ket{0} = \ket{0},
\qquad\qquad
\bra{0} \Gamma_-(Y) = \bra{0}.
\end{gather}
\end{subequations}
Furthermore, by the Baker--Campell--Hausdoff formula, the half-vertex operators satisfy
\begin{subequations}
\begin{align}
\Gamma_+(X) \Gamma_-(Y) & = H(X; Y) \Gamma_-(Y) \Gamma_+(X),
\\ \Gamma_{\pm}(X) \psi(z) & = H(X; z^{\pm}) \psi(z) \Gamma_{\pm}(X),
\\ \Gamma_{\pm}(X) \psi^{\dagger}(w) & = H(X; w^{\pm1})^{-1} \psi^{\dagger}(w) \Gamma_{\pm}(X),
\end{align}
\end{subequations}
where
\begin{equation}
H(X; Y) = \prod_{i,j} \frac{1}{1 - x_i y_j},
\qquad\qquad
E(X; Y) = \prod_{i,j} (1 + x_i y_j).
\end{equation}

We can write Schur polynomials as the matrix coefficients (see, \textit{e.g.},~\cite{Okounkov01})
\begin{equation}
  \label{eq:schur-through-gamma}
  \bra{0} \Gamma_+(X) \ket{\lambda} = s_{\lambda}(X)
\end{equation}
by using Wick's theorem and the Jacobi--Trudi formula.
From the Clifford algebra anti-involution (or duality), we also have
\begin{equation}
\bra{\lambda} \Gamma_-(X) \ket{0} = s_{\lambda}(X).
\end{equation}
There is an involution on symmetric functions defined by $\omega s_{\lambda}(Y) = s_{\lambda'}(Y)$, which sends $p_{\ell}(Y) \longleftrightarrow (-1)^{\ell-1} p_{\ell}(Y)$.
If we formally apply $\omega$ to the half-vertex operator $\Gamma_{\pm}(Y)$, we obtain the half-vertex operator
\begin{equation}
  \label{eq:gamma-prime-definition}
\Gamma'_{\pm}(Y) := \exp\left(\sum_{\ell=1}^{\infty}(-1)^{\ell-1} p_{\ell}(Y) \alpha_{\pm \ell}\right) = \Gamma^{-1}_{\pm}(-Y).
\end{equation}
By the dual Jacobi--Trudi formua and Wick's theorem, we have
\begin{equation}
  \label{eq:schur-conj-through-gamma}
  \bra{0} \Gamma'_+(Y) \ket{\lambda} = s_{\lambda'}(Y).
\end{equation}
Note that~\eqref{eq:schur-conj-through-gamma} with~\eqref{eq:schur-through-gamma} implies the nontrivial relation $\Gamma'_+(Y) \ket{\lambda} = \Gamma_+(Y) \ket{\lambda'}$.

Next we have the dual Cauchy identity by
\begin{align}
\sum_{\lambda} s_{\lambda}(X) s_{\lambda'}(Y) & = \sum_{\lambda} \bra{0} \Gamma_+(X) \ket{\lambda} \cdot \bra{\lambda} \Gamma'_-(Y) \ket{0} \nonumber
\\
& =\bra{0} \Gamma_+(X) \Gamma'_-(Y) \ket{0} = \bra{0} \Gamma_+(X) \Gamma_-^{-1}(-Y) \ket{0} \nonumber
\\ & = H(X; -Y)^{-1} \bra{0} \Gamma_-^{-1}(-Y)  \Gamma_+(X) \ket{0} \nonumber
\\ & = E(X; Y) \braket{0}{0} = E(X; Y).
\end{align}
The correlation kernel is then obtained by commuting $\Gamma'_{-}$ and
fermionic operators $\psi_{m}$ in the correlator
$\bra{0} \Gamma_{+}(t) \psi_{m} \psi^{\dagger}_{m'} \Gamma'_{-}(t') \ket{0}$ as
\begin{equation}
  \label{eq:correlation-kernel}
  \mcK(m,m') = \oint\oint_{\abs{w} < \abs{z}} \frac{\dz}{2\pi \imi z}
  \frac{\dw}{2\pi \imi w} \frac{K(z)}{K(w)}
  z^{-m}w^{m'}\frac{\sqrt{zw}}{z-w},
\end{equation}
where the contour for $z$ contains $-y_j$ for all $j$ and does not contain $x_i^{-1}$ for all $i$, and the contour for $w$ encircles zero and
\begin{equation}
  \label{eq:Fz-expression_ps}
  K(z) = \prod_{i=1}^{n}\frac{1}{1-x_{i}z} \prod_{j=1}^{k}\frac{1}{1+y_{j}/z}. 
\end{equation}
Indeed, we first compute
\begin{align}
\bra{0} \Gamma_{+}(X) \psi(z) \psi^{\dagger}(w) \Gamma'_{-}(Y) \ket{0}
 & = \bra{0} \Gamma_{+}(X) \psi(z) \psi^{\dagger}(w) \Gamma_-^{-1}(-Y) \ket{0} \nonumber
\\ & = H(X; -Y)^{-1} \frac{H(X; z) H(-Y; z^{-1})}{H(X; w) H(-Y; w^{-1})} \bra{0} \psi(z) \psi^{\dagger}(w) \ket{0} \nonumber
\\ & = H(X; -Y)^{-1} \frac{H(X; z) H(-Y; z^{-1})}{H(X; w) H(-Y; w^{-1})} \sum_{\ell=0}^{\infty} z^{-1/2-\ell} w^{1/2+\ell} \nonumber
\\ & = H(X; -Y)^{-1} \frac{H(X; z) H(-Y; z^{-1})}{H(X; w) H(-Y; w^{-1})} \frac{\sqrt{zw}}{z - w}.
\end{align}
To obtain the kernel, we use the above computation with the residue theorem:
\begin{align}
\mcK(m,m') & = E(X; Y)^{-1} \bra{0} \Gamma_{+}(t) \psi_{m} \psi^{\dagger}_{m'} \Gamma'_{-}(t') \ket{0} \nonumber
\\ & = E(X; Y)^{-1} \oint \oint_{\abs{w} < \abs{z}} \frac{\dz}{2\pi \imi z} \frac{\dw}{2\pi \imi w} z^{-m} w^{m'} \bra{0} \Gamma_{+}(t) \psi(z) \psi^{\dagger}(w) \Gamma'_{-}(t') \ket{0},
\end{align}
where we note that $H(X; -Y)^{-1} = E(X; Y)$ is our normalization factor.
We thus complete the proof of Theorem \ref{thm:determinantal-ensemble-corr-kernel}.

\section{Asymptotics}

\begin{figure}
\[
\begin{tikzpicture}[scale=1.2]
\draw[->,thin,gray] (-2,0) -- (2,0);
\draw[->,thin,gray] (0,-2) -- (0,2);
\draw[red,thick] (0,0) circle (1.4);
\draw[blue,thick,xscale=1.4] (.3,0) circle (1);
\draw[UQpurple,dashed] (38.5:1.4) arc (38.5:-38.5:1.4);
\fill[black] (38.5:1.4) circle (.07) node[anchor=south west] {$z_1$};
\fill[black] (-38.5:1.4) circle (.07) node[anchor=north west] {$z_2$};
\end{tikzpicture}
\qquad
\begin{tikzpicture}[scale=1.2]
\draw[->,thin,gray] (-2,0) -- (2,0);
\draw[->,thin,gray] (0,-2) -- (0,2);
\draw[blue,thick] (0,0) circle (1.4);
\draw[red,thick,rounded corners] (1.4,0) -- (15:1.8) arc (15:345:1.8) -- cycle;
\fill[black] (1.4,0) circle (.1) node[anchor=south east] {$\zcrit$};
\end{tikzpicture}
\qquad
\begin{tikzpicture}[scale=1.2]
\clip (-2.1,-2.1) rectangle (2.1,2.1);
\draw[->,thin,gray] (-2,0) -- (2,0);
\draw[->,thin,gray] (0,-2) -- (0,2);
\draw[blue,thick] (0,0) circle (.3);
\draw[red,thick] (6,0) ++ (160:6) arc (160:200:6);
\fill[black] (87:.3) circle (.07) node[anchor=south west] {$z_1$};
\fill[black] (-87:.3) circle (.07) node[anchor=north west] {$z_2$};
\fill[UQpurple] (-.1,-.1) rectangle (.1,.1);
\draw[HUgreen,line width=2pt] (1.6,.1) -- (1.6,-.1) node[anchor=north] {$x_n^{-1}$};
\draw[HUgreen,line width=2pt] (-1.6,.1) -- (-1.6,-.1) node[anchor=north] {$-y_1$};
\end{tikzpicture}
\]
\caption{Contours for the sine kernel (left), Airy kernel (center), and discrete Hermite kernel with $s = 0$ (right).}
\label{fig:contour_cartoons}
\end{figure}
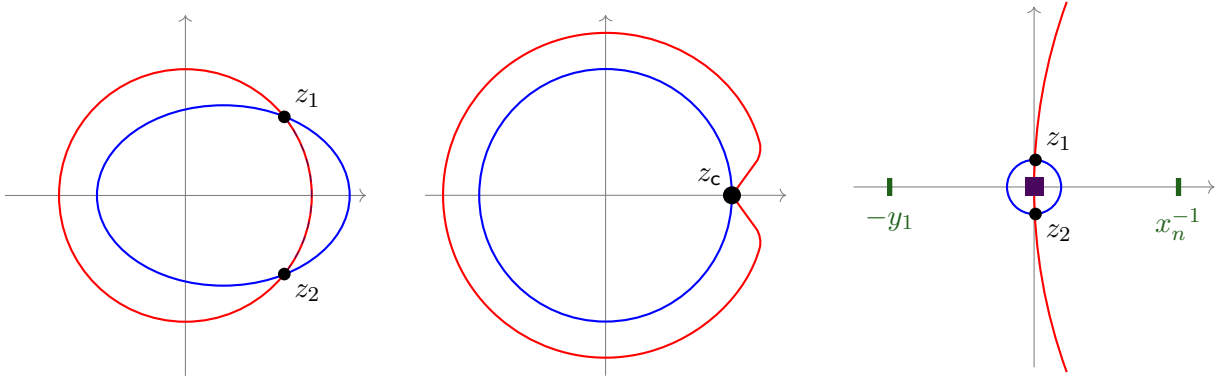

\subsection{Bulk and frozen regions}
\label{sec:bulk-asymptotics}
In this subsection we prove Theorem~\ref{thm:correlation-kernel-bulk}, Theorem~\ref{thm:limit-shape} and Corollary~\ref{cor:uniform-convergence}.

Assume that the parameters $x_{i},y_{j}$ are given by $x_{i}=f(i/n)$,
$y_{j}=g(j/k)$, with $f,g \colon [0, 1] \to \RR_{\geq 0}$ being
piecewise $\mcC^1$ nonnegative functions $f(s) > 0, g(s) \geq 0$.
Consider the limit $n,k\to\infty$ such that $\lim\frac{k}{n}=c$. Take $m=\lfloor nt\rfloor+l$, $m'=\lfloor nt\rfloor+l'$ and denote by $S_{n}(z,t)$ the following expression
\begin{equation}
  \label{eq:action-for-finite-n}
  S_{n}(z,t)=-\frac{1}{n}\sum_{i=1}^{n}\ln(1-f(i/n) z)-\frac{1}{n}\sum_{j=1}^{k}\ln(z+g(j/k))+\frac{k-\lfloor nt\rfloor}{n}\ln z.
\end{equation}
The correlation kernel is then written as
\begin{equation}
\mcK(m,m') =\oint\oint_{\abs{w}<\abs{z}}\frac{dz}{2\pi \imi z} \frac{dw}{2\pi \imi w} e^{n(S_{n}(z,t)-S_{n}(w,t))} \frac{\sqrt{zw}}{z-w}z^{-l}w^{l'}.
\end{equation}

Define the action as the limit of $S_{n}(z,t)$ as $n\to\infty$, which is given by the (Riemann) integral
\begin{equation}
  \label{eq:action}
  S(z)=-\int_{0}^{1} \ds
  \ln(1-f(s)z)-c\int_{0}^{1} \ds \ln(1+g(s)/z) - t \ln z.
\end{equation}
We choose our branch cuts of the logarithms such that the derivative of $S(z)$ has branch cuts $[-\max g,-\min g]$ and $[(\max f)^{-1}, (\min f)^{-1}]$. 
Then $\frac{1}{n}\ln K(z)-t\ln z=S(z)+\mathcal{O}(\frac{1}{n})$ and the limit shape of
Young diagrams with respect to the distribution~\eqref{eq:GL-probability-measure} is determined by the limit density
\begin{equation}
  \label{eq:rho-from-correlation-kernel}
  \rho(t)=\lim_{n,k\to\infty}\mcK(nt,nt) =
  \lim_{n,k\to\infty}\oint\oint_{\abs{w}<\abs{z}}\frac{dz}{2\pi \imi z}
  \frac{dw}{2\pi \imi w} e^{n(S(z)-S(w))} \frac{\sqrt{zw}}{z-w}\left(1+o(1)\right).
\end{equation}
Analysis as in~\cite{okounkov2007random} demonstrates that only critical points of the action (the solutions of $\partial S(z)=0$) contribute to the integral \eqref{eq:rho-from-correlation-kernel}. Let us demonstrate that the equation
\begin{equation}
  \label{eq:zdz-S-eq-zero}
  z\partial_{z}S(z) = \int_{0}^{1} \ds \frac{f(s)z}{1 - f(s)z} +
  c \int_{0}^{1} \ds \frac{g(s)}{z+g(s)}-t = 0, 
\end{equation}
has either two complex-conjugate roots  $z_{1},z_{2}$ with $z_{2}=\bar z_{1}$, in which case the contours should be deformed as in Fig.~\ref{fig:contour_cartoons}(left), or only real roots. 

Consider the approximation to Equation~\eqref{eq:zdz-S-eq-zero} for finite $n$, which from~\eqref{eq:action-for-finite-n} reads
\begin{equation}
  \label{eq:zdz-S-for-finite-n}
  z\partial_{z}S_{n}(z,t)=\frac{1}{n}\sum_{i=1}^{n}\frac{1}{1-f(i/n) z}+\frac{1}{n}\sum_{j=1}^{k}\frac{g(j/k)}{z+g(j/k)}-\frac{n+\lfloor nt\rfloor}{n}=0.
\end{equation}
Equation~\eqref{eq:zdz-S-for-finite-n} has at most $n+k$ real roots.
We want to consider the (real-valued) function $T_{n}(z)=\frac{1}{n}\sum_{i=1}^{n}\frac{1}{1-f(i/n) z}+\frac{1}{n}\sum_{j=1}^{k}\frac{g(j/k)}{z+g(j/k)}$. 
The function $T_{n}(z)$ has poles at $\{f(i/n)^{-1}\}_{i=1}^{n}$ and $\{-g(j/k)\}_{j=1}^{k}$.
If $z\to f(i/n)^{-1}$ from the right (resp.\ left), then $T(z)\to-\infty$ (resp.\ $T_{n}(z)\to+\infty$).
So the horizontal line $y=\lfloor nt\rfloor/n+1$ intersects the graph of $T_{n}(z)$ at least once on each interval $(f((i+1)/n)^{-1},f(i/n)^{-1})$, and similarly for the intervals $(-g((j+1)/k),-g(j/k))$.
Therefore Equation~\eqref{eq:zdz-S-for-finite-n} has at least $n+k-2$ real roots and at most one pair of complex conjugate roots. 

We now show the same pattern holds in the limit when $n,k\to\infty$.
Fix some $t \in [-1, c]$.
Set $\widehat{T}_{n}(z) = t+1 - T_{n}(z)$.
If $\widehat{T}_{n}(z)=0$ has a pair of complex conjugate roots $z_0^{(n)},\bar{z}_0^{(n)}$, we can represent it in the form
\begin{equation}
    \widehat{T}_{n}(z) = \dfrac{\displaystyle (t+1)(-1)^{n}f(1)\prod_{i=1}^{n-1}f(i/n) (z-z_i) \prod_{j=1}^{k-1} (z-w_j)(z-z_0^{(n)})(z-\bar{z}_0^{(n)})}{\displaystyle \prod_{i=1}^{n}(1-f(i/n) z)\prod_{j=1}^{k}(z+g(j/k))} ,
\end{equation}
where $z_i$ lies between $1/f(i/n)$ and $1/f(i+1/n)$, $w_j$ lies between $-g(j/k)$ and $-g((j+1)/k)$ and $\Im z_0^{(n)} > 0$. Fix $\epsilon>0$ and set $\Im z > \epsilon$, $\Im z^{-1}<-\epsilon$ . First consider the contributions with the function $g$, which is bounded, since $g$ is piecewise $\mcC^{1}$. Let us estimate the absolute value of the corresponding contribution when $k$ is large enough:
\begin{equation}
    \ln \left\lvert\prod_{j=1}^{k-1} \frac{z-w_j}{z+g(j/k)} \right\rvert =
    \sum_{j=1}^{k-1} \ln \left\lvert 1 - \frac{w_j + g(j/k)}{z+g(j/k)} \right\rvert \ge \sum_{j=1}^{k-1} \ln \left\lvert 1 - \left\lvert\frac{g((j+1)/k) - g(j/k)}{z+g(j/k)}\right\rvert \right\rvert.
  \end{equation}
Now use that $\ln(1-\abs{x})> -(1+\epsilon)\abs{x}$ for small $\abs{x}<\delta$, take $k$ such that $\left\lvert \frac{g((j+1)/k) - g(j/k)}{z+g(j/k)}\right\rvert < \delta$ and use that $\abs{z+g(j/k)}>\epsilon$ to obtain the estimate
\begin{equation}
  \ln \left|\prod_{j=1}^{k-1} \frac{z-w_j}{z+g(j/k)} \right|\ge -(1+\epsilon) \sum_{j=1}^{k-1} \epsilon^{-1}|g(j+1/k)-g(j/k)| \ge  -(\epsilon^{-1}+1) TV(g), 
\end{equation}
where $TV(g) < \infty$ is the total variation of $g$ on $[0, 1]$. 

To have a similar estimate for $f$, note that $z_{i}^{-1}\in \bigl( f(i/n),f((i+1)/n) \bigr)$, and write the contribution corresponding to $f$ as
$\left(f(1/n)\prod_{i=1}^{n-1}f((i+1)/n) z_{i}\right)\cdot\frac{\prod_{i=1}^{n-1}(z^{-1}-z_{i}^{-1})}{\prod_{i=1}^{n}(z^{-1}-f(i/n))}$.
The absolute value of first term is estimated by $f(0)$, and the second term is estimated in the same way as the contribution of the function $g$. 

Also $\widehat{T}_n(z)$ tends to zero when $n\to \infty$ for any root $z$ of \eqref{eq:zdz-S-eq-zero}. Therefore if there is such a root $z_0^{\infty}$ with $\Im z_0^{\infty} > 0$, then for any $n$ sufficiently large we should have precisely one pair of complex conjugate roots for $\widehat{T}_n(z)=0$ and $\lim_{n\to\infty} z_0^{(n)} = z_0^{\infty}$. 
Therefore $\widehat{T}_n(z)$ is uniformly separated from zero if $z$ is not sufficiently close either to the pair $\{z_0^{(n)}, \bar{z}_0^{(n)}\}$ or to the real line.

Note that this argument uses $f(s)>0$ for all $s$, but we can consider functions that go to zero as $s$ goes to zero; that is, $f(s) \xrightarrow[s\to 0]{} 0$.
Then more careful analysis is needed.
In Section~\ref{sec:monomial-functions}, we consider the power functions $f(s) = s^{m}, m\in\ZZ_{\geq0}$ and Equation~\eqref{eq:zdz-S-eq-zero} becomes transcendental, but numerically we can see that we still have at most one complex conjugate pair of roots. 

Take a support interval $[t_{-},t_{+}]\subset [-1,c]$ and denote by $z_{1}$ and $z_{2} = \bar z_{1}$ two complex conjugate roots of Equation~\eqref{eq:zdz-S-eq-zero}, then the correlation kernel converges to the discrete sine kernel $\lim_{n,k\to\infty}\mcK(nt+l,nt+l')=\frac{\sin\pi\rho(t)(l-l')}{\pi(l-l')}$ with 
\begin{equation}
  \label{eq:rho-arg-z}
  \rho(t) = \frac{1}{\pi}\arg z_{1} = \frac{1}{\pi} \arccos\frac{\Re z_{1}}{\sqrt{(\Re z_{1})^{2}+(\Im z_{1})^{2}}}
\end{equation}
for $z_{1} = \Re z_{1} + \imi\Im z_{1}$.
The support of the density $\rho$ contains the union of the support intervals.
Any support interval  $[t_{-},t_{+}]$ is determined by nonzero real roots $z_{\pm}\neq 0$ of the equation
\begin{equation}
  \label{eq:zdz-2-S-eq-zero}
  \begin{split}
    (z\partial_{z})^{2}S(z) & = \int_{0}^{1} \ds \frac{f(s)z}{1-f(s)z} +
    \int_{0}^{1} \ds \frac{(f(s)z)^{2}}{(1-f(s)z)^{2}}
    - c \int_{0}^{1} \ds \frac{g(s) z}{(z+g(s))^{2}}
    \\ & = \int_{0}^{1} \ds \frac{f(s)z}{(1-f(s)z)^2}
    - c \int_{0}^{1} \ds \frac{g(s) z}{(z+g(s))^{2}} = 0.
  \end{split}
\end{equation}
We substitute $z_{\pm}$ into Equation~\eqref{eq:zdz-S-eq-zero} and obtain $t_{\pm}$. If $t$ tends to $t_{-}$ or $t_{+}$, then the complex conjugate roots $z_{1},z_{2}$ tend to $z_{-}$ or $z_{+}$. 

When $t$ is outside of the union of the support intervals, Equation~\eqref{eq:zdz-S-eq-zero} has only real roots and in order for the formula in Theorem~\ref{thm:limit-shape} to hold we need to extend $\rho(t)$ to either be $0$ or $1$. Near a support interval $[t_{-},t_{+}]$ we choose this extension depending on the values of $z_-$ if $t < t_{-}$ and $z_+$ if $t > t_{+}$.
Specifically, we set $\rho(t) = 1$ if $z_- < 0$ (resp.\ $z_+ < 0$) and $\rho(t) = 0$ if $z_- > 0$ (resp.\ $z_+ > 0$).

To establish the pointwise convergence in probability for $F_{n}$, we denote by $N(m)=\#\{\ell \in \ket{\lambda} \mid \ell>m\}$, the number of particles lying to the right of the position $m$.
Then the expectation and variance of $N(m)$ is expressed in terms of the correlation kernel as
\begin{equation}
\EV[N(m)] = \tr_{(m,\infty)}\mcK,
\qquad\qquad
\Var[N(m)] = \tr_{(m,\infty)}(\mcK-\mcK^{2}).
\end{equation}

Since the kernel is real and symmetric, the trace of its square is non-negative $\tr \mcK^{2}\geq 0$ and $\Var[N(m)]\leq \EV[N(m)]$.
In the asymptotic regime $a_{i}=nu_{i}$, setting $\widetilde{N}(nu)=\frac{1}{n}N(nu)$ we obtain
\begin{equation}
  \Var[\widetilde{N}(nu)] = \frac{1}{n^{2}} \Var[N(nu)] \leq \frac{1}{n} \EV[\widetilde{N}(nu)].
\end{equation}
For the upper boundary of the rescaled rotated diagram, we have $F_{n}(u) = u+\frac{1}{n}\cdot \#\{\ell \in \lambda \mid \ell > nu\}$ for $u\in\frac{1}{n}(\ZZ+\frac{1}{2})$.
Since we have the limit shape
\begin{equation}
\Omega(u) = \lim_{n\to\infty} \EV[F_{n}(u)] = u + \lim_{n\to\infty}\EV[\widetilde{N}(nu)],
\end{equation}
we obtain $\Var[F_{n}(u)]\leq \frac{1}{n}\EV[\widetilde{N}(nu)] \longrightarrow 0$ as $n\to\infty$.
Hence we have pointwise convergence in probability for $F_{n}(u)$ and conclude the proof of the Theorem~\ref{thm:limit-shape}.

To prove the uniform convergence of Corollary~\ref{cor:uniform-convergence}, we note that for the bounded interval $I=[-1,c]\subset\RR$, we denote $I_{\varepsilon}=I\cap \varepsilon\ZZ$ and by $1$-Lipschitz property of $F_{n}$ we have for each $\varepsilon>0$
\begin{equation}
  \PP\left(\sup_{u\in I} \abs{F_{n}(u)-\Omega(u)} >\varepsilon\right) \leq \PP\left(\sup_{u\in I_{\varepsilon}} \abs{F_{n}(u)-\Omega(u)} > \frac{\varepsilon}{2}\right).
\end{equation}
On the right hand side the supremum is computed over the finite set so the convergence to zero at each point $u\in I_{\varepsilon}$ implies the convergence of the supremum to zero.
This leads to the convergence of the supremum norm over $I$ to zero in probability.

\subsection{Edge boundary}
\label{sec:boundary-asymptotics}

To study the boundary asymptotics and prove Theorem \ref{thm:boundary-asymptotics}, we need to consider a critical value $t_{+}$ (or $t_{-}$) where the roots $z_{1},z_{2}$ coincide. The integration contours then look like in Fig.~\ref{fig:contour_cartoons}(center).
Denote the corresponding value of $z$ by $\zcrit$.
Then near the double critical point $\zcrit$ we have
\begin{equation}
S(z)=S(\zcrit)+\frac{1}{6}S'''(\zcrit)(z - \zcrit)^{3}+\mathcal{O}\left((z - \zcrit)^{4}\right).
\end{equation}
We assume that $S'''(\zcrit)\neq 0$, $\zcrit\neq 0$ and denote by $\sigma$ a normalization constant
\begin{equation}
  \label{eq:normalization-for-airy}
  \sigma=\left(\frac{2}{S'''(\zcrit)}\right)^{\frac{1}{3}}\frac{1}{\zcrit}.
\end{equation}
We change the variables $z,w$ to $\zeta,\nu$ such that $z = \zcrit e^{\sigma \zeta n^{-\frac{1}{3}}} \approx \zcrit (1+\sigma \zeta n^{-\frac{1}{3}}+\cdots)$, $w=\zcrit e^{\sigma \nu n^{-\frac{1}{3}}} \approx \zcrit (1+\sigma \nu n^{-\frac{1}{3}}+\cdots)$ as $n\to\infty$ and consider the asymptotic regime $m\approx t_{+}n+\xi n^{\frac{1}{3}} \sigma^{-1}$, $m'\approx t_{+}n+\eta n^{\frac{1}{3}} \sigma^{-1}$.

The correlation kernel is then expressed as
\begin{equation}
  \mcK(m,m')\approx\sigma n^{-\frac{1}{3}} \zcrit^{m'-m} \iint \frac{\dzeta \dnu}{(2\pi i)^{2}} \frac{\exp\left(\frac{\zeta^{3}}{3}-\zeta\xi\right)}{\exp\left(\frac{\nu^{3}}{3}-\nu\eta\right)}\frac{1}{\zeta-\nu},
\end{equation}
since
\begin{equation}
n\bigl( S(z)-S(w) \bigr) \approx \frac{1}{6}\sigma^{3} \zcrit^{3} S'''(\zcrit)(\zeta^{3}-\nu^{3}) = \frac{1}{6}\sigma^{3}\left(\left.(z\partial_{z})^{3}S(z)\right|_{z=\zcrit}\right)(\zeta^{3}-\nu^{3}) = \frac{1}{3}(\zeta^{3}-\nu^{3}).
\end{equation}

Therefore the correlation kernel \eqref{eq:correlation-kernel-integral-representation}  after multiplication by $n^{\frac{1}{3}}\sigma^{-1}\zcrit^{m-m'}$ converges to the Airy kernel 
\begin{equation}
  \label{eq:airy-kernel}
  \lim_{n\to\infty} n^{1/3}\sigma^{-1}\zcrit^{m-m'}\mathcal{K}(m,m')=\mathcal{K}_{\mathrm{Airy}}(\xi,\eta) = \iint \frac{\dzeta \dnu}{(2\pi i)^{2}} \frac{\exp\left(\frac{\zeta^{3}}{3}-\zeta\xi\right)}{\exp\left(\frac{\nu^{3}}{3}-\nu\eta\right)}\frac{1}{\zeta-\nu}.
\end{equation}

We are interested in the value of $(z\partial_{z})^{3}S(z)\bigr\rvert_{z=\zcrit}$, which is given by the integral
\begin{equation}
  \label{eq:zdz-3-S}
  (z\partial_{z})^{3}S(z)\bigg\vert_{z=\zcrit}=\int_{0}^{1}\ds \left(\frac{2 f^{2}(s) \zcrit^{2}}{(1-f(s) \zcrit)^{3}}+\frac{2c g(s) \zcrit^{2}}{(\zcrit+g(s))^{3}}\right). 
\end{equation}
If the integral in~\eqref{eq:zdz-3-S} is zero, the normalization constant $\sigma$ becomes infinite.
For $\sigma < \infty$, there are two cases:
\begin{enumerate}
\item we have $\lambda_{1}<k$ and then the limit shape $\Omega$ is convex near $t_{+}$;
\item we have multiple fully filled rows of length $k$ and $\Omega$ is concave near $t_{+}$.
\end{enumerate}
We conclude that in the first case the normalized fluctuations of the first row of the random Young diagram $\frac{\lambda_{1}-t_{+}n}{\sigma^{-1}n^{\frac{1}{3}}}$ are described by the Tracy--Widom distribution with $\beta=2$. 
Similarly, in the second case when $\Omega$ is concave near $t_{+}$ the same distribution describes the normalized fluctuations of the first column of the diagram $\overline{\lambda}$, the complement to $\lambda$ inside $n\times k$ rectangle.
Alternatively, we are looking at $n - \#\{i \mid \lambda_{i}=k\}$. 
The computation around the other double critical point $t_{-}$, which is measuring the fluctuations of the first column of $\lambda$ (or first row of $\overline \lambda$), is similar.

Informally, if $\sigma = \infty$, then the result corresponds to an infinitely-thin Tracy--Widom distribution; in practice, it means that the fluctuations are described another distribution.
We have two possibilities, either $S'''(\zcrit)=0$ or $\zcrit=0$.
The former leads to the Pearcey kernel~\cite{brezin1998universal,okounkov2007random,tracy2006pearcey} or to higher order Airy kernels \cite{betea2023multicritical,kimura2021universal} if higher derivatives are also zero. 
The latter corresponds to the discrete Hermite kernel defined in~\cite{borodin2007asymptotics}.
We discuss both cases in the sequel.


\subsection{Two touching support intervals}
\label{sec:pearcey-asymptotics}

To establish Theorem \ref{thm:Pearcey_kernel}, assume that we have two support intervals $(t_-, t_d)$ and $(t_d, t_+)$ with $\zcrit$ corresponding to $t_{d}$.
Then we have $\partial_{z}^{3}S(z,t_{d})|_{z=\zcrit}=0$ as otherwise on the left (resp.\ right) of $t_{d}$, we have the convergence of the kernel to $\mathcal{K}_{\mathrm{Airy}}(\xi,\eta)$, (resp.\ $\mathcal{K}_{\mathrm{Airy}}(-\xi,-\eta)$), but the Airy kernel is not symmetric.

Next, we proceed similarly to the Airy case.
We consider the asymptotic regime $m \approx t_d n+\xi n^{\frac{1}{4}}\sigma^{-1}$, $m' \approx t_d n+\eta n^{\frac{1}{4}}\sigma^{-1}$ for some constant $\sigma$ as $n\to \infty$.
As previously discussed, the first three derivatives of the action with respect to $z$ at $t=t_{d}$ and $z=\zcrit$ are zero.
Hence, for nonzero fourth derivative $\partial_{z}^{4}S(z,t_{d})|_{z=\zcrit}\neq 0$ we set
\begin{equation}
\sigma=\left(\frac{6}{\partial_{z}^{4}S(z,t_{d})|_{z=\zcrit}}\right)^{\frac{1}{4}}\frac{1}{\zcrit}
\end{equation}
and change the variables $z=\zcrit e^{\sigma\zeta n^{-\frac{1}{4}}}\approx \zcrit(1+\sigma \zeta n^{-\frac{1}{4}}+\dots)$, $w=\zcrit e^{\sigma\nu n^{-\frac{1}{4}}}\approx \zcrit (1+\sigma \nu n^{-\frac{1}{4}}+\dots)$.
Then in our asymptotic regime the normalized correlation kernel converges to the Pearcey kernel~\eqref{eq:pearcey-kernel}
since
\begin{equation}
n(S(z)-S(w))\approx\frac{1}{24}\zcrit^{4}\sigma^{4}\left(\partial_{z}^{4}S(z,t_{d})|_{z=\zcrit}\right)(\zeta^{4}-\nu^{4})=\frac{1}{4}(\zeta^{4}-\nu^{4}).
\end{equation}


\subsection{Near the corner}
\label{sec:asymptotics-near-the-corner}

To obtain the asymptotic behavior near the corner of the rectangle (Theorem~\ref{thm:near-corner-asymptotics}), we again use the integral representation of the correlation kernel~\eqref{eq:correlation-kernel}.
In this case, there is an obstruction to our asymptotic analysis of the generic case as the double critical point $z_+$ of the action under consideration is $z_+ = 0$.
As a result, the constant $\sigma = \infty$ and the integral in~\eqref{eq:zdz-3-S} is equal to zero.
The behavior near the corner is described by the discrete Hermite kernel introduced in~\cite{borodin2007asymptotics}, as demonstrated for the constant specialization in~\cite{borodin2017asep}.
Here we derive this result from the integral representation for the correlation kernel~\eqref{eq:correlation-kernel-integral-representation} and generalize the derivation to arbitrary specializations.

Consider the asymptotic regime $n\to\infty$ with $k = cn+\frac{\widetilde{s}}{\tau}\sqrt{n} + o(1)$, where $\widetilde{s}$ is a parameter and $\tau$ is a normalization constant that we will choose later.
Take $m = \lfloor cn+\frac{\widetilde{s}}{\tau}\sqrt{n}\rfloor-l+\frac{1}{2}$ and $m' = \lfloor cn+\frac{\widetilde{s}}{\tau}\sqrt{n}\rfloor-l'+\frac{1}{2}$ so that $l,l'\in\ZZ$.
We assume that the condition given by Equation \eqref{corner-condition} is satisfied.
Alternatively, these conditions can be written as 
\begin{equation}
\lim_{n,k\to\infty}\frac{k}{n}=c,
\qquad\qquad
\sum_{i=0}^{n-1}f(i/n)=c\sum_{j=0}^{k-1}g(j/k)^{-1}+\frac{\widetilde{s}}{\tau}\sqrt{n}.
\end{equation}
Under these assumptions, we have
\begin{equation}
  \ln K(z)\approx n\left[-\int_{0}^{1}\dt\ln(1-f(t)z)-c\int_{0}^{1}\dt\ln(1+g(t)/z)+\frac{1}{\sqrt{n}}\frac{\widetilde{s}}{\tau}\ln z -\frac{1}{\sqrt{n}}\frac{\widetilde{s}}{\tau}\int_{0}^{1}\dt\ln(z+g(t))\right],
\end{equation}
where $K(z)$ is the function from~\eqref{eq:Fz-expression_ps}.
Setting $t_{+}=c$ in~\eqref{eq:action} transforms the action into $S(z)= -\int_{0}^{1}\dt\ln(1-f(t)z)-c\int_{0}^{1}\dt\ln(z+g(t))$.
Taking into account the cancellation of the terms $\frac{\widetilde{s}}{\tau} \sqrt{n} \ln z$ in $\ln K(z)$ and in $z^{-m}$, we can write the correlation kernel as
\begin{gather}
  \mcK(l,l') = \oint\oint_{\abs{w}<\abs{z}}\frac{dz}{2\pi \imi z}
    \frac{dw}{2\pi \imi w} e^{\phi(z,w)} \frac{z^{l}w^{-l'+1}}{z-w},
    \\
     \phi(z,w) := n(S(z)-S(w))-\sqrt{n}\frac{\widetilde{s}}{\tau}\left[\int_{0}^{1}\dt\ln(z+g(t))-\int_{0}^{1}\dt\ln(w+g(t))\right].
\end{gather}
We are interested in the vicinity of the critical point $z=0$ as demonstrated in Fig.~\ref{fig:contour_cartoons}(right).
Thus, we consider the change of variables $z=\frac{\tau}{\sqrt{n}}\zeta$, $w=\frac{\tau}{\sqrt{n}}\nu$.
For finite $z$ and $w$, we then can approximate the logarithms under the integral in the exponent as $\ln(z+g(t))\approx\ln(g(t))+\frac{\tau}{\sqrt{n} g(t)}\zeta+\mathcal{O}\left(\frac{1}{n}\right)$.
The first derivative of the action is $S'(0)=0$ due to condition~\eqref{corner-condition}; therefore we can use the approximation $S(z) \approx S(0)+\frac{S''(0)}{2}z^{2} = S(0)+\frac{S''(0)\tau^{2}}{2n}\zeta^{2}$.
The constant contributions are cancelled.
We now specify $\tau=\frac{1}{\sqrt{S''(0)}}$ and note that
\begin{equation}
S''(0)=\int_{0}^{1} \dt \, f^{2}(t) + c\int_{0}^{1} \dt \, g^{-2}(t) = \Abs{f}_{2}^{2}+c \Abs{g^{-1}}_{2}^{2}.
\end{equation}
Next, define $s := \widetilde{s}\int_{0}^{1}\frac{\dt}{g(t)}$, and hence the correlation kernel takes the form
\begin{equation}
  \label{eq:near_corner_kernel}
 \mcK_s(l,l') := \left(\frac{\tau}{\sqrt{n}}\right)^{l-l'}\oint\oint_{\abs{\nu}<\abs{\zeta}}\frac{\dzeta}{2\pi \imi }
    \frac{\dnu}{2\pi \imi} e^{\frac{\zeta^{2}}{2}-\zeta s+\nu s-\frac{\nu^{2}}{2}} \frac{\zeta^{l-1} \nu^{-l'}}{\zeta-\nu}.
\end{equation}
For $l\neq l'$, the integrals over $\zeta$ and $\nu$ can be decoupled by integrating by parts and using
\begin{subequations}
\begin{align}
\left( \zeta\frac{\partial}{\partial \zeta}+\nu\frac{\partial}{\partial \nu}+1\right)\zeta^{l-1}\nu^{-l'} & = (l-l')\zeta^{l-1}\nu^{-l'},
\\
\left(\zeta\frac{\partial}{\partial \zeta}+\nu\frac{\partial}{\partial \nu}+1\right)\frac{e^{\frac{\zeta^{2}}{2}-\zeta s+\nu s-\frac{\nu^{2}}{2}}}{\zeta-\nu} & = (\zeta+\nu-s)e^{\frac{\zeta^{2}}{2}-\zeta s+\nu s-\frac{\nu^{2}}{2}}.
\end{align}
\end{subequations}
Hence, the kernel can be written as
\begin{equation}
\mcK_s(l,l')=\left(\frac{\tau}{\sqrt{n}}\right)^{l-l'}\oint\oint_{\abs{\nu}<\abs{\zeta}}\frac{\dzeta}{2\pi \imi }
    \frac{\dnu}{2\pi \imi} e^{\frac{\zeta^{2}}{2}-\zeta s+\nu s-\frac{\nu^{2}}{2}} \frac{\zeta^{l-1} \nu^{-l'}(\zeta+\nu-s)}{l-l'},
\end{equation}
where the integration contour over $\nu$ can be considered as an arbitrary small counterclockwise circle around $0$, while the integration contour over $\zeta$ includes the $\nu$-contour and can be extended along the imaginary axis from $-\imi \infty$ to $\imi \infty$.
Then integral over $\nu$ is the standard contour integral representation of Hermite polynomial
\begin{equation}
\He_{l'-1}(s)=\frac{(l'-1)!}{2\pi\imi}\oint \dnu \; e^{\nu s -\frac{\nu^{2}}{2}} \nu^{-l'}.
\end{equation}
The integral over $\zeta$ can be brought to a real integral representation for the Hermite polynomials
\begin{equation}
  \label{eq:he_l-integral}
\He_{l}(s) = \frac{1}{\sqrt{2\pi}} \int_{-\infty}^{\infty} \dy \, (s+\imi y)^{l}e^{-\frac{y^{2}}{2}} = \frac{1}{\sqrt{2\pi}}\int_{-\infty}^{\infty} \mathrm{d}\widetilde{\mathrm{y}} \, (\imi \widetilde{y})^{l}e ^{-\frac{\widetilde{y}^{2}}{2}-\imi \widetilde{y}s+\frac{s^{2}}{2}}
\end{equation}
by the change of variable $\zeta=\imi \widetilde{y}$ and multiplication by $e^{\frac{s^{2}}{2}}$:
\begin{equation}
  \int_{-\imi\infty}^{\imi\infty} \dzeta \; e^{\frac{\zeta^{2}}{2}-\zeta s} \zeta^{l} = \sqrt{2\pi}e^{-\frac{s^{2}}{2}} \He_{l}(s).
\end{equation}
Using the recurrence relation for the Hermite polynomials $s \He_{l}(s)=\He_{l+1}(s)+l \He_{l-1}(s)$, we can write the kernel as 
\begin{equation}
  \label{eq:discrete-hermite-kernel}
  \mcK_s(l,l')=\left(\frac{\tau}{\sqrt{n}}\right)^{l-l'} \frac{1}{\sqrt{2\pi}} e^{-\frac{s^{2}}{2}} \frac{1}{(l'-1)!}\frac{\He_{l}(s) \He_{l'-1}(s) - \He_{l-1}(s) \He_{l'}(s)}{l-l'}.
\end{equation}
This kernel coincides with the discrete Hermite kernel introduced in~\cite{borodin2007asymptotics} as
\begin{equation}
  \mcK^{\He}_{s}(l,l')=\frac{1}{\sqrt{2\pi (l-1)! (l'-1)!}} e^{-\frac{s^{2}}{2}} \frac{\He_{l}(s) \He_{l'-1}(s) - \He_{l-1}(s) \He_{l'}(s)}{l-l'}
\end{equation}
up to a factor $\left(\frac{\tau}{\sqrt{n}}\right)^{l-l'}\sqrt{\frac{(l'-1)!}{(l-1)!}}$, which is cancelled in the determinant $\det [\delta_{i,j}-\mathcal{K}(i,j)]_{i,j}$ and therefore does not contribute to the correlation functions.

The case $l=l'$ can be obtained by taking the limit $l'\to l$ and using l'H\^opital's rule
\begin{equation}
  \label{eq:mck-herm-l-l}
  \mcK_s(l,l)= \frac{1}{\sqrt{2\pi}} e^{-\frac{s^{2}}{2}} \frac{1}{(l-1)!}\left(\left[\frac{d}{dl}\He_{l}(s)\right] \He_{l-1}(s) - \left[\frac{d}{dl}\He_{l-1}(s)\right] \He_{l}(s)\right). 
\end{equation}
As we have
\begin{equation}
\frac{e^{-\frac{s^{2}}{2}}}{\sqrt{2}}\frac{\He_{l}(s) \He_{l'-1}(s) - \He_{l-1}(s) \He_{l'}(s)}{l-l'}=\int_{s}^{\infty} \dt \; e^{-\frac{t^{2}}{2}}\He_{l-1}(t)\He_{l'-1}(t)
\end{equation}
by~\cite[Prop.~3.3]{borodin2007asymptotics}, taking the limit $l' \to l$ we see that the correlation kernel for $l=l'$ can be written in the form:
\begin{equation}
  \label{eq:hermite_kernel_diagonal}
  \mcK_s(l,l)= \frac{1}{\sqrt{2\pi}} \frac{1}{(l-1)!} \int_{s}^{\infty} \dt \; e^{-\frac{t^{2}}{2}} \He_{l-1}^{2}(t). 
\end{equation}
Then the probability of the first row to have length no more than $k-\Delta$ is given by the gap probability formula $\lim_{n,k\to\infty}\PP (\lambda_1 - k \leq -\Delta) = \det\bigl[\delta_{ij}-\mcK_s(i,j)\bigr]_{i,j=0}^{\Delta-1}$. 
In Fig.~\ref{fig:discrete-hermite-distribution}, we present the comparison of this discrete distribution for various values of $s$ to samplings by the dual Robinson--Schensted--Knuth algorithm. 

\begin{figure}
  \begin{center}
  \includegraphics[width=16cm]{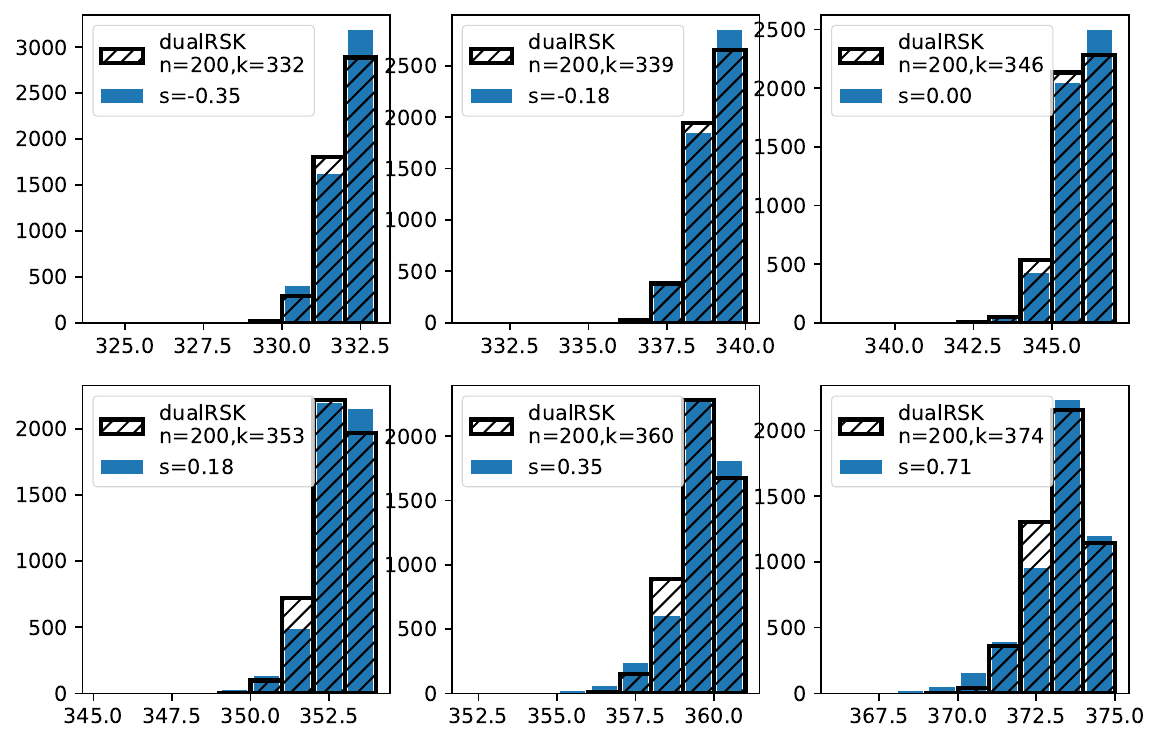}
  \end{center}
  \caption{Histograms of the distribution given by the discrete Hermite kernel for $s=-0.35, -0.18, 0, 0.18, 0.35, 0.71$ {\it (solid blue)} and distributions of first row length obtained by sampling 5000 diagrams {\it(black hatch)} for the specialization functions $f(s)=4s^{3}$, $g(s)=2+\sin(2\pi s)$ with $n=200$ and $k$ chosen to satisfy the condition \eqref{corner-condition}.}
  \label{fig:discrete-hermite-distribution}
\end{figure}

\subsection{At the corner}
\label{sec:asymptotics-at-the-corner}

The case when the limit shape ends exactly at the corner corresponds to $s = 0$.
This was considered in~\cite{GTW01} for the constant specialization $f(s) = c$ and $g(s) = 1$, where they called this regime ``critical'' as it represented a phase transition. 
Their derivation relies on the results of Borodin and Okounkov~\cite{borodin2000fredholm} and essentially differs from the previous section only by expanding $\frac{1}{z-w}$ into series and integrating by terms instead of integration by parts.
In particular, the assumption~\eqref{corner-condition} holds automatically. 
Their result~\cite[Sec.~2.2]{GTW01} in the notation of present paper is 
\begin{equation}
  \label{eq:gap_prob_det_GTW}
  \lim_{n \to \infty} \PP (\lambda_1 - n c \leq -\Delta) = \det_{0 \leq i, j \leq \Delta-1} \bigl[ \delta_{i, j} - \mcK^{\Delta}_{\rm crit}(i,j) \bigr]
\end{equation}
with the entries of the matrix given by
\begin{equation}
  \mcK^{\Delta}_{\rm crit} (i, j) = \sum\limits_{\ell = 0}^{(\Delta - j - 1)/2}
  \begin{cases}
    \frac{1}{2 \pi} \frac{1}{\ell!} \sin \frac{\pi (j-i)}{2} \Gamma(\ell + \frac{j-i}{2}) & \text{if } \ell + \frac{j-i}{2} \notin \ZZ_{\leq 0}, \\
    \frac{1}{2}  \frac{(-1)^\ell}{\ell! (\frac{i-j}{2} - \ell)!} & \text{if } \ell + \frac{j-i}{2} \in \ZZ_{\leq 0}.
  \end{cases}
\end{equation}
We wish to compare~\eqref{eq:gap_prob_det_GTW} with~\eqref{eq:discrete-hermite-kernel} and show the determinants give the same gap probability formula.
We will show this by essentially identifying the matrices, which we make precise as follows.

\begin{thm}
\label{thm:critical_kernels}
For all $\frac{j - i}{2} \notin \ZZ$ and $\Delta > 0$, we have
\begin{equation}
2^{(i-j)/2} \mcK^{\Delta}_{\rm crit}(\Delta - 1 - i, \Delta - 1 - j) = \mcK_0(i, j).
\end{equation}
\end{thm}

We remark that Theorem~\ref{thm:critical_kernels} implies the determinants are equal as the factor $2^{(i-j)/2}$ will not contribute to the determinant and we will show the diagonal entries of both matrices are all $\frac{1}{2}$.

\begin{proof}
We begin by analyzing the kernel $\mcK^{\Delta}_{\rm crit}$ in~\eqref{eq:gap_prob_det_GTW}.
If $\frac{j-i}{2}\in\ZZ_{>0}$, then we clearly get $\mcK^{\Delta}_{\rm crit}(i,j) = 0$.
For the diagonal entries $i = j$, the only term that is nonzero is the $\ell = 0$ term, and so $\mcK^{\Delta}_{\rm crit} (i, j) = 1/2$, which also equals the diagonal entries of the matrix $[\delta_{ij} - \mcK_{\rm crit}^{\Delta}(i,j)]_{i,j}$.
Next we consider when $\frac{j-i}{2} \in \ZZ_{<0}$, and for simplicity we define $A := \frac{i-j}{2}$.
By the binomial theorem, we have
\begin{equation}
  \label{eq:Kcrit_neg_int_zero}
  \mcK^{\Delta}_{\rm crit} (i, j) = \frac{1}{2} \sum_{\ell=0}^A \frac{(-1)^\ell}{\ell! (A - \ell)!} = \frac{1}{2A!} \sum_{\ell=0}^A \frac{(-1)^\ell A!}{\ell! (A - \ell)!} = \frac{(1 + (-1))^A}{2A!} = 0.
\end{equation}
Next for $B := \frac{j-i}{2} \notin \ZZ$, set $B' := B - \frac{1}{2} = \frac{j-i-1}{2}$, and by well-known properties of the Gamma function, we have
\begin{equation}
\frac{\sin(\pi B) \Gamma(\ell+B)}{\pi \ell!} = \frac{(-1)^{B'}}{\sqrt{\pi} \ell!} \begin{cases}
  \frac{(2(\ell+B)-2)!!}{2^{\ell+B'}} & \text{if } \ell + B > 0, \\
  \frac{(-2)^{-\ell-B'}}{(-2(\ell+B))!!} & \text{if } \ell + B < 0,
\end{cases}
\end{equation}
where the double factorial is defined as $n!! := \prod_{k=0}^{\lfloor n/2 \rfloor} (n - 2k)$ and by convention $(-1)!! = (0)!! = 1$.

On the other hand, we want to examine the (modified) discrete Hermite kernel~\eqref{eq:discrete-hermite-kernel}.
Note that we shift the indices to $[0, \Delta - 1]$ instead of in $[1, \Delta]$ to more closely match the above determinant formula.
We first simplify the diagonal entries~\eqref{eq:hermite_kernel_diagonal} (with replacing $i \mapsto i+1$ for the change in indexing convention), which when $s = 0$ becomes
\begin{equation}
\mcK_0(i, i) = \frac{1}{\sqrt{2\pi}} \frac{1}{i!} \cdot \frac{\inner{\He_i(t)}{\He_i(t)}_{\He}}{2} = \frac{1}{\sqrt{2\pi}} \frac{1}{i!} \cdot \frac{i! \sqrt{2\pi}}{2} = \frac{1}{2},
\end{equation}
where $\inner{\cdot}{\cdot}_{\He}$ denotes the inner product in which the Hermite polynomials are orthogonal.
Now we assume $i \neq j$, and so we want to compute
\begin{equation}
\mcK_0(i, j) = \frac{1}{\sqrt{2 \pi} j!}\frac{\He_{i+1}(0) \He_j(0) - \He_i(0) \He_{j+1}(0)}{i-j},
\end{equation}
When $i$ and $j$ have the same parity, that is $\frac{j - i}{2} \in \ZZ$, then $\mcK_0(i, j) = 0$ since $H_{\ell}(0) = 0$ when $\ell$ is odd.
Now we assume $i$ and $j$ have different parity, so $\frac{j - i}{2} \notin \ZZ$, and using $\He_{\ell}(0) = (-1)^{\ell/2}(\ell-1)!!$ for $\ell$ even, we compute
\begin{equation}
\mcK_0(i, j) = 
\frac{(-1)^{(j-i+1)/2}}{\sqrt{2 \pi} j! (i-j)} \begin{cases}
  i!! (j-1)!! & \text{if $i$ is odd}, \\
  (i-1)!! j!! & \text{otherwise}.
\end{cases}
\end{equation}

Summarizing the above, remains to show for all $i$ and $j$ such that $\frac{j-i}{2} \notin \ZZ$ that
\begin{equation}
\label{eq:claim_kernel_entry_equality}
\frac{2^{(i-j-1)/2}}{2\pi \sqrt{2}} (-1)^{(i-j-1)/2} \sum_{\ell = 0}^{\lfloor j/2 \rfloor} \frac{\Gamma(\ell + \tfrac{i-j}{2})}{\ell!}
=
\frac{(-1)^{(j-i+1)/2}}{\sqrt{2 \pi} j! (i-j)} \begin{cases}
  i!! (j-1)!! & \text{if $i$ is odd}, \\
  (i-1)!! j!! & \text{otherwise}.
\end{cases}
\end{equation}
We can multiply both sides by $(-1)^{(j-i+1)/2} \sqrt{2\pi}$ and note that $(-1)^{i-j+1} = 1$ by our parity restriction.
Next we split the problem into when $i > j$ and $i < j$.
Hence, Equation~\eqref{eq:claim_kernel_entry_equality} for $i > j$ is equivalent to
\begin{equation}
\label{eq:pos_df_identities}
\sum_{\ell = 0}^{\lfloor j/2 \rfloor} \frac{(2\ell+i-j-2)!!}{2^{\ell} \ell!}
=
\frac{1}{j! (i-j)} \begin{cases}
  i!! (j-1)!! & \text{if $i$ is odd}, \\
  (i-1)!! j!! & \text{otherwise},
\end{cases}
\end{equation}
and for $i < j$, we need to split the sum into two parts, where setting $A := \frac{j-i+1}{2}$, we will show below (Lemma~\ref{lemma:neg_df_id}) that
\begin{equation}
\label{eq:neg_df_identities}
\sum_{\ell = 0}^{A-1} \frac{(-1)^{\ell+A}}{2^{\ell} \ell! (j-i-2\ell)!!}
+
\sum_{\ell=A}^{\lfloor j/2 \rfloor} \frac{(2\ell+i-j-2)!!}{2^{\ell} \ell!}
=
\frac{1}{j! (i-j)} \begin{cases}
  i!! (j-1)!! & \text{if $i$ is odd}, \\
  (i-1)!! j!! & \text{otherwise}.
\end{cases}
\end{equation}
In both~\eqref{eq:pos_df_identities} and~\eqref{eq:neg_df_identities}, the right hand side becomes
\begin{equation}
\frac{i!! (j-1)!!}{j! \, (i-j)} = \frac{i!!}{j!! \, (i-j)} \quad (i \text{ odd}),
\qquad\qquad
\frac{(i-1)!! j!!}{j! \, (i-j)} = \frac{(i-1)!!}{(j-1)!! \, (i-j)} \quad (i \text{ even}).
\end{equation}
We show the double factorial identities below, which completes the proof.
\end{proof}

\begin{prop}
For $i > j \geq 0$ with $i + j \equiv 1 \pmod{2}$, we have
\begin{subequations}
\label{eq:pos_df_identities_sub_rhs}
\begin{align}
\sum_{\ell = 0}^{j/2} \frac{(2\ell+i-j-2)!! \, j!! \, (i-j)}{(2\ell)!!} & = i!!, && \text{if $i$ odd and $j$ even}, \label{eq:odd_i_pos_df} \\
\sum_{\ell = 0}^{(j-1)/2} \frac{(2\ell+i-j-2)!! \, (j-1)!! \, (i-j)}{(2\ell)!!} & = (i-1)!!, && \text{if $i$ even and $j$ odd}. \label{eq:even_i_pos_df}
\end{align}
\end{subequations}
\end{prop}

\begin{proof}
We prove Equation~\eqref{eq:odd_i_pos_df} by using simultaneous induction on $i \mapsto i + 2$ and $j \mapsto j + 2$ by
\begin{align}
\sum_{\ell = 0}^{j/2} \frac{(2\ell+i-j-2)!! \, j!! \, (i-j)}{(2\ell)!!} & = (i-j)(i-2)!! + j\sum_{\ell = 0}^{(j-2)/2} \frac{(2\ell+i-j-2)!! \, (j-2)!! \, (i-j)}{(2\ell)!!} \nonumber\\
& = (i-j)(i-2)!! + j(i-2)!! = i!!,
\end{align}
where the base case of $j = 1$ an arbitrary $i$ has a single term with $(i -2)!! \cdot 1 \cdot i = i!!$.
Equation~\eqref{eq:even_i_pos_df} is proved analogously, or one can note that it is the same as~\eqref{eq:odd_i_pos_df} by adding $1$ to $i$ and $j$.
\end{proof}

Two noteworthy aspects of~\eqref{eq:pos_df_identities_sub_rhs} is that the right hand sides are independent of $j$ and each term is easily seen to be a positive integer since $2\ell \leq j$.
It would be interesting to have a bijective proof of these identities.

\begin{lemma}
\label{lemma:neg_df_id}
Equation~\eqref{eq:neg_df_identities} holds.
\end{lemma}

\begin{proof}
The proof for the induction step is analogous to the proof of~\eqref{eq:pos_df_identities_sub_rhs} as $A$ and the negative sum does not change.
Therefore, we only need to show the base cases.
We show the case for $i = 1$ and arbitrary (even) $j$ as the other case is similar.
Hence, setting $J := j/2 \in \ZZ$, Equation~\eqref{eq:neg_df_identities} becomes
\begin{equation}
\label{eq:df_alt_base_case}
\sum_{\ell = 0}^{J-1} \frac{(-1)^{\ell+J}}{2^{\ell} \ell! (2J-2\ell-1)!!}
+
\sum_{\ell=J}^J \frac{(2\ell-2J-1)!!}{2^{\ell} \ell!}
=
\frac{1}{(2J)!! \, (1-2J)}
\end{equation}
Next, we rewrite~\eqref{eq:df_alt_base_case} by multiplying both sides by $(2J-1)!!$ and bringing the $\ell = J$ term to the right hand side to the equivalent identity
\begin{equation}
\label{eq:df_alt_new_base}
\sum_{\ell = 0}^{J-1} \frac{(-1)^{\ell+J}}{(2\ell)!! (2(J-\ell)-1)!!} = -\frac{(2J-3)!!}{(2J-2)!!},
\end{equation}
since $2^J J! = (2J)!!$.

\begin{table}
\[
\begin{array}{cccccccccccccccccccc}
&&&&&&&&&&1\\
&&&&&&&&&1 && 1\\
&&&&&&&&1 && 2 && 1\\
&&&&&&&1 && \frac{3}{2} && \frac{3}{2} && 1\\
&&&&&&1 && \frac{8}{3} && 2 && \frac{8}{3} && 1\\
&&&&&1 && \frac{15}{8} && \frac{5}{2} && \frac{5}{2} && \frac{15}{8} && 1\\
&&&&1 && \frac{16}{5} && 3 && \frac{16}{3} && 3 && \frac{16}{5} && 1\\
&&&1 && \frac{35}{16} && \frac{7}{2} && \frac{35}{8} && \frac{35}{8} && \frac{7}{2} && \frac{35}{16} && 1\\
&&1 && \frac{128}{35} && 4 && \frac{128}{15} && 6 && \frac{128}{15} && 4 && \frac{128}{35} && 1\\
&1 && \frac{315}{128} && \frac{9}{2} && \frac{105}{16} && \frac{63}{8} && \frac{63}{8} && \frac{105}{16} && \frac{9}{2} && \frac{315}{128} && 1
\end{array}
\]
\caption{The first 10 rows of the double factorial binomial coefficients $\bigl(\!\binom{N}{K}\!\bigr)$.}
\label{table:df_binom}
\end{table}

It is convenient to introduce double factorial analogs of binomial coefficients,
\begin{equation}
    \dfbinom{N}{K} := \frac{N!!}{K!! (N-K)!!}
\end{equation}
with the convention that $\bigl(\!\binom{N}{K}\!\bigr) = 0$ for all $K < -1$ and $K > N + 1$.
In general, these are not integers, and we give the first few rows in the triangle in Table~\ref{table:df_binom}.
Hence~\eqref{eq:df_alt_new_base} in our new notation becomes
\begin{equation}
\sum_{\ell = 0}^{J-1} (-1)^{\ell+J} \dfbinom{2J-1}{2\ell} = -\frac{(2J-3)!!}{(2J-2)!!}.
\end{equation}
One can immediately verify that an analog of the usual recurrence relation for binomial coefficients holds for all $N > 0$ and $0 \leq K \leq N$:
\begin{equation}
    \dfbinom{N}{K} = \dfbinom{N-2}{K-2} + \dfbinom{N-2}{K}.
\end{equation}
Therefore, by induction on $J$ we have
\begin{align}
\sum_{\ell = 0}^{J-1} (-1)^{\ell+J} \dfbinom{2J-1}{2\ell} & = 
(-1)^J + \sum_{\ell = 1}^{J-2} (-1)^{\ell+J} \left[ \dfbinom{2J-3}{2(\ell-1)} + \dfbinom{2J-3}{2\ell} \right] - \dfbinom{2J-1}{2J-2} \nonumber \\
& = (-1)^J + (-1)^{J+1} + \dfbinom{2J-3}{2J-4} - \dfbinom{2J-1}{2J-2} \nonumber \\
& = \frac{(2J-3)!!}{(2J-4)!!} - \frac{(2J-1)!!}{(2J-2)!!} = -\frac{(2J-3)!!}{(2J-2)!!},
\end{align}
where the second equality is by telescoping and the base case of $J = 1$ is simply $-1 = -1$.
\end{proof}

\section{Examples}
\label{sec:examples}

In this section, we present a number of specific examples of our results.

\begin{remark}
\label{rem:rescaling}
The measure~\eqref{eq:GL-probability-measure} is invariant under the (simultaneous) rescaling $x_i \mapsto x_i/\beta$, $y_j \mapsto \beta y_j$ for any fixed $\beta \in \RR_{>0}$.
In particular, if $y_1 > 0$, we can normalize our measure so that $y_1 = 1$.
\end{remark}

\subsection{Constant functions and Krawtchouk polynomials}
\label{sec:const-funct-krawtch}

As the first example, we consider $x_i = \alpha$ and $y_j = 1$ for some fixed $\alpha \in \RR_{>0}$.
This case was studied in~\cite{GTW01} using saddle point analysis.
Then the measure~\eqref{eq:GL-probability-measure} becomes
\begin{equation}
  \label{eq:alpha-dim-measure}
  \mu_{n,k}(\lambda | \alpha)=\frac{\alpha^{\abs{\lambda}}
    \dim V_{\GL_{n}}(\lambda)\dim V_{\GL_{k}}(\lambda')}{(1 + \alpha)^{nk}},
\end{equation}
Using the Weyl dimension formula, we can write the measure as
\begin{equation}
  \label{eq:krawtchouk-measure}
  \mu_{n,k}(\lambda|\alpha)=\frac{1}{Z_{n,k}(\alpha)}\prod_{i<j}(a_{i}-a_{j})^{2}\prod_{l=1}^{n}\binom{n+k-1}{a_{l}}\alpha^{a_{l}},
\end{equation}
where $Z_{n,k}(\alpha) = (1+\alpha)^{nk}$.
Setting $\alpha = \frac{p}{1-p}$, we obtain the weight for the Krawtchouk polynomials: $W_{n,k}(b) = \sqrt[n]{Z_{n,k}(\alpha)^{-1}} \binom{n+k-1}{b} p^b (1-p)^{n+k-1-b}$. By Remark~\ref{rem:rescaling}, if we take $y_{j}=\beta$ then rescale $\alpha \mapsto \alpha/\beta$, we arrive again at the measure~\eqref{eq:alpha-dim-measure}.

Next, under this specialization Equation~\eqref{eq:zdz-S-eq-zero} becomes
\begin{equation}
  \label{eq:zdz-S-eq-zero-scaled-plancherel}
  \frac{\alpha z}{1 - \alpha z} + \frac{c}{z + 1} - t = 0, 
\end{equation}
and Equation~\eqref{eq:zdz-2-S-eq-zero} evaluates as
\begin{equation}
  \label{eq:zdz-2-S-eq-zero-scaled-plancherel}
  \frac{\alpha z}{(1 - \alpha z)^{2}} - \frac{c z}{(z + 1)^2} = 0. 
\end{equation}
Solving Equation~\eqref{eq:zdz-2-S-eq-zero-scaled-plancherel} for $z$, we obtain the roots
\begin{equation}
\label{eq:zpm_constant}
z_{0} = 0,
\qquad\qquad
z_{\pm} = \frac{\alpha (c+1) \pm (\alpha + 1) \sqrt{\alpha c}}{\alpha ( \alpha c - 1)}.
\end{equation}
Substituting $z_{\pm}$ into~\eqref{eq:zdz-S-eq-zero-scaled-plancherel}, we obtain
  \begin{equation}
    \label{eq:xpm-for-const}
    t_{\pm} = \frac{\alpha (c-1) \pm 2\sqrt{\alpha c}}{\alpha + 1}    
  \end{equation}
as the end points for the limit shape, and we compute
\begin{equation}
t_{+} - t_{-} = \frac{4\sqrt{\alpha c}}{\alpha + 1}.
\end{equation}
as the total length of the interval containing the limit shape.

To compute the limit shape function, we solve~\eqref{eq:zdz-S-eq-zero-scaled-plancherel} for $z$, where we generically have two roots
\begin{equation}
  \label{eq:z-for-const}
  z_{1,2} = \frac{\alpha (c-1)+t(1-\alpha) \mp \sqrt{4\alpha(t+1)(t-c)+(\alpha(c-1)+t(1-\alpha))^{2}}}{2\alpha(t+1)}.
\end{equation}
For $t \in (t_{-},t_{+})$, we have a pair of complex conjugate roots.
Hence, we have
\begin{equation}
\imi \Im z_{1,2} = \frac{\pm \sqrt{4\alpha(t+1)(t-c)+(\alpha(c-1)+t(1-\alpha))^{2}}}{2\alpha(t+1)},
\end{equation}
and therefore
\begin{equation}
\label{eq:rho-for-alpha-dimensions}
\rho(t) = \frac{1}{\pi}\arg z
= \frac{1}{\pi}\arccos\left( \frac{\alpha (c-1) + t(1-\alpha)}{2\sqrt{\alpha(c-t)(t+1)}} \right)
\quad (t\in [t_{-},t_{+}]).
\end{equation}
The earliest paper known to the authors that contains this asymptotic result is~\cite{ismail1998strong}, where it was first obtained from the study of Krawtchouk polynomials.

Next we want to compute the edge asymptotics.
We substitute $t_{+}$ from~\eqref{eq:xpm-for-const} to~\eqref{eq:z-for-const} and obtain
\begin{equation}
  \label{eq:z-crit-boundary-for-const}
  \zcrit = \frac{\alpha(c+1)-(\alpha+1)\sqrt{\alpha c}}{\alpha (\alpha c-1)},
\end{equation}
with
\begin{equation}
  \label{eq:S-third-derivative-for-const}
  S'''(\zcrit) = -\frac{2\alpha^{2}(1+\sqrt{\alpha c})^{5}}{(\alpha+1)^{3}(\sqrt{\alpha}-\sqrt{c})\sqrt{c}}.
\end{equation}
For $\alpha, c > 0$ 
we have
\begin{equation}
  \label{eq:tracy-widom-normalization-for-const}
  \sigma=\frac{(\alpha+1)c^{\frac{1}{6}}}{\alpha^{\frac{1}{6}}(\sqrt{c}-\sqrt{\alpha})^{\frac{2}{3}}(1+\sqrt{\alpha c})^{\frac{2}{3}}}. 
\end{equation}
Furthermore, we note that $\rho(t_+) = \frac{1}{\pi} \arccos\left(\frac{\sqrt{c}-\sqrt{\alpha}}{\abs{\sqrt{c}-\sqrt{\alpha}}}\right)$, and hence $\Omega'(t_+) = \frac{c-\alpha}{\abs{c-\alpha}}$ by~\eqref{eq:limit-shape-as-integral}.
From this, it is easy to see that the limit shape ends on the upper (resp.\ lower) right boundary of the (tilted) rectangle if and only if $c > a$ (resp.\ $c < a$) if and only if it is (locally) strictly concave (resp.\ convex).
A similar analysis holds for the limit shape around $t_-$.
We give an example of the distribution of the longest row in Fig.~\ref{fig:tracy-widom-dimensions}.

\begin{figure}
  \begin{center}
  \includegraphics{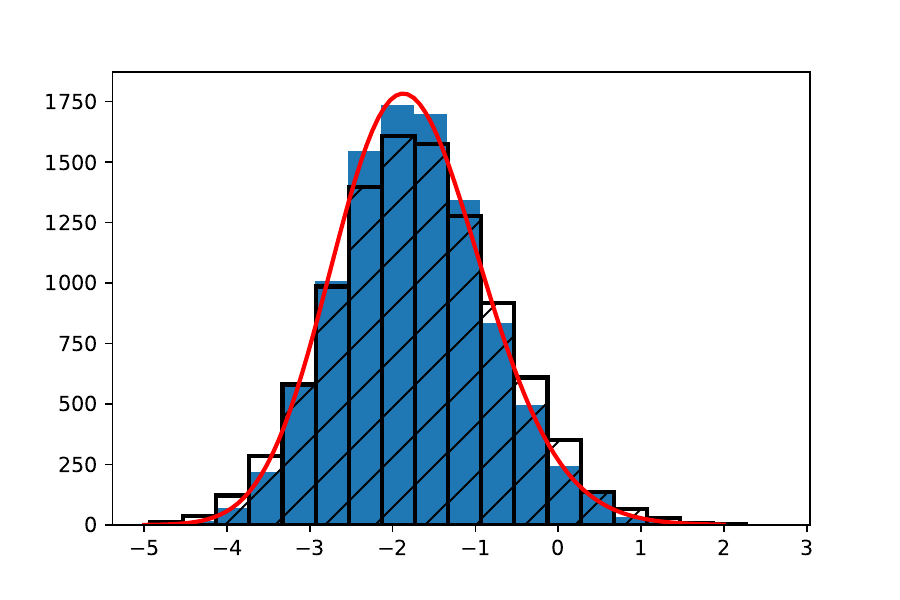}
  \end{center}
  \caption{Histogram of the normalized first row length fluctuations~\eqref{eq:tracy-widom-observable} in black for $n=200, k=400, \alpha=1$ and $10000$ samples, Tracy-Widom distribution in red, and histogram of $10000$ samples from Tracy-Widom distributions in blue.}
  \label{fig:tracy-widom-dimensions}
\end{figure}

Note that as $\alpha \to c$ then $\zcrit, z_- \to 0$, $z_+ \to \frac{2(c+1)}{c^2-1}$, $t_{+} \to c$, $\Omega'(t_+)$ jumps to $0$, and $\sigma$ diverges.
In this case, we have a different asymptotic description as the fluctuations depend on $n$ very weakly and as $n\to\infty$ are described by the discrete distribution given by~\eqref{eq:discrete_distribution}.
Hence, the probability of $\lambda_{1}=n$ is $\frac{1}{2}$, of $\lambda_{1}=n-1$ is $\frac{1}{4}+\frac{1}{2\pi}$, of $\lambda_{1}=n-2$ is $\frac{1}{8}-\frac{1}{8\pi}$ and is much smaller for smaller values of $\lambda_{1}$ (\textit{cf}.~\cite[Table~1]{GTW01}).

Now let us look at what happens if we assume $t_{+} = c$.
From~\eqref{eq:xpm-for-const}, this means that we must have $\alpha + c - 2\sqrt{\alpha c} = (\sqrt{\alpha} - \sqrt{c})^2 = 0$, which only occurs if $\alpha = c$.
On the other hand,~\eqref{eq:zpm_constant} implies we have $z_- = 0$ (recall $z_-$ dictates the value of $t_+$) if and only if $\frac{\sqrt{\alpha}}{\alpha+1} = \frac{\sqrt{c}}{c+1}$ (strictly speaking, we should assume $\alpha c - 1 \neq 0$, but it is easy to see this is a removable singularity for $z_-$).
This implies that $\alpha = c, c^{-1}$, but $\alpha = c^{-1}$ also sends the denominator to $0$.
By evaluating $\lim_{\alpha \to c^{-1}} z_- = \frac{1}{2}(c-1)$ using l'H\^opital's rule, we conclude that $z_- = 0$ if and only if $\alpha = c$ again.
Moreover, above we saw the limit shape $\Omega(t)$ is flat at $t_+$ if and only if $c = \alpha$.
Hence, we have proven Conjecture~\ref{conj:critical_classification} for this specialization.

Next, to directly compare with~\cite{GTW01}, we need the following translation of notation:
\begin{equation}
p \longleftrightarrow \frac{\alpha}{1 + \alpha},
\qquad\qquad
p_c \longleftrightarrow \frac{c}{1 + c},
\end{equation}
in particular, we can now view $\alpha$ as the transition rate yielding the probability $p$.
The critical regime of~\cite[Section 3.2]{GTW01} is when $p \to p_c$, which in our notation becomes $\alpha \to c$.
In terms of our limit shapes, the critical regime is precisely when $\lambda_1 \sim k$ but we are allowing $\lambda_2 < k$.
Thus, we are able to have some fluctuations in $\lambda_1$.
If instead we consider the deterministic regime with $p > p_c$, which in our notation is $c > \alpha$, then the right boundary of the limit shape is at $t_+ < c$ and then has $\rho(t) = 1$ for $t \in [t_+, c]$.
Therefore, we have a large number of rows of $\lambda$ equal to $k$ almost surely for $n \gg 1$; in particular, we have $\lambda _1 = k$ almost surely.


\subsection{Piecewise-constant specialization and multi-interval support of limit shape}
\label{sec:piec-const-spec}

For this example, we are essentially doing the estimation used in the proof of Theorem~\ref{thm:limit-shape} in Section~\ref{sec:bulk-asymptotics}, but extending it for a macroscopic share of parameters.
More precisely, we use $x_{i}=\alpha_{1}$ for $i=1,\dotsc,\lfloor A_{1}n\rfloor$, $x_{i}=\alpha_{2}$ for $i=\lceil A_{1}n\rceil,\dotsc,\lfloor (A_{1}+A_{2})n\rfloor$ and so on with constants $\alpha_{1},\dotsc, \alpha_{u}$ and shares $A_{1},\dotsc, A_{u}$, where $\sum_{i=1}^{u}A_{i}=1$.
Similarly denoting the constants for $y_{j}$ by $\beta_{1},\dotsc, \beta_{v}$  and shares by $B_{1},\dotsc,B_{v}$ such that $\sum_{j=1}^{v}B_{j}=c$, we get
\begin{equation}
  \label{eq:multi-const-action}
  S(z,t)=-\sum_{i=1}^{u} A_{i}\ln(1-\alpha_{i}z)-\sum_{j=1}^{v}B_{j}\ln(z+\beta_{j})+(c-t)\ln z.
\end{equation}
The critical points of the action are determined from Equation \eqref{eq:zdz-S-eq-zero}, which in this case takes the form
\begin{equation}
  \label{eq:zdz-S-for-multiple-constants}
  t+1=\sum_{i=1}^{u} \frac{A_{i}}{1-\alpha_{i}z}+\sum_{j=1}^{v}\frac{B_{j}\beta_{j}}{\beta_{j}+z}.
\end{equation}

To find the support of the limit density we need to solve Equation~\eqref{eq:zdz-2-S-eq-zero}, which reads
\begin{equation}
  \label{eq:multi-const-zdz-2}
  \frac{z\left(\sum_{i=1}^{u}A_{i}\alpha_{i}\prod_{j=1}^{v}(z+\beta_{j})^{2}\prod_{k\neq i}(1-\alpha_{k}z)^{2}-\sum_{j=1}^{v}B_{j}\beta_{j}\prod_{i=1}^{u}(1-\alpha_{i}z)^{2}\prod_{k\neq j}(z+\beta_{k})^{2}\right)}{\prod_{i=1}^{u}(1-\alpha_{i}z)^{2}\prod_{j=1}^{v}(z+\beta_{j})^{2}}=0,
\end{equation}
for real roots $z$ in the interval $[-1,c]$.
As we have already seen in Section~\ref{sec:const-funct-krawtch}, for $u=v=1$ we get the quadratic equation in the parentheses if $\alpha\beta c\neq 1$ (and without loss of generality by Remark~\ref{rem:rescaling}, we can take $\beta = 1$).
The case $\alpha\beta c=1$ corresponds to one of the roots being $z=0$ and the support of the density ending in the corner of the rectangle, as discussed above.
For $u,v>1$ this is a difficult problem in general, so we discuss only fourth order equation here.

As such, we take $u=2, v=1$ or $u=1, v=2$ to obtain~\eqref{eq:multi-const-zdz-2} as a fourth order equation (times $z$), but these two cases are equivalent under a change $n \leftrightarrow k$.
Hence, without loss of generality, take $u=1, v=2$, the fourth order equation is
\begin{equation}
\label{eq:fourth_order_piecewise}
\alpha(z+\beta_{1})^{2}(z+\beta_{2})^{2}-B_{1}\beta_{1}(1-\alpha z)^{2}(z+\beta_{2})^{2}-(c-B_{1})\beta_{2}(1-\alpha z)^{2}(z+\beta_{1})^{2} = 0.
\end{equation}
The analysis in generic case involve very cumbersome expressions (but it can be computed explicitly), so we assume for simplicity that $\alpha=1$, $B=\frac{c}{2}, \beta_{1} = \beta, \beta_{2} = \frac{1}{\beta}$.
We remark that the latter choice corresponds to the skew Howe duality for symplectic or orthogonal groups, which we discuss in more detail in Section~\ref{sec:gl=2sp} below.
The condition for a two interval support then corresponds to the condition of symplectic Young diagram to start away from the corner of the corresponding rectangle.
Equation~\eqref{eq:fourth_order_piecewise} can have four real roots only if discriminant is positive, which reads
\begin{equation}
\frac{1}{\beta^{10}}(\beta-1)^2 (\beta+1)^{12} c^2 \left((\beta-1)^2-4 \beta c\right) (\beta (2 \beta+c-4)+2)^2 > 0.
\end{equation}
This holds for  $\beta>2c+1+2\sqrt{c(c+1)}$ and for $\beta<2c+1-2\sqrt{c(c+1)}$. For the roots to be real we also need
\begin{multline}
  P = -16 c \beta ^7-4 \left(-c^2+12 c+4\right) \beta ^6-4 \left(4 c^2+4 c\right) \beta ^5-4 \left(10 c^2-8c-8\right) \beta ^4 \\
  -4 \left(4 c^2+4 c\right) \beta^3-4 \left(-c^2+12 c+4\right) \beta ^2-16 c \beta<0
\end{multline}
and
\begin{multline}
  D=-\beta ^2 16 c \left(4 c \beta^{12}+ (32 c+16) \beta^{11}+ \left(-c^3+8 c^2+48 c+32\right) \beta ^{10} \right.\\\left.
   +\left(-8 c^3+64 c^2-16\right) \beta^9+\left(-12 c^3+128 c^2-20 c-64\right) \beta ^8\right.\\\left.
   +\left(8 c^3+64 c^2+32 c\right) \beta^7+\left(26 c^3-16 c^2+64 c+64\right) \beta^6\right. \\\left.
   +\left(8 c^3+64 c^2+32 c\right) \beta^5+\left(-12 c^3+128 c^2-20 c-64\right) \beta^4\right.\\\left.
   +\left(-8 c^3+64 c^2-16\right) \beta^3+\left(-c^3+8 c^2+48 c+32\right) \beta^2+(32 c+16) \beta +4 c\right)<0.
\end{multline}
For positive $c$ it is possible to check that these polynomials do not have roots greater than $2c+1+2\sqrt{c(c+1)}$, so for all $\beta>2c+1+2\sqrt{c(c+1)}$ we have four real roots in \eqref{eq:multi-const-zdz-2} and two-interval support of the density $\rho(t)$.
The graph of the corresponding function $T(z)$ from Equation~\eqref{eq:zdz-S-for-multiple-constants} is presented in Fig.~\ref{fig:gl-multi-interval-action-derivative}. 
We present the examples of random Young diagrams with two-interval supports of the density in Fig. \ref{fig:gl-multicut-diagrams}.  

\begin{figure}
  \begin{center}
  \includegraphics[width=10cm]{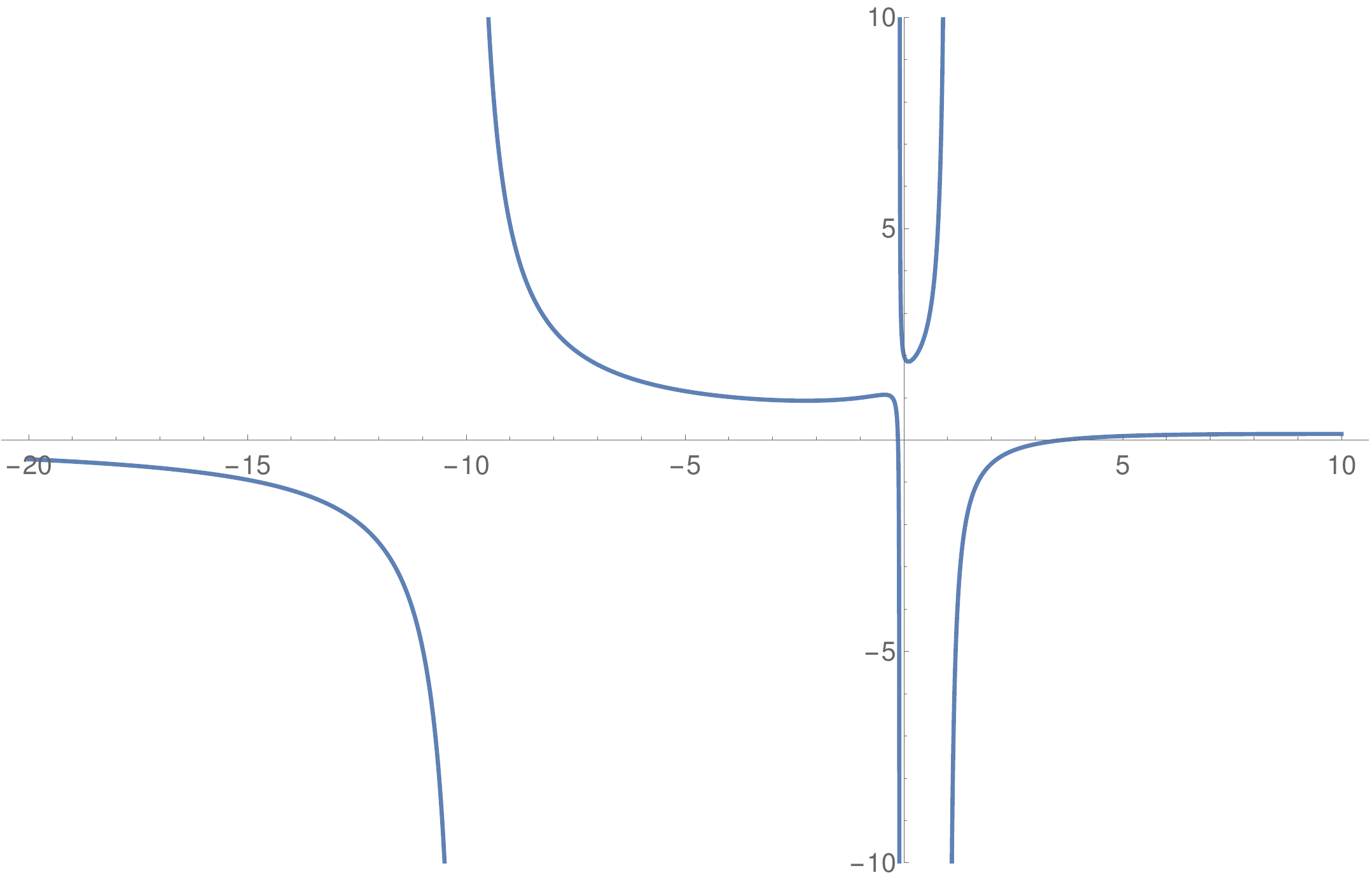}
  \end{center}
  \caption{Plot of the function $T(z)$ for $c=1$ in the two-interval case with $u=1, v=2$ $A_{1}=1$, $\alpha_{1}=1$, $B_{1}=B_{2}=c/2$, $\beta_{1}=1/10, \beta_{2}=10$. Note that horizontal line crosses the graph of $T(z)$ in either one or three points, which means that \eqref{eq:zdz-S-for-multiple-constants} has either two complex-conjugate roots or only real roots.}
  \label{fig:gl-multi-interval-action-derivative}
\end{figure}

\begin{figure}
  \begin{center}
  \includegraphics[width=10cm]{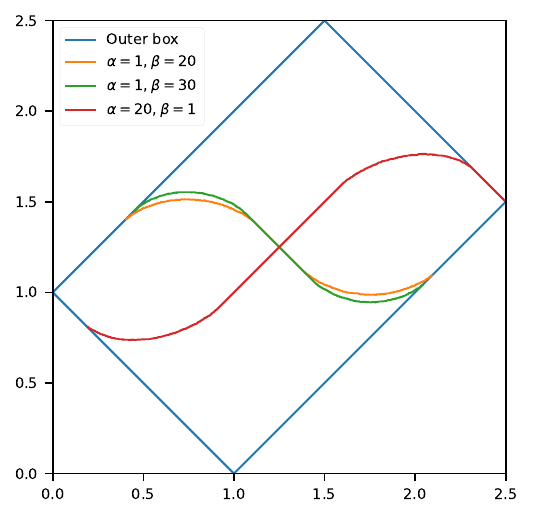}
  \end{center}
  \caption{Random Young diagrams with $n=800,k=1200$ for specializations $\{x_{i}\}_{i=1}^{n}=(\alpha,\dotsc,\alpha,\alpha^{-1},\dotsc,\alpha^{-1})$ and $\{y_{j}\}_{j=1}^{k}=(\beta,\dotsc,\beta,\beta^{-1},\dotsc,\beta^{-1})$.}
  \label{fig:gl-multicut-diagrams}
\end{figure}

For mutually inverse values of $\beta=2c+1\pm 2\sqrt{c(c+1)}$, the two intervals join in one point.
This leads to the behavior described by Pearcey kernel as per Theorem~\ref{thm:Pearcey_kernel}.
In more detail, for $\beta=2c+1+ 2\sqrt{c(c+1)}$ , Equation~\eqref{eq:fourth_order_piecewise} becomes
\begin{equation}
-2 (c+1) \left(8 c \left(c+\sqrt{c (c+1)}+1\right)+4\sqrt{c (c+1)}+1\right)(z+1)^{2}\left((2 c-1) z^2-(8 c+2) z+2 c-1\right) = 0,
\end{equation}
and we see that $z=-1$ is its multiple root.
Substituting $\beta$ and $z=-1$ to~\eqref{eq:zdz-S-eq-zero} and solving for $x$ we get $x=\frac{c-1}{2}$ as the point where two intervals of support are touching.
The first three derivatives of the action with respect to $z$ at $x=\frac{c-1}{2}$ and $z=-1$ are zero.
For the fourth derivative we have $\partial_{z}^{4}S(z,(c-1)/2)|_{z=-1}=\frac{3(c+1)}{8c}$, and similarly to the Airy kernel case we set $\sigma=-\left(\frac{6}{\partial_{z}^{4}S(z,(c-1)/2)|_{z=-1}}\right)^{\frac{1}{4}}=-2\sqrt[4]{\frac{c}{c+1}}$ and change the variables $z=-e^{\sigma\zeta n^{-\frac{1}{4}}}\approx -(1+\sigma \zeta n^{-\frac{1}{4}}+\cdots)$, $w=-e^{\sigma\nu n^{-\frac{1}{4}}}\approx -(1+\sigma \nu n^{-\frac{1}{4}}+\cdots)$.
Hence, according to Theorem~\ref{thm:Pearcey_kernel} in the asymptotic regime $m\approx \frac{c-1}{2}n+\xi n^{\frac{1}{4}}\sigma^{-1}$, $m'\approx \frac{c-1}{2}n+\eta n^{\frac{1}{4}}\sigma^{-1}$ as $n,k\to \infty$ the correlation kernel converges to the Pearcey kernel~\eqref{eq:pearcey-kernel}. 
This regime is illustrated in Fig.~\ref{fig:gl-pearcey-lozenge-tiling}, where we present a random lozenge tiling of a skew hexagon.

\begin{figure}
  \begin{center}
  \includegraphics[width=14cm]{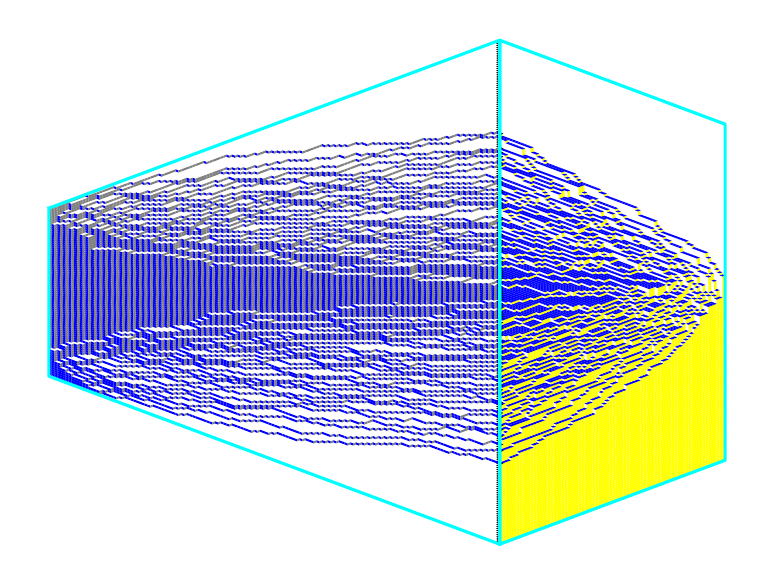}
  \end{center}
  \caption{Lozenge tiling of a skew hexagon corresponding to Pearcey regime for $n=40, k=80, c=2, \beta=2c+1+ 2\sqrt{c(c+1)}\approx 9.9$  for specializations $\{x_{i}\}_{i=1}^{2n}=(1,1\dotsc,1)$ and $\{y_{j}\}_{j=1}^{2k}=(\beta, \beta^{-1},\dotsc,\beta,\beta^{-1})$.}
  \label{fig:gl-pearcey-lozenge-tiling}
\end{figure}

\subsection{Monomial functions}
\label{sec:monomial-functions}

We consider $f(s) = \alpha s^{\ell}$ and $g(s) = s^m$ for some fixed nonnegative real numbers $\ell, m \in \RR_{\geq 0}$. 
Then the integrals in Equation~\eqref{eq:zdz-S-eq-zero} and Equation~\eqref{eq:zdz-2-S-eq-zero} are instances of Chebyshev's differential binomial integral (see, \textit{e.g.},~\cite{Hardy71,piskunov1965differential}):
\begin{subequations}
\begin{align}
\int \ds \; s^\ta (\kappa + \nu s^{\tb})^{\tc} & = \frac{\kappa^{\tc+\frac{\ta+1}{\tb}} \nu^{-\frac{\ta+1}{\tb}}}{\tb} B\left(-\frac{\nu s^{\tb}}{\kappa}; \frac{1+\ta}{\tb}, \tc+1\right)
\\ & = \frac{\kappa^\tc s^{\ta+1}}{1+\ta} {}_2F_1\left( \frac{\ta+1}{\tb}, -\tc; \frac{1 + \ta + \tb}{\tb}; - \frac{\nu s^{\tb}}{\kappa} \right),
\end{align}
\end{subequations}
where $B(y; \ta, \tb) = \int_0^y \ds \; s^{\ta-1} (1-s)^{\tb-1}$ is the (lower) incomplete Beta function, and ${}_2F_1(\ta_1, \ta_2; \tb_1; x)$ is the basic hypergeometric function.
Thus, Equations~\eqref{eq:zdz-S-eq-zero} and~\eqref{eq:zdz-2-S-eq-zero} for $\ell > -1$ and $m > -1$ become, respectively,
\begin{subequations}
\label{eq:hypergeom-S-eq}
\begin{gather}
  \label{eq:zdz-S-eq-zero-hypergeom}
  \frac{\alpha z}{\ell+1} {}_2F_1\left( 1, 1 + \frac{1}{\ell}; 2 + \frac{1}{\ell}; \alpha z \right) + \frac{c}{(m+1)z} {}_2F_1\left( 1, 1 + \frac{1}{m}; 2 + \frac{1}{m}; -\frac{1}{z} \right) - t = 0,
\allowdisplaybreaks\\
  \label{eq:zdz-2-S-eq-zero-hypergeom}
  \frac{\alpha z}{\ell+1} {}_2F_1\left( 2, 1 + \frac{1}{\ell}; 2 + \frac{1}{\ell}; \alpha z \right) - \frac{c}{(m+1) z} {}_2F_1\left( 2, 1 + \frac{1}{m}; 2 + \frac{1}{m}; -\frac{1}{z} \right) = 0.
\end{gather}
\end{subequations}

We note that there are particular formulas for $\ell, m \in \ZZ$ that simplify~\eqref{eq:hypergeom-S-eq}, where the integrals can be computed in terms of more elementary functions.
For example, if we take $\ell = m = 1$, then
\begin{subequations}
\begin{gather}
 \label{eq:zdz-S-eq-zero-linear}
  -\frac{1}{\alpha z}\ln(1 - \alpha z) - 1 + cz \ln\left(\frac{z}{1+z}\right) + c - t = 0, 
\allowdisplaybreaks\\
  \label{eq:zdz-2-S-eq-zero-linear}
  \frac{(\alpha z - 1) \ln(1 - \alpha z) - \alpha z}{\alpha z (\alpha z - 1)} + c \left( z\ln\left(\frac{z}{z+1}\right) + \frac{z}{z + 1} \right)= 0.
\end{gather}
\end{subequations}
However, except for $\ell = m = 0$, we are unable to solve~\eqref{eq:zdz-2-S-eq-zero-hypergeom} (for $z$) explicitly, but it is possible to compute the limit shape numerically.
In more detail, we compute the two roots $z_{\pm}$ numerically, where an example for $\ell = m = 1$ and $\alpha =1$ is illustrated in Fig.~\ref{fig:zDz-2-S}. 
Then the support of the density $[t_{-},t_{+}]$ is computed numerically by substituting the roots into Equation~\eqref{eq:zdz-S-eq-zero-linear}, and finally the limit shape can then be obtained numerically from Equation~\eqref{eq:rho-arg-z}.
Examples of random diagrams sampled using dual RSK and the limit shapes computed from the solution numerically are presented in Fig.~\ref{fig:limit-shape-monomial-weights}.

\begin{figure}
  \begin{center}
  \includegraphics[width=10cm]{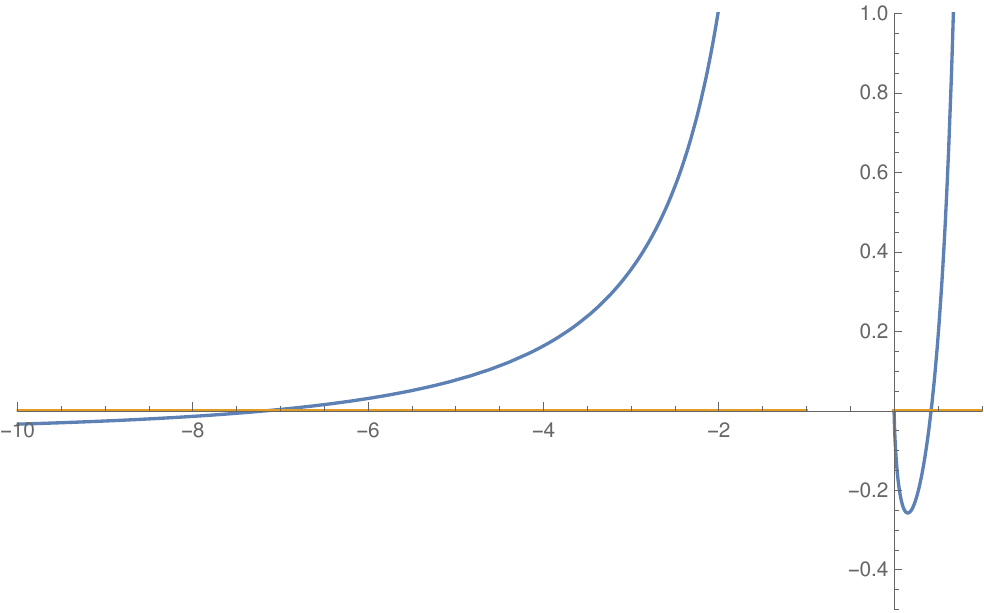}
  \end{center}
  \caption{Real (blue) and imaginary (orange) parts of
    $(z\partial_z)^2 S(z)$ as given in Equation~\eqref{eq:zdz-2-S-eq-zero-linear} for $\alpha = 1$ $c=2$.}
  \label{fig:zDz-2-S}
\end{figure}

\begin{figure}
\begin{subfigure}{0.3\textwidth}
  \[
  \vspace{1pt}
  \includegraphics[width=1.0\linewidth]{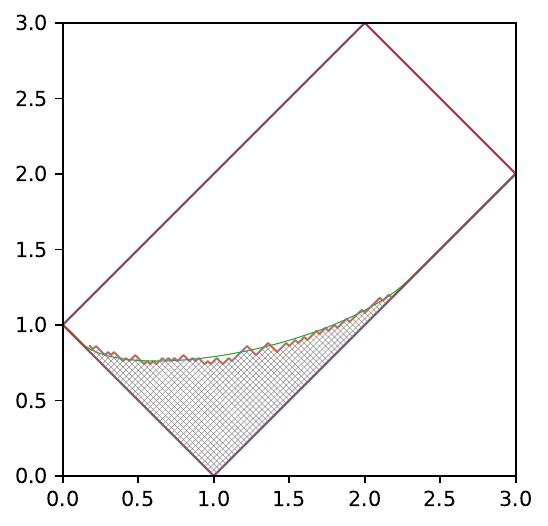}
  \]
  \caption{$\ell = m = 1$, $\alpha = 1$.}
\end{subfigure}
\begin{subfigure}{0.3\textwidth}
  \[
  \includegraphics[width=1.0\linewidth]{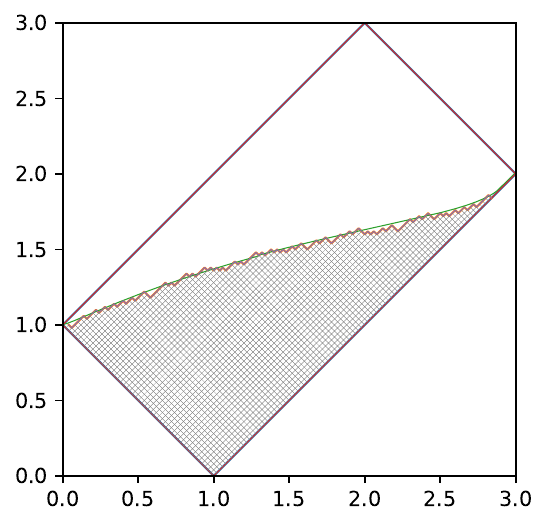}
  \]
  \caption{$\ell = m = 1$, $\alpha = 6$.}
\end{subfigure}
\begin{subfigure}{0.3\textwidth}
  \[
  \includegraphics[width=1.0\linewidth]{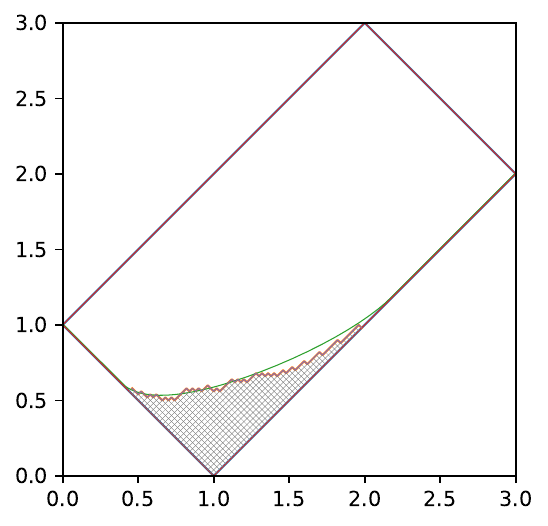}
  \]
  \caption{$\ell=3.833$, $m=2.55$, $\alpha=3$.}
\end{subfigure}
  \caption{
  Random diagram for $n=50, k=100$, $x_i=(i/n)^{\ell}, y_j=(j/k)^m$ and limit shape (green) derived by the numerical solution of Equations~\eqref{eq:hypergeom-S-eq}.
  }
  \label{fig:limit-shape-monomial-weights}
\end{figure}

\subsection{Alternating weights and symplectic Young diagrams}
\label{sec:gl=2sp}

As another example consider the groups $\GL_{2n}\times \GL_{2k}$ and interlacing
specializations 
$x_{2i-1}=f\left(\frac{i}{n}\right), x_{2i}=f\left(\frac{i}{n}\right)^{-1}$ and
$y_{2j-1}=g\left(\frac{j}{k}\right), y_{2j}=g\left(\frac{j}{k}\right)^{-1}$.
As the order of specialization parameters is not important, we can take $x_{i} = f\left(\frac{i}{n}\right)$ for $i\leq n$ and $x_{i}=f\left(\frac{i}{n}\right)^{-1}$ for $i>n$ and similarly for $y_{j}$, but the interlacing order is more natural in comparison to the symplectic groups.
This particular case is a bit more general than Theorem~\ref{thm:correlation-kernel-bulk}.
It is interesting in relation to the skew Howe duality between symplectic groups $\Sp_{2n}$ and $\Sp_{2k}$.
The exterior algebra $\bigwedge \left(\CC^{n}\otimes \CC^{2k}\right)$ admits a multiplicity-free action of the direct product $\Sp_{2n}\times\Sp_{2k}$.
Therefore we have a decomposition into the irreducible representations
\begin{equation}
  \label{eq:sp-sp-duality}
  \bigwedge \left(\CC^{n}\otimes \CC^{2k}\right)=\bigoplus_{\lambda\subseteq k^{n}} V_{\Sp_{2n}}(\lambda)\otimes V_{\Sp_{2k}}(\overline\lambda'),
\end{equation}
where $\overline\lambda'$ is a conjugate to a diagram that complements $\lambda$ inside of the $n\times k$ rectangle.
Writing this decomposition in terms of the characters, we can introduce the probability measure
\begin{equation}
  \label{eq:sp-probability-measure}
  \mu_{n,k}^{\Sp}(\lambda|x,y)=\frac{\operatorname{sp}_{\lambda}(x_{1},\dotsc,x_{n}) \operatorname{sp}_{\overline\lambda'}(y_{1},\dotsc,y_{k})}{\prod_{i=1}^{n}\prod_{j=1}^{k}(x_{i}+x_{i}^{-1}+y_{j}+y_{j}^{-1})},
\end{equation}
where we have denoted by $\operatorname{sp}_{\lambda}$ character of irreducible representation of symplectic group $\Sp_{2n}$.
On the other hand, the exterior algebra  $\bigwedge \left(\CC^{n}\otimes \CC^{2k}\right)$ can be seen as the spinor representation $\bigwedge \left(\CC^{n}\right)$ of $\Sp_{2n}$, raised to the tensor power $2k$, $\bigl(\bigwedge \CC^{n} \bigr)^{\otimes 2k}$.
This tensor power can be implemented by the Berele insertion algorithm~\cite{berele1986schensted}; see also~\cite{proctor1993reflection} for the proof of the $(\Sp_{2n}, \Sp_{2k})$ duality with the insertion algorithm.
This algorithm can be used to sample random diagrams with respect to the probability measure~\eqref{eq:sp-probability-measure}. 

In the paper~\cite{nazarov2021skew}, we have demonstrated that for the trivial specialization $\{x_{i}=1, y_{j}=1\}$, that is when the measure is given by the formula
\begin{equation}
\mu_{n,k}^{\Sp}(\lambda)=\frac{\dim V_{\Sp_{2n}}(\lambda)\dim V_{\Sp_{2k}}(\overline\lambda')}{2^{2nk}},
\end{equation}
the limit shape of the symplectic Young diagrams when $n,k\to\infty$ such that $\frac{k}{n} = c + \mathcal{O}\left(\frac{1}{n}\right)$ is half of the limit shape of Young diagrams for $\GL_{2n}\times \GL_{2k}$.

Using the Berele sampling algorithm we conjecture that the limit shape of symplectic Young diagrams with respect to the measure~\eqref{eq:sp-probability-measure} for $x_{i}=f\left(\frac{i}{n}\right)$, $y_{j}=g\left(\frac{j}{k}\right)$ is half of the limit shape of Young diagrams for $\GL_{2n}\times \GL_{2k}$, derived below.
It would be interesting to prove our conjecture by using a free fermionic (or vertex operator) construction for the symplectic characters; see, \textit{e.g.},~\cite{Baker_1996,Betea18,JLW24}, but it is beyond of scope of the present paper.
See also Question~(\ref{Q:dual_pairs}) in Section~\ref{sec:conclusion-outlook}.
See Fig.~\ref{fig:gl-sp-diagram-limit-shape} for a random symplectic diagram and the limit shape for $f(s)=e^{-\gamma s}$, $g(s)=e^{\gamma c s}$.

\begin{figure}[t]
  \centering
  \includegraphics[width=.8\textwidth]{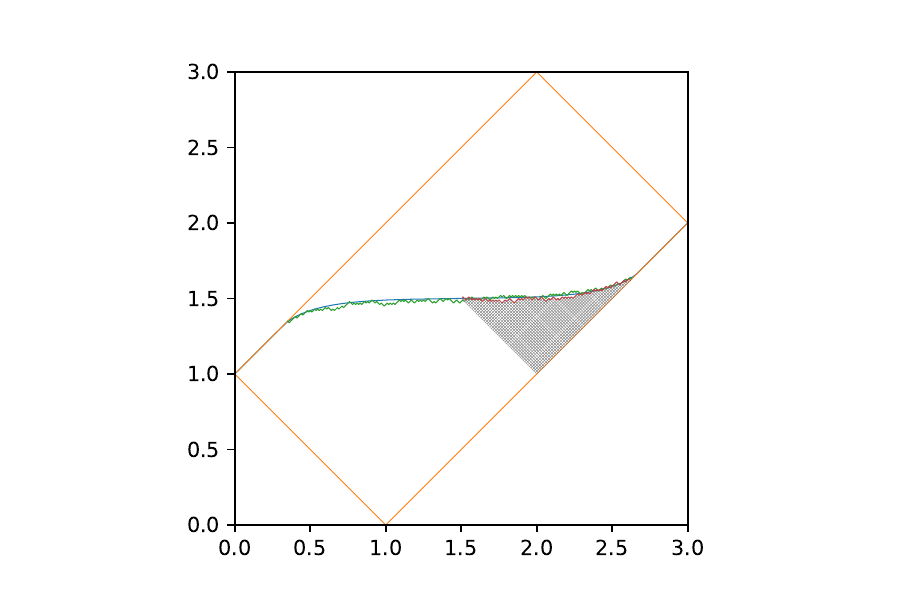}
  \caption{Random symplectic Young diagram for $n=50, k=100, \gamma=4$, its upper boundary (red line), upper boundary of random Young diagram for $\GL_{2n}\times \GL_{2k}$ (green line) and limit shape (blue line).}
  \label{fig:gl-sp-diagram-limit-shape}
\end{figure}
To derive the limit shape we first separate variables with odd and even indices:
\begin{equation}
	K(z)=\prod_{i=1}^{2n}\frac{1}{1-x_{i}z}\prod_{j=1}^{2k}\frac{1}{1+y_{j}/z}
	=\prod_{i=1}^{n}\frac{1}{1-f(i/n)z}\frac{1}{1-\frac{z}{f(i/n)}}\prod_{j=1}^{k}\frac{1}{1+\frac{g(j/k)}{z}}\frac{1}{1+\frac{1}{g(j/k)z}}. 
\end{equation}
Then for the action we have
\begin{align*}
	S(z) &= \frac{1}{2n}\ln K(z)-t\ln z \\
	& \approx \frac{1}{2} \int_{0}^{1} \ds\left[ - 
	\ln(1-f(s)z) -  \ln\left(1-\frac{z}{f(s)}\right) - c  \ln\left(1+\frac{g(s)}{z}\right)
	- c  \ln\left(1+\frac{1}{g(s)z} \right) \right] - t \ln z.
\end{align*}
Taking the derivative we get the equations 
\begin{subequations}
\begin{align}
\label{eq:zdzS-periodic}
  (z\partial_{z})S(z) &= \frac{1}{2}\int_{0}^{1} \ds\left[
    \frac{f(s)z}{1-f(s)z}+ \frac{z}{f(s)-z}- \frac{cz}{z+g(s)}-
    \frac{cg(s)z}{1+g(s)z}\right]-\left(x-c\right)=0,
\allowdisplaybreaks \\
(z\partial_{z}  )^{2}S(z) &= \int_{0}^{1}\ds \left[\frac{f(s)z}{(1-f(s)z)^{2}}+\frac{f(s)z}{(f(s)-z)^{2}}-\frac{cg(s)z}{(1+g(s)z)^{2}}-\frac{cg(s)z}{(z+g(s))^{2}}\right]=0.
\end{align}
\end{subequations}

The piecewise-constant specialization with $x_{i}=1, i=1,\dotsc, n$ and $y_{2j}=\beta, y_{2j-1}=\beta^{-1}, j=1,\dotsc,k$ for $\GL_{2n}\times \GL_{2k}$ gives $f(s)=1, g(s)=\beta$ and was considered in Section~\ref{sec:piec-const-spec}.
It corresponds to the constant specialization $x_{i}=1, y_{j}=\beta$ for symplectic groups $\Sp_{2n}\times \Sp_{2k}$.
Here single-interval support of the limit density for general linear group means that the diagram for the symplectic group starts at the left corner of the $n\times k$ rectangle.
The two-interval regime leads to the limit shape touching to the boundary of the rectangle before getting to the corner, which is at the center of diagram in Fig.~\ref{fig:gl-multicut-diagrams}.
The fluctuations in $\GL_{2n}\times \GL_{2k}$ case in this regime are described by the Pearcey kernel as demonstrated in Theorem \ref{thm:Pearcey_kernel} and discussed in Section~\ref{sec:piec-const-spec}.
The behavior in $\Sp_{2n}\times \Sp_{2k}$ is unknown, but we expect it to be described by the Pearcey kernel, symmetric Pearcey kernel~\cite{BK10}, or the related Pearcey-like kernel discussed in~\cite[Sec.~8.6]{cuenca2024symplecticschurprocess}.

Another specialization that admits the explicit solution is exponential.
In the case using the specialization $f(s)=e^{-\gamma s}$ and $g(s)=e^{c\gamma s}$, solving the first equation we get 
\begin{equation}
  z_{1,2}=\frac{-(e^{\gamma}-e^{c\gamma})(e^{c\gamma}+e^{\gamma(1+2t)})\pm\sqrt{(e^{\gamma}-e^{c\gamma})^{2}(e^{c\gamma}+e^{\gamma(1+2t)})^{2}-4e^{(c+1)\gamma}(e^{2c\gamma}-e^{2t\gamma})(e^{2\gamma(t+1)}-1)}}{2e^{c\gamma}(e^{2\gamma(t+1)}-1)}.
\end{equation}
Solving the second equation we again get the equation for the critical values of $z$:
\begin{equation}
  \frac{(e^{\gamma(c+1)}+1)z\left(\left(e^{\gamma}-e^{c\gamma}\right)z^{2}+2\left(e^{\gamma(c+1)}-1\right)z-e^{c\gamma}+e^{\gamma}\right)}{\gamma\left(e^{\gamma}-z\right)\left(e^{c\gamma}+z\right)\left(e^{\gamma}z-1\right)\left(e^{c\gamma}z+1\right)}=0.
\end{equation}
The solutions are
\begin{equation}
  z_{\pm}=\frac{e^{2\gamma(c+1)}-1\pm\sqrt{(e^{2\gamma}-1)(e^{2\gamma c}-1)(e^{\gamma(c+1)}+1)^{2}}}{-e^{\gamma}+e^{\gamma c}-e^{\gamma(c+2)}+e^{\gamma(2c+1)}}.
\end{equation}
Substituting into~\eqref{eq:zdzS-periodic} we get the support of the density:
\begin{equation}
  t_{\pm}=\frac{1}{2\gamma}\ln\left(\frac{2e^{2\gamma(c+1)}+2e^{3\gamma(c+1)}-e^{\gamma(c+3)}-e^{\gamma (3c+1)}\mp\sqrt{e^{2\gamma(c+1)}(e^{2\gamma}-1)(e^{2c\gamma}-1)(e^{\gamma(c+1)}+1)^{2}}}{(e^{\gamma(c-1)}+1)^{2}}\right).
\end{equation}
Using \eqref{eq:rho-arg-z} we get the limit density for $t\in[t_{-},t_{+}]$:
\begin{equation}
	\label{eq:sp_half_limit_density}
	\rho(t)=\frac{1}{\pi}\arccos\left(\frac{e^{2c\gamma}-e^{\gamma(c+1)}-e^{2\gamma(t+1)}+e^{\gamma(2t+c+1)}}{2\sqrt{(e^{\gamma(2c+1)}-e^{\gamma(2t+1)})(e^{\gamma(2t+c+2)}-e^{c\gamma})}}\right)
\end{equation}
An example of this limit shape is presented in Fig.~\ref{fig:gl-sp-diagram-limit-shape}.
This case can be seen as a particular case of the exponential specialization $x_{i}=\alpha e^{-\gamma\frac{i-1}{n}}$, $y_{j}=\beta e^{-\delta \frac{j-1}{k}}$ discussed in the next section.
To see that we need to rearrange the interlacing parameters $\{x_{2i-1},x_{2i}\}_{i=1}^{n}$, $\{y_{2j-1},y_{2j}\}_{j=1}^{k}$ in the increasing order to get the sequences
\begin{equation}
\{e^{-\gamma\frac{n-1}{n}},\dotsc, e^{-\gamma\frac{1}{n}},1,1,e^{\gamma \frac{1}{n}},\dotsc, e^{\gamma\frac{n-1}{n}}\},
\qquad
\{e^{-c\gamma\frac{k-1}{k}},\dotsc, e^{-c\gamma\frac{1}{k}},1,1,e^{c\gamma \frac{1}{k}},\dotsc, e^{c\gamma\frac{k-1}{k}}\}.
\end{equation}
Extracting the constants $\alpha=e^{\gamma}$, $\beta=e^{c\gamma}$ and ignoring the single duplicated value we get to the exponential specialization.

\section{Principal specializations and \texorpdfstring{$q$}{q}-Krawtchouk ensemble}
\label{sec:princ-spec-q-krawtchouk}

Another natural choice of specialization parameters is to take $x_{i}
= \alpha q^{i-1}, y_{j} = q^{j-1}$ or $x_{i} = \alpha q^{i-1},
y_{j}=q^{1-j}$ for some positive $q$. Both of these cases can be interpreted from the point of view of $q$-Krawtchouk polynomial ensemble~\cite{nazarov2022skew} as follows.
If we substitute these specializations into the probability measure
\eqref{eq:GL-probability-measure}, as was demonstrated in~\cite{nazarov2021skew}, the measures then take the form
\begin{subequations}
\begin{align}
    \label{eq:q-krawtchouk-measure}
  \mu_{n,k}(\lambda|\alpha, q, q) & = \frac{\alpha^{\abs{\lambda}} q^{\Abs{\lambda}} \dim_{q}\left(V_{\GL_{n}}(\lambda)\right)\cdot
    q^{\Abs{\bar{\lambda}'}}\dim_{q}\left(V_{\GL_{k}}(\bar{\lambda}')\right)}
  {\prod_{i=1}^{n}\prod_{j=1}^{k}(q^{i-1}+q^{j-1})}, 
\allowdisplaybreaks \\
    \label{eq:q-krawtchouk-inv-measure}
   \mu_{n,k}(\lambda|\alpha, q,q^{-1}) & = \frac{\alpha^{\abs{\lambda}} q^{\Abs{\lambda}}\dim_{q}\left(V_{\GL_{n}}(\lambda)\right)\cdot
    q^{-\Abs{\bar{\lambda}'}}\dim_{1/q}\left(V_{\GL_{k}}(\bar{\lambda}')\right)}
  {\prod_{i=1}^{n}\prod_{j=1}^{k}(q^{i-1}+q^{1-j})},
\end{align}
\end{subequations}
for each specialization respectively, where $\Abs{\lambda} = \sum_{i=1}^{n}(i-1)\lambda_i$ and $q$-dimension of the irreducible $\GL_{n}$ representation is defined as the
principal gradation (see \cite[\S10.10]{kac90}) that is the weighted
sum of the dimensions of weight subspaces:
\begin{equation}
  \label{eq:q-dim-def}
  \dim_{q}\left(V_{\GL_{n}}(\lambda)\right)=\sum_{(u_{1},\dotsc,u_{n-1})\in
    \mathbb{Z}^{n-1}_{\geq 0}}q^{\sum_{i=1}^{n-1}u_{i}}\dim
  V(\lambda)_{\lambda-\sum_{i=1}^{n-1}u_{i}\alpha_{i}}, 
\end{equation}
where $\alpha_{1},\dotsc,\alpha_{n-1}$ are the simple roots of $\GL_{n}$ and we
identify the diagram $\lambda$ with the dominant $\GL_{n}$ weight $\lambda$ in the usual way.

We use the standard notation for (combinatorial) $q$-analogs:
\begin{equation}
[n]_q = \frac{q^n-1}{q-1} = 1 + \cdots + q^{n-1},
\qquad
[n]_q! = [1]_q \dotsm [n]_q,
\qquad
\qbinom{n}{k}{q} = \frac{[n]_q!}{[k]_q! [n-k]_q!}.
\end{equation}
Using $q$-analogues of the Weyl dimension formula, the Lindstr\"om--Gessel--Viennot lemma, and Dodgson condensation, the measure \eqref{eq:q-krawtchouk-measure} is then rewritten explicitly as
\begin{equation}
  \label{eq:q-krawtchouk-explicit}
  \mu_{n,k}(\{a_i\}| \alpha, q, q) = C^{+}_{n,k,q}\prod_{i<j}(q^{-a_i-n+\frac{1}{2}}-q^{-a_j-n+\frac{1}{2}})^2\prod_{i=1}^{n} W^{+}_{n,k}(a_i+n-1/2| \alpha),
\end{equation}
where
\begin{equation}
  \label{eq:q-krawtchouk-weight}
  W^{+}_{n,k}(a|\alpha) = \alpha^{a} q^{\binom{a}{2}+a(n-k)}\qbinom{n+k-1}{a}{q}
\end{equation}
and
\begin{equation}
  \label{eq:q-krawtchouk-const}
  C^{+}_{n,k,q} =
  \dfrac{\alpha^{-\frac{n(n-1)}{2}}q^{\frac{k n}{2}(n+k-2)}}
  {\prod_{i=1}^{n}\prod_{j=1}^{k}(q^{i-1}+q^{j-1})}
  \prod_{i=1}^{n}\dfrac{[k+i-1]_q!}
  {\left[i-1\right]_q! [n+k-1]_q!}\frac{1}{(1-q)^{\frac{n(n-1)}{2}}}.
\end{equation}
The measure \eqref{eq:q-krawtchouk-inv-measure} similarly takes the form

\begin{equation}
  \label{eq:q-krawtchouk-inv-explicit}
  \mu_{n,k}(\lambda|\alpha,q,q^{-1})=C^{-}_{n,k,q}\prod_{i<j}(q^{-a_i-n+\frac{1}{2}}-q^{-a_j-n+\frac{1}{2}})^2\prod_{i=1}^{n}W^{-}_{n,k}(a_i+n-1/2 | \alpha), 
\end{equation}
where
\begin{equation}
  \label{eq:q-krawtchouk-inv-weight}
  W^{-}_{n,k}(a | \alpha) = \alpha^{a} q^{\binom{a}{2}+a(n-1)}\qbinom{n+k-1}{a}{q}
\end{equation}
and
\begin{equation}
  \label{eq:q-krawtchouk-inv-const}
  C^{-}_{n,k,q}= 
  \dfrac{\alpha^{-\frac{n(n-1)}{2}}q^{\frac{n}{2}(n-1)(n+2k-2)}}
  {\prod_{i=1}^{n}\prod_{j=1}^{k}(q^{i-1}+q^{1-j})}
  \prod_{i=1}^{n}\dfrac{\left[k+i-1\right]_{q}!}
  {\left[i-1\right]_q! \left[n+k-1\right]_{q}!}\frac{1}{(1-q)^{\frac{n(n-1)}{2}}}.
\end{equation}

The $q$-Krawtchouk polynomials $K_{l}^{q}(q^{-a};p,N;q)$ are defined on the multiplicative lattice $\{q^{-a}\}_{a=0}^{N}$ and are orthogonal with respect to the weight $\qbinom{N}{a}{q} p^{-a} q^{\binom{a}{2}-aN}$~\cite{koekoek2010hypergeometric}.
Therefore the  weights  $W^{+}_{n,k}(a|\alpha)$ and $W^{-}_{n,k}(a|\alpha)$ are weights of $q$-Krawtchouk polynomials $K_{l}^{q}(q^{-a}; q^{1-2n}/\alpha, n+k-1; q)$ and $K_{l}^{q}(q^{-a};q^{2-2n-k}/\alpha, n+k-1; q)$.
As described in~\cite{nazarov2022skew}, $q$-difference equations for
$q$-Krawtchouk polynomials give a way to derive the limit shapes in
the regime when $n,k\to\infty$, while $q\to 1$ in such a way that
$c=\lim\frac{k}{n}$ and $q = 1-\frac{\gamma}{n}$ with $c,\gamma$ being
constants. Moreover, recurrence relations for the orthogonal
polynomials can be used to study global fluctuations around the limit
shape and to prove the convergence to the limit shape.

Below we present the derivation of the same limit shapes and study the
local fluctuations in this regime from the point of view of our
general framework.
Generalizing the problem a bit, we take $x_i = \alpha q^{i-1}$, $y_j = t^{j-1}$ and take the limits $n,k\to\infty$ and $q,t\to 1$ such that $q = 1-\frac{\gamma}{n}$ and $t = 1-\frac{\delta}{n}$.
We will assume $\gamma, \delta \neq 0$ and $\alpha > 0$.

As such, we instead set $x_{i}=\alpha e^{-\gamma\frac{i-1}{n}}$, $y_{j} = e^{-\delta\frac{j-1}{n}}$, so that we need to solve Equation~\eqref{eq:zdz-S-eq-zero} for $f(s) = \alpha e^{-\gamma s}, g(s) = e^{-\delta c s}$:
\begin{equation}
  \label{eq:zdz-S-eq-zero-principal}
  t = \frac{1}{\gamma} \ln \left(\frac{1 - \alpha e^{-\gamma} z}{1 -
      \alpha z}\right) + \frac{1}{\delta} \ln \left(\frac{z + 1}{z + e^{-\delta c}}\right).
\end{equation}
We can exponentiate~\eqref{eq:zdz-S-eq-zero-principal} to get
\begin{equation}
\label{eq:zdz-S-biprinicipal-roots}
e^{\gamma t} = \frac{1 - \alpha e^{-\gamma}z}{1-\alpha z} \cdot \left( \frac{z+1}{z+ e^{-\delta c}} \right)^{\gamma/\delta},
\end{equation}
but we should check that its roots satisfy \eqref{eq:zdz-S-eq-zero-principal}. 
For this specialization, Equation~\eqref{eq:zdz-2-S-eq-zero} becomes
\begin{equation}
\label{eq:zdz-2-exp}
0 = \frac{(A z^2 + B z + C ) z}{(z +  e^{-c \delta}) (1 + z)
  (e^{\gamma} - \alpha z) (\alpha z - 1)}
\end{equation}
with
\begin{subequations}
\begin{align}
\gamma \delta A & = \alpha \delta(1 - e^{\gamma})+\alpha^{2} \gamma (1-e^{-c\delta}),
\\ \gamma \delta B & =
                     \alpha\delta(1-e^{\gamma}+e^{-c\delta}-e^{c\delta+\gamma})+\alpha\gamma
                     (e^{-c\delta}-1-e^{\gamma}+e^{-c\delta+\gamma}),
\\ \gamma \delta C & = \gamma (1 - e^{-c \delta}) e^{\gamma} + \alpha \delta (1-e^{\gamma} ) e^{-c \delta}).
\end{align}
\end{subequations}
Hence, the roots of~\eqref{eq:zdz-2-exp} are given by
\begin{equation}
\label{eq:zdz-2-exp-roots}
z_0 = 0,
\qquad\qquad
z_{\pm} = \frac{-B \pm \sqrt{B^2 - 4AC}}{2A}.
\end{equation}
Hence, we can compute our endpoints $t_{\pm}$ by substituting $z_{\pm}$ into~\eqref{eq:zdz-S-eq-zero-principal}.

Next, we examine~\eqref{eq:zdz-S-biprinicipal-roots} in more detail.
We see that if $\gamma / \delta \in \ZZ$, then Equation~\eqref{eq:zdz-S-biprinicipal-roots} becomes a polynomial equation in $z$.
Now if we take $\gamma = \delta$, then for real $t$ we have
\begin{equation}
t = \frac{1}{\gamma} \ln \left( \frac{(1 + z)(1 -\alpha e^{-\gamma} z)}{(1 -\alpha z)(z + e^{-\gamma c})}\right)
\qquad \Longleftrightarrow \qquad
e^{\gamma t} = \frac{(1 - \alpha e^{-\gamma}z)(z+1)}{(1-\alpha z)(z+ e^{-\gamma c})},
\end{equation}
which is a quadratic equation for $z$ in terms of $e^{\gamma t}$:
\begin{equation}
\alpha (e^{\gamma t} - e^{-\gamma}) z^2 + (\alpha (e^{\gamma (t-c)} - e^{-\gamma}) - e^{\gamma t} + 1) z +  (1 - e^{\gamma (t - c)}) = 0.
\end{equation}
The solutions are
\begin{equation}
  \begin{aligned}
  z_{1,2}& =\frac{-\alpha (e^{\gamma (t-c)} - e^{-\gamma}) + e^{\gamma t} - 1}{2 \alpha(e^{\gamma t} - e^{-\gamma})} \\
  & \hspace{10pt} \pm\frac{\sqrt{(\alpha (e^{\gamma (t-c)} -
      e^{-\gamma}) - e^{\gamma t} + 1)^{2} - 4 \cdot \alpha (e^{\gamma t} - e^{-\gamma}) \cdot (1 - e^{\gamma (t - c)})}}{2 \alpha(e^{\gamma t} - e^{-\gamma})}.
  \end{aligned}
\end{equation}
To determine the domain $[t_{-},t_{+}]$ we need to solve~\eqref{eq:zdz-2-S-eq-zero}, which in this case reads
\begin{equation}
  \label{eq:zdz-2-S-eq-zero-principal}
  \frac{1}{\gamma} \left( \frac{1}{1-\alpha z} - \frac{1}{1-\alpha e^{-\gamma}z} + \frac{z}{z+1} - \frac{z}{z+ e^{-\gamma c}} \right) = 0.
\end{equation}
Transforming to the common denominator we obtain
\begin{equation}
  \label{eq:zdz-2-S-eq-zero-principal-common-denominator}
  \frac{-z}{\gamma} \cdot
  \frac{\alpha (\alpha-e^{\gamma(c+1)} + (\alpha + 1) e^{\gamma c}) z^2 + 2 \alpha (1-e^{\gamma (c+1)}) z + \alpha (1-e^{\gamma})+e^{\gamma (c+1)} - e^{\gamma}}
    {(1-\alpha z)(1+z)(\alpha z-e^{\gamma})(1+e^{\gamma c}z)} = 0,
\end{equation}
and the roots are
\begin{equation}
  \label{eq:zdz-2-S-eq-zero-principal-roots}
  z_{\pm} = \frac{\alpha (1 - e^{\gamma (c+1)}) \mp \sqrt{\alpha(1 + \alpha)} \sqrt{\left(e^{\gamma}-1\right) \left(e^{\gamma c}-1\right) \left(e^{\gamma (c+1)}+\alpha\right)}}
  {\alpha(e^{\gamma(c+1)} -\alpha- (\alpha + 1) e^{\gamma c})},
\end{equation}
so

\begin{equation}
  \label{eq:x-pm-principal}
  t_{\pm}=\frac{1}{\gamma}\ln\left(\frac{e^{\gamma(c-1)}\left(2\alpha-\alpha e^{\gamma}+e^{\gamma(c+1)}+2\alpha e^{\gamma(c+1)}\pm 2\sqrt{\alpha((1+\alpha)e^{\gamma}-1)((e^{\gamma c}-1) e^{\gamma(c+1)}-\alpha)}\right)}{(\alpha+e^{\gamma c})^{2}}\right).
\end{equation}
Finally, the limit density is given by the formula~\eqref{eq:rho-arg-z}, that reads
\begin{equation}
  \label{eq:rho-principal}
  \rho(t)= \frac{1}{\pi} \arccos\left( \frac{\alpha e^{c\gamma} - e^{(c+1)\gamma} - \alpha e^{\gamma(t+1)} + e^{\gamma(c+t+1)}}{2 \sqrt{\alpha (e^{\gamma (c+1)} - e^{\gamma (t+1)}) (e^{\gamma (c+t+1)} - e^{\gamma c})}}
\right).
\end{equation}
If we take $\alpha=1$ and shit  the parameter $t\to t+1$ this formula recovers the limit shape presented in \cite{nazarov2022skew}.

Taking $x_{i}=\alpha e^{-\gamma\frac{i-1}{n}}$, $y_{j} = e^{\gamma\frac{j-1}{n}}$, so
that we need to solve Equation~\eqref{eq:zdz-S-eq-zero} for $f(s) = \alpha e^{-\gamma s}, g(s) = e^{\gamma c s}$, we similarly obtain the limit shape.
Solving Equation~\eqref{eq:zdz-S-eq-zero} we obtain the roots

\begin{equation}
  \label{eq:z12-principal-inv-principal}
  z_{1,2}=\frac{\alpha e^{\gamma c}-(\alpha-1)e^{\gamma(t+1)}-e^{\gamma}}{2\alpha(e^{\gamma(t+1)}-1)}\pm
  \frac{\sqrt{\left((\alpha-1)e^{\gamma(t+1)}-\alpha e^{\gamma c}+e^{\gamma}\right)^{2} -4\alpha e^{\gamma}(e^{\gamma(t+1)}-1)(e^{\gamma c}-e^{\gamma t})}}{2\alpha(e^{\gamma(t+1)}-1)}.
\end{equation}
We obtain the support $[t_{-},t_{+}]$ by solving~\eqref{eq:zdz-2-S-eq-zero}:
\begin{equation}
  \label{eq:xpm-principal-inv-principal}
  \begin{aligned}
   t_{\pm} & = - 1 - \frac{2}{\gamma}\ln(1+\alpha)
   \\ & \hspace{20pt} + \frac{1}{\gamma}\ln\left(e^{\gamma} +2\alpha -\alpha e^{\gamma} +\alpha (2 e^{\gamma}+\alpha -1) e^{\gamma c} \pm 2\sqrt{\alpha(e^{\gamma}-1)(\alpha+e^{\gamma})(e^{\gamma c}-1)(1+\alpha e^{\gamma c})}\right).
 \end{aligned}
\end{equation}
By using~\eqref{eq:rho-arg-z} we obtain the density
\begin{equation}
  \label{eq:rho-principal-inverse}
  \rho(t)=\frac{1}{\pi}\arccos\left(\frac{\alpha e^{\gamma c}-(\alpha-1)e^{\gamma(t+1)}-e^{\gamma}}{2\sqrt{\left((\alpha-1)e^{\gamma(t+1)}-\alpha e^{\gamma c}+e^{\gamma}\right)^{2} -4\alpha e^{\gamma}(e^{\gamma(t+1)}-1)(e^{\gamma c}-e^{\gamma t})}}\right).
\end{equation}
Taking the limit $\alpha \to 1$, the density~\eqref{eq:rho-principal-inverse} simplifies to 
\begin{equation}
  \label{eq:rho-principal-inverse-alpha1}
  \rho(t)=\frac{1}{\pi}\arccos\left(\mathrm{sign}(-\gamma)\times\frac{e^{\gamma-\frac{\gamma t}{2}}}{2}\dfrac{1-e^{\gamma(c-1)}}{\sqrt{(1-e^{\gamma t})(1-e^{\gamma(c+1-t)})}}\right),
\end{equation}
which coincides with the result in \cite{nazarov2022skew} after the shift of $t$.

Limit shapes for various values of $\gamma$ are presented in Fig.~\ref{fig:limit-shape-plots}.

\begin{figure}
  \centering
  \includegraphics[width=0.47\linewidth]{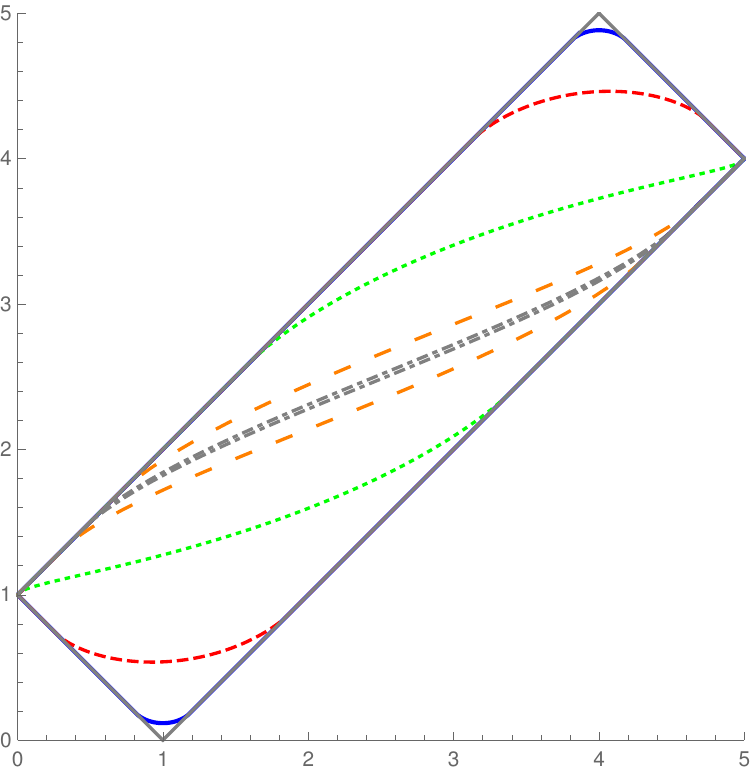}
  \quad
  \includegraphics[width=0.47\linewidth]{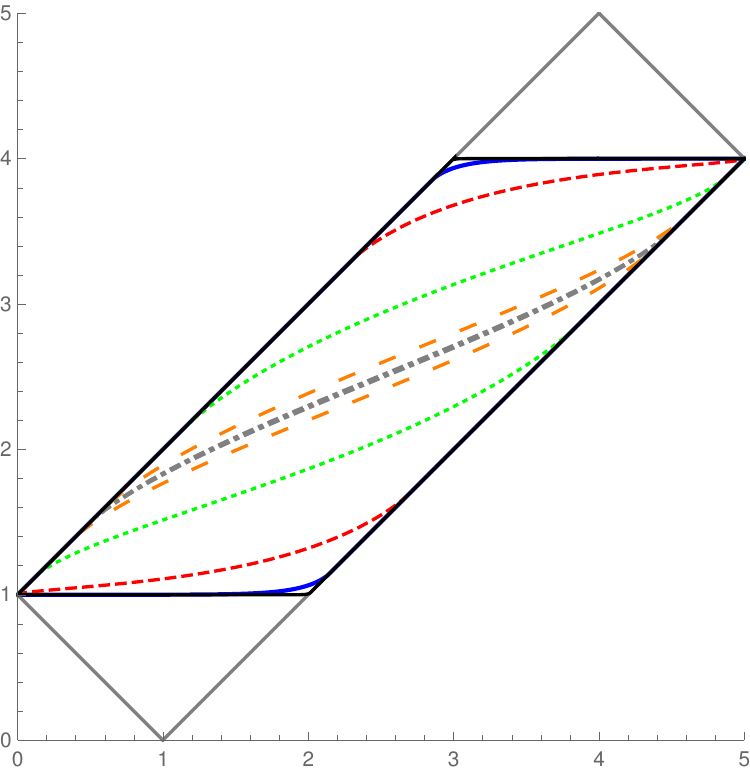}  
  \caption{Plots of the limit shapes (shifted $t \mapsto t+1$) for Young diagrams corresponding to the densities \eqref{eq:rho-principal} (left) and \eqref{eq:rho-principal-inverse-alpha1} (right) for $c=4$ and the values of $\gamma$ (bottom to top): $-10$ (solid blue), $-2$ (dashed red), $-0.5$ (dotted green), $-0.1$ (sparsely dashed orange),  $-0.01$ (dot-dashed gray), $0.01$ (dot-dashed gray), $0.1$ (sparsely dashed orange), $0.5$ (dotted green), $2$ (dashed red), $10$ (solid blue).
    Solid black lines on the right panel correspond to $\gamma=\pm\infty$ ($q=\mathrm{const}$).  } 
  \label{fig:limit-shape-plots}
\end{figure}

Next, we consider the edge asymptotics governed by the Airy kernel.
To make formulas less cumbersome we take $\alpha=1$, but we can still explicitly do the computations below for generic $\alpha$.
First for $f(s)=e^{-\gamma s}, g(s)=e^{-\gamma c s}$ we have critical value
\begin{equation}
  \zcrit = \frac{1-e^{(c+1)\gamma}+\sqrt{2}
    \sqrt{\left(e^{\gamma}-1\right) 
      \left(e^{c \gamma}-1\right) \left(e^{(c+1)\gamma}+1\right)}}
  {e^{(c+1)\gamma}-2 e^{c \gamma}+1}.
\end{equation}
Denoting by $\Delta=\sqrt{2}
    \sqrt{\left(e^{\gamma}-1\right) 
      \left(e^{c \gamma}-1\right) \left(e^{(c+1)\gamma}+1\right)}$, for the $\left.(z\partial_{z})^{3}S(z)\right|_{z=\zcrit}$ we have
\begin{equation}
\label{eq:zdz3-exp-roots}
\begin{aligned}
  & \left.(z\partial_{z})^{3}S(z)\right|_{z=\zcrit}=-\left(-2 e^{c \gamma}+e^{c \gamma+\gamma}+1\right)^2 \\
  &\times \left[\frac{ \left(\Delta
   \left(2 e^{c \gamma}-3 e^{2 (c+1) \gamma}+2 e^{(c+2)
   \gamma}-2 e^{c \gamma+\gamma}+2 e^{2 c \gamma+\gamma}+2
   e^\gamma-3\right)\right)}{\frac{\gamma}{2}\Delta^{2} \left(-2 \Delta+2
   e^{c \gamma}+e^{(c+2) \gamma}-e^{c \gamma+\gamma}+e^\gamma-3\right)
   \left(e^{c \gamma} \left(1-2 \Delta\right)-2 e^{2 c \gamma}-e^{c
       \gamma+\gamma}+3 e^{2 c \gamma+\gamma}-1\right)}\right.\\
&\hspace{20pt}+\left.\frac{+4 e^{c \gamma}+4 e^{2 (c+1) \gamma}+4
   e^{3 (c+1) \gamma}-4 e^{(2 c+3) \gamma}-4 e^{(3 c+2)
     \gamma}-4 e^{c \gamma+\gamma}+4 e^\gamma-4}{\frac{\gamma}{2}\Delta^{2} \left(-2 \Delta+2
   e^{c \gamma}+e^{(c+2) \gamma}-e^{c \gamma+\gamma}+e^\gamma-3\right)
   \left(e^{c \gamma} \left(1-2 \Delta\right)-2 e^{2 c \gamma}-e^{c
       \gamma+\gamma}+3 e^{2 c \gamma+\gamma}-1\right)}\right].
\end{aligned}
\end{equation}
By Equation~\eqref{eq:normalization-for-airy}, the normalization constant $\sigma$ for the Airy kernel is given by~\eqref{eq:zdz3-exp-roots} raised to the power $-\frac{1}{3}$ and multiplied by $2^{\frac{1}{3}}$. 
We present a plot of $\sigma$ as a function of $c$ for various values of $\gamma$ in Fig.~\ref{fig:tracy-widom-width-for-gamma}.
Note that for all values of $\gamma$, we have a divergence at various values of $c \leq 1$.
These degenerate cases are described by the discrete Hermite kernel, as stated in Theorem \ref{thm:near-corner-asymptotics}, discussed in Section \ref{sec:asymptotics-near-the-corner} and illustrated in the example of constant specialization in Section \ref{sec:const-funct-krawtch} for $c \to \alpha$.
In Fig.~\ref{fig:discrete-fluctuations-for-gamma-05}(left), we present a histogram of first row lengths and on the right is one randomly sampled diagram.
We see that this discrete distribution appears when $n t_{+} = k$, which means that the boundary of the support of $\rho(t)$ is exactly in the corner of the $n \times k$ box and first row can not fluctuate freely. 

\begin{figure}
  \centering
  \includegraphics[width=11cm]{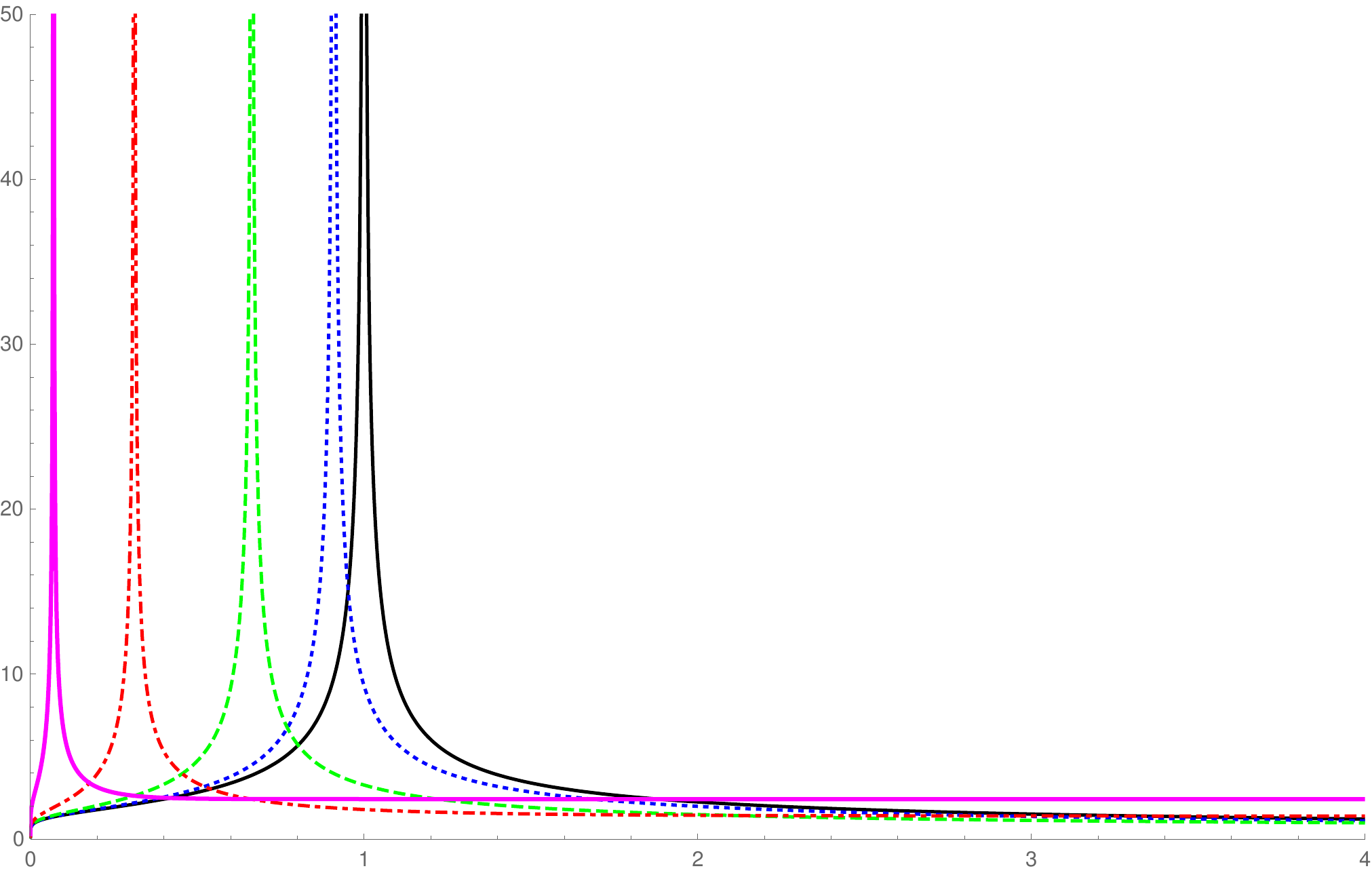}
  \caption{Plots of $\sigma$ as a function of $c$ for $\gamma=0$ (solid black), $\gamma=0.1$ (dotted blue), $\gamma=0.5$ (dashed green), $\gamma=2$ (dot-dashed red), $\gamma=10$ (solid magenta).}
  \label{fig:tracy-widom-width-for-gamma}
\end{figure}

\begin{figure}
  \centering
  \includegraphics[width=8cm]{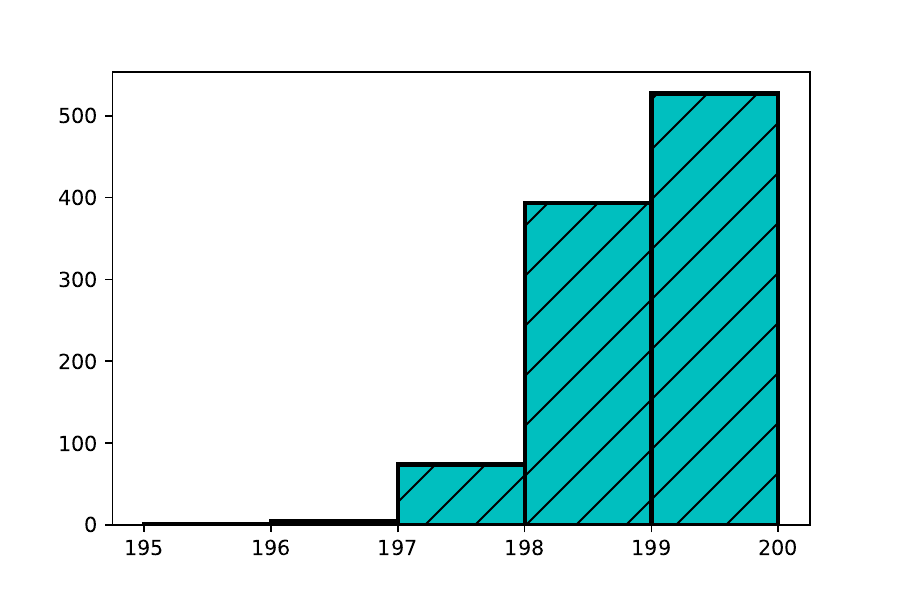}
  \includegraphics[width=8cm]{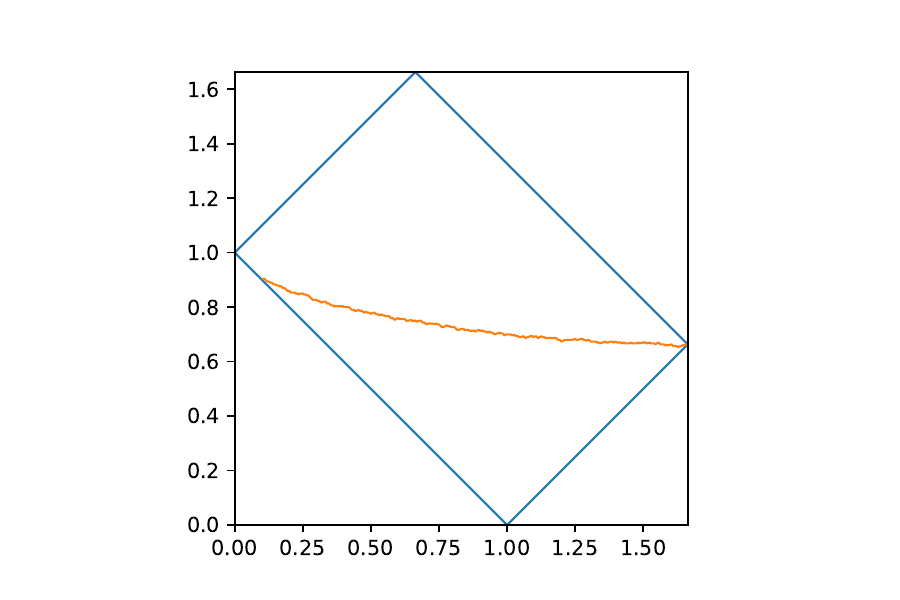}
  \caption{{\it Left:} Histogram of first row length $\lambda_{1}$ for 1000 samples with $n=300$, $k=199$, $\gamma=0.5$. {\it Right:} Sample diagram for  $n=300$, $k=199$, $\gamma=0.5$.}
  \label{fig:discrete-fluctuations-for-gamma-05}
\end{figure}

Similarly, for $f(s)=e^{-\gamma s}$, $g(s)=e^{\gamma c s}$ we use \eqref{eq:z12-principal-inv-principal} to obtain
\begin{equation}
  \zcrit = \frac{e^{\gamma}-e^{\gamma c}}{1-e^{\gamma(c+1)}-\sqrt{(e^{2\gamma}-1)(e^{2\gamma c}-1)}}. 
\end{equation}
Denoting the square root by $\Delta=\sqrt{(e^{2\gamma}-1)(e^{2\gamma c}-1)}$ and computing $(z\partial_{z})^{3}S(z)$ using \eqref{eq:zdz-3-S}, we obtain
\begin{equation}
\label{eq:zdz3-exp-dual-roots}
\begin{aligned}
\left.(z\partial_{z})^{3}S(z)\right|_{z=\zcrit} =\frac{1}{\gamma}\left(e^\gamma-e^{ \gamma c}\right)^2
&\left(\frac{(e^{\gamma(c+1)}-1)(e^{\gamma(c+1)}-1+\Delta)}{\left(-1+e^{2\gamma}+e^{2\gamma c}-e^{2\gamma(c+1)}+\Delta(1-e^{\gamma(c+1)})\right)^{2}}\right.-\\
&\hspace{20pt}\left.-\frac{1}{\left(e^{\gamma c}- e^{\gamma(c+2)}-e^\gamma\Delta\right)^2}
  -\frac{1}{\left(e^{\gamma}- e^{\gamma (c+2)}-e^{\gamma c}\Delta\right)^2}\right).
\end{aligned}
\end{equation}
By formula \eqref{eq:normalization-for-airy}, the normalization constant $\sigma$ is given by~\eqref{eq:zdz3-exp-dual-roots} raised to the power $-\frac{1}{3}$ and multiplied by $2^{\frac{1}{3}}$.
Here $c=1$ is the case when we have the discrete Hermite kernel for all values of $\gamma$, as demonstrated in Fig.~\ref{fig:discrete-fluctuations-for-principal-inverse-gamma-05}(left), whereas one randomly sampled diagram is given in Fig.~\ref{fig:discrete-fluctuations-for-principal-inverse-gamma-05}(right).

\begin{figure}
  \begin{center}
  \includegraphics[height=5cm]{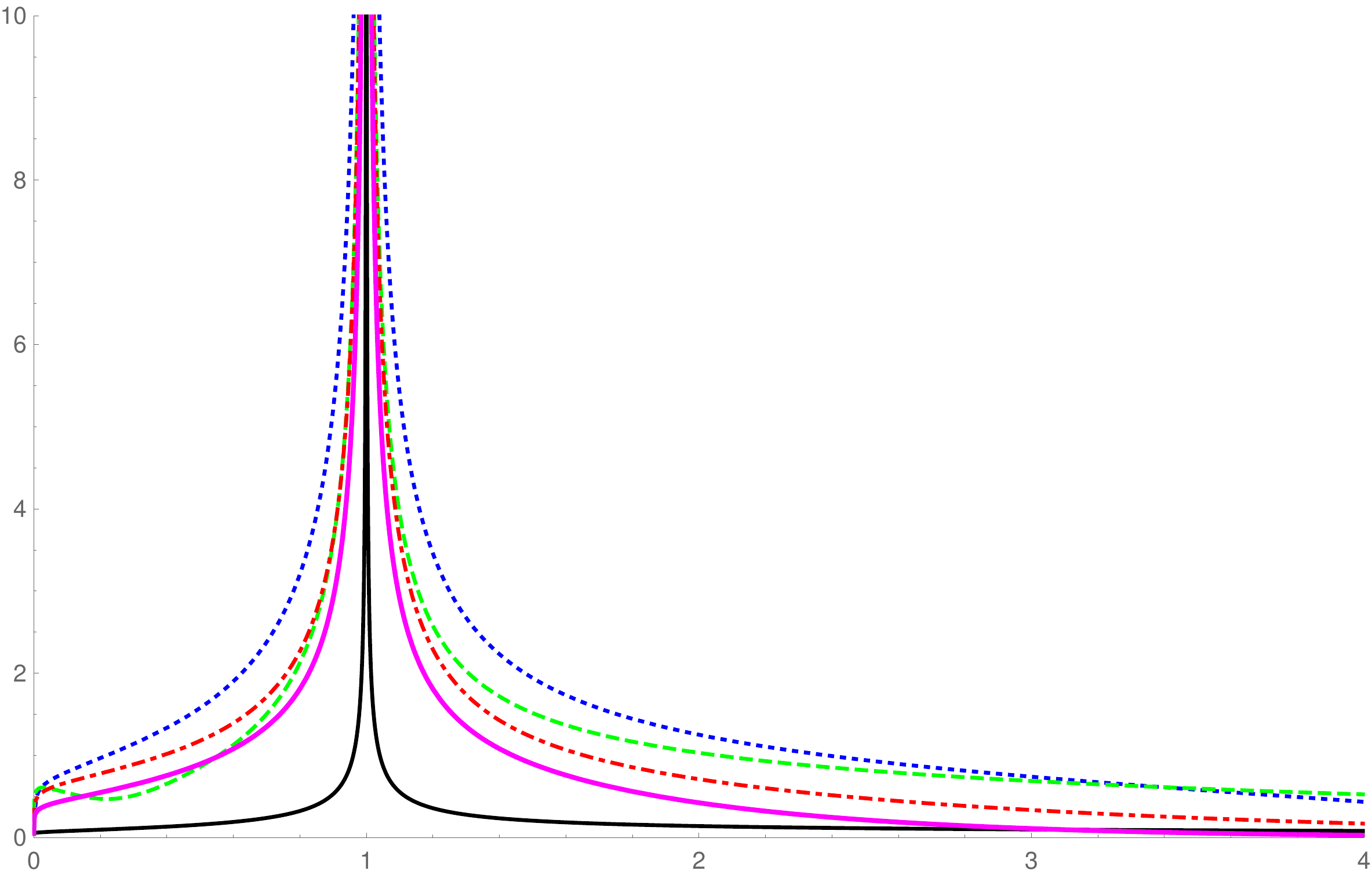}
  \includegraphics[width=8cm]{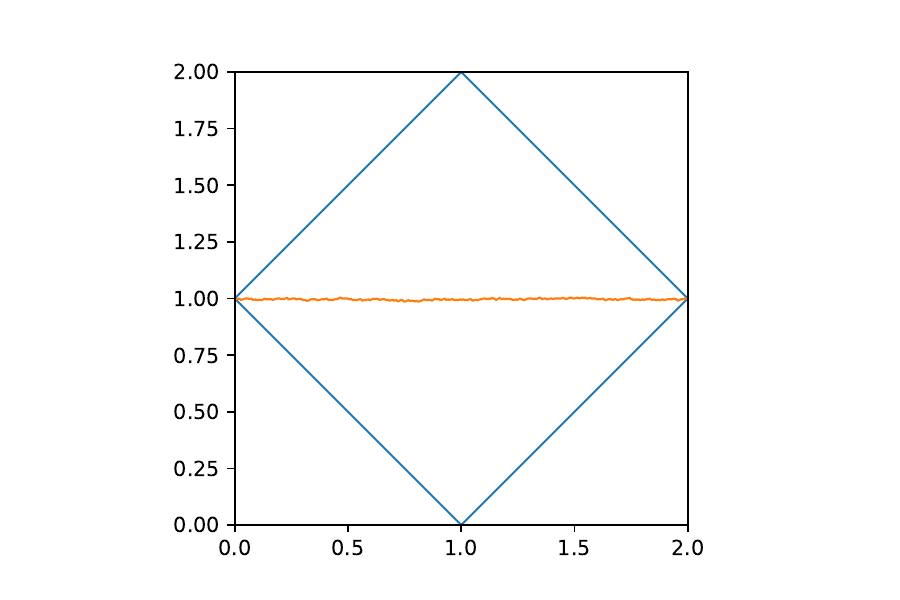}
  \end{center}
  \caption{{\it Left:} Plots of $\sigma$ as a function of $c$ for $\gamma=0$ (solid black), $\gamma=0.1$ (dotted blue), $\gamma=0.5$ (dashed green), $\gamma=2$ (dot-dashed red), $\gamma=4$ (solid magenta). {\it Right:} Sample diagram for  $n=300$, $k=300$, $\gamma=0.5$.}
  \label{fig:discrete-fluctuations-for-principal-inverse-gamma-05}
\end{figure}

\subsection{Uniform convergence}
\label{sec:uniform-convergence}

In this subsection, we will give another proof of Corollary~\ref{cor:uniform-convergence} for the $q$-Krawtchouk ensemble by using different method, emphasizing the exponential decay of the measure when the diagrams are far from the limit shape.
Starting with formulas \eqref{eq:q-krawtchouk-explicit}, \eqref{eq:q-krawtchouk-inv-explicit} for the probability measure, introducing the variables $s_{i}=q^{-a_{i}}$, and recalling that $q=e^{-\gamma/n}$, we rewrite the measure in the exponential form as
\begin{equation}
  \label{eq:q-measure-exp}
  \mu_{n,k}(\lambda|q,q^{\pm 1})=\exp\left[\sum_{i\neq j} \ln|s_{i}-s_{j}| +\sum_{i=1}^{n} \ln W^{\pm}_{n,k}(s_{i}|\alpha)\right].
\end{equation}

We consider the limit $n,k \to\infty,\; q \to 1 $ such that $q=1-\frac{\gamma}{n}$,
$\frac{k}{n} = c + \mathcal{O}\left(\frac{1}{n}\right)$
and substitute the leading approximation for $q$-factorials 
\begin{equation}
  \label{eq:stirling-q-factorials}
 \ln [a]_{q}!=\sum_{i=1}^{a}\ln\frac{1-q^{i}}{1-q}\approx n\int_{0}^{a/n} \ln\left(1-e^{-\gamma y}\right)\dy +a\ln (n/\gamma)=\frac{n}{\gamma}\left[\Li_{2}\left(e^{-\gamma a/n}\right)-\Li_{2}(1)\right]+a\ln (n/\gamma),
\end{equation}
into  $\ln W^{\pm}_{n,k}(a|\alpha)$, written as
\begin{equation}
  \label{eq:q-krawtchouk-potential}
  \ln W^{\pm}_{n,k}(a|\alpha)=\ln [n+k-1]_{q}!-\ln[a]_{q}!-\ln [n+k-1-a]_{q}!-\frac{\gamma a(a-1)}{2n}+a\ln \alpha -a\gamma+\frac{a\gamma}{n}\left(\frac{k+1}{2}\pm\frac{k-1}{2}\right),
\end{equation}
to obtain, after the substitution $a=\frac{n}{\gamma}\ln s$:
\begin{multline}
  \label{eq:q-krawtchouk-potential-in-s}
  \ln W^{\pm}_{n,k}(s|\alpha)\approx\\\approx\frac{n}{\gamma}\left[\frac{\pi^{2}}{6}+\Li_{2}\left(e^{-\gamma(c+1)}\right)-\Li_{2}\left(s^{-1}\right)-\Li_{2}\left(e^{-\gamma(c+1)}s\right)-\frac{(\ln s)^{2}}{2}+\left(\gamma\left(\frac{c}{2}-1\pm \frac{c}{2}\right)+\ln\alpha\right)\ln s\right].
\end{multline}
The derivative of $\ln W^{\pm}_{n,k}(s|\alpha)$ can be used to obtain the limit shape by solving the variational problem, as discussed below:
\begin{align}
\frac{\mathrm{d}}{\ds}\ln W^{\pm}_{n,k}(s|\alpha) & \approx
  \frac{n}{\gamma s}\left[-\ln \left(1-1/s\right)+\ln\left(1-e^{-\gamma(c+1)}s\right)-\ln s+\gamma\left(\frac{c}{2}-1\pm \frac{c}{2}\right)\right] \nonumber
  \\ & = \frac{n}{\gamma s} \left[\ln\left(e^{\gamma(c+1)}-s\right)-\ln \left(s-1\right)+\gamma\left(-\frac{c}{2}-2\pm \frac{c}{2}\right)\right].
\end{align}

If we consider the upper boundary of the rotated and scaled diagram as
the function $f_{n}(s)$, we have $f_{n}'(a_{i}/n)=\pm 1$ and
$\frac{1-f_{n}'(a_{i}/n)}{2}=1$ if there is a particle at the midpoint
of the interval $a_{i}$ as shown in Fig.~\ref{fig:maya_diagram} or $0$
otherwise. Therefore we can consider the sum
$\sum_{i\neq j} \ln|s_{i}-s_{j}|$ in \eqref{eq:q-measure-exp}  as an approximation to an integral
\begin{equation}
  \label{eq:log-double-integral}
\sum_{i\neq j} \ln|s_{i}-s_{j}|\approx\frac{n^{2}}{\gamma^{2}}\int_{1}^{e^{\gamma(c+1)}} \frac{\dt}{t} \int_{1}^{e^{\gamma(c+1)}} \frac{\ds}{s}
\frac{1}{4}\left(1-f_{n}'\left(\frac{n}{\gamma}\ln s\right)\right)\left(1-f_{n}'\left(\frac{n}{\gamma}\ln t\right)\right)\ln
\abs{s-t}.
\end{equation}
If we substitute Taylor expansion of
$\ln\abs{s-t}$ at the points $s=e^{\gamma i/n}, t=e^{\gamma j/n}$ for
$i,j=0,\dotsc, n+k-1$, and use $\int_{\exp(\gamma i/n)}^{\exp(\gamma
  (i+1)/n)}\frac{\ds}{s}=\gamma/n$, we recover the double sum of the logarithms. 
Expanding the brackets in the integral \eqref{eq:log-double-integral} and considering $f_{n}$ as a
function of $s$, we obtain the quadratic
functional $Q[f_{n}]=B[f_{n},f_{n}]$ where
\begin{equation}
B[f_{1},f_{2}] = \int_{1}^{e^{\gamma(c+1)}} \dt \int_{1}^{e^{\gamma(c+1)}} \ds \,
\frac{1}{4} f'_{1}(s) f'_{2}(t) \ln\abs{s-t}^{-1} ,
\end{equation}
a linear term
\begin{equation}
  \label{eq:linear-term-in-Q}
L[f_{n}]=-\frac{1}{2}\int_{1}^{e^{\gamma(c+1)}} \ds \int_{1}^{e^{\gamma(c+1)}} \frac{\dt}{t} f'_{n}(s) \ln\abs{s-t},
\end{equation}
and a constant term $C_{n}=\frac{n^{2}}{\gamma^{2}}\int_{1}^{e^{\gamma(c+1)}}  \int_{1}^{e^{\gamma(c+1)}} \ln
\abs{s-t}\frac{\dt}{t}\frac{\ds}{s}$:
\begin{equation}
  \frac{n^{2}}{\gamma^{2}}\int_{1}^{e^{\gamma(c+1)}} \frac{\dt}{t} \int_{1}^{e^{\gamma(c+1)}} \frac{\ds}{s}
\frac{1}{4}\left(1-f_{n}'\left(\frac{n}{\gamma}\ln s\right)\right)\left(1-f_{n}'\left(\frac{n}{\gamma}\ln t\right)\right)\ln
\abs{s-t}=Q[f_{n}]+L[f_{n}]+C_{n}.
\end{equation}

Similarly to the sum $\sum_{i\neq j} \ln|s_{i}-s_{j}|$, the sum $\sum_{i=1}^{n}\ln W^{\pm}_{n,k}(s_{i}|\alpha)$ can be approximated by
an integral
\begin{equation}
  \label{eq:q-krawtchouk-weight-integral}
  \begin{aligned}
  &\sum_{i=1}^{n}\ln W^{\pm}_{n,k}(s_{i}|\alpha)\approx
  \frac{n^{2}}{\gamma^{2}}\int_{1}^{e^{\gamma(c+1)}} \frac{\ds}{2}\left(\frac{1}{s} - f'_{n}(s) \right)  \\& \times\left[\frac{\pi^{2}}{6}+\Li_{2}\left(e^{-\gamma(c+1)}\right)-\Li_{2}\left(s^{-1}\right)-\Li_{2}\left(e^{-\gamma(c+1)}s\right)-\frac{(\ln s)^{2}}{2}+\left(\gamma\left(\frac{c}{2}-1\pm \frac{c}{2}\right)+\ln\alpha\right)\ln s\right] .
\end{aligned}
\end{equation}
Integrating linear term \eqref{eq:linear-term-in-Q} over $t$, we obtain
\begin{equation}
  \label{eq:integral-of-logarithm}
  L[f_{n}]=-\frac{1}{2}\int_{1}^{e^{\gamma(c+1)}}\ds f_{n}'(s)\left[\Li_{2}\left(s^{-1}\right)+\Li_{2}\left(e^{-\gamma(c+1)}s\right)+\frac{(\ln s)^{2}}{2}-\frac{\pi^{2}}{3}+\frac{\gamma^{2}(c+1)^{2}}{2}\right]. 
\end{equation}
If we combine this expression with the terms linear in $f_{n}'(s)$ in~\eqref{eq:q-krawtchouk-weight-integral} and use $\int_{1}^{e^{\gamma(c+1)}} \ds f_{n}'(s) = f_{n}(c+1)-f_{n}(0)=c-1$ to get rid of the $s$-independent contributions, we get the final form of the linear term in the exponential:
\begin{equation}
  \label{eq:linear-term}
 L^{\pm}[f_{n}]=-\frac{1}{2}\int_{1}^{e^{\gamma(c+1)}}\ds f_{n}'(s)\left(\gamma\left(\frac{c}{2}-1\pm \frac{c}{2}\right)+\ln\alpha\right)\ln s
\end{equation}
Combining all $s$-independent contributions with $C_{n}$ we get the final normalization constant $\widetilde{C}_{n}$, so we can write the probability of the diagram $\lambda$ as an exponent of a functional $J[f_{n}] = Q[f_{n}]+L^{\pm}[f_{n}]+\widetilde{C}_{n}$ of the upper boundary $f_{n}$:
\begin{equation}
  \label{eq:measure-as-functional}
  \mu_{n,k}(\lambda|q,q^{\pm 1}) = \exp\left(-\frac{n^{2}}{\gamma^{2}}J[f_{n}]+\mathcal{O}(n\ln n)\right),
\end{equation}
where the correction term $\mathcal{O}(n\ln n)$ comes from the estimates of next orders in Stirling approximation and difference between sums and integrals.
(For more detailed computations of this sort, see \cite[Lemma 2]{NNP20}.)

Analysis of the functional $J[f_{n}]$ is similar to~\cite[Ch.~1]{romik2015surprising}. The first step is to show that the limit density $\rho(t)$ related to the limit shape $\Omega$ by the formula \eqref{eq:limit-shape-as-integral}, can be recovered as a solution to the limiting minimization problem
\begin{equation}
  \label{eq:var-problem}
  \iint \ds \, \dt \ln \abs{s-t}^{-1} \rho(s) \rho(t) + \int \ds \rho(s) \ln W^{\pm}(s|\alpha),
\end{equation}
where $\ln W^{\pm}(s|\alpha)=\lim_{n\to\infty} \frac{1}{n}\ln W^{\pm}_{n,k}(s|\alpha)$. 
By~\cite[Ch.~I, Thm.~1.3]{ST97} (see also~\cite[Thm.~ 2.1]{johansson98}) there exists a unique solution of this problem. 

In~\cite{nazarov2022skew}, we have proven the weak convergence of the measures $\mu_{n,k}(\lambda|q,q^{\pm 1})$ to the equilibrium measure $\rho(t) \, \dx$.
The minimizing property of the limit density $\rho(t)$ follows from this convergence. We were not able to find this kind of statement (for weights depending on $n$) in the existing sources, therefore we present the sketch of the proof below.

The key point of the proof is the following ``large deviation'' estimate (see~\cite[Lemma 6.67]{deift1999orthogonal}, \cite{johansson98} for the case $\ln W^{\pm}_{n,k}(s|\alpha) \equiv n\ln W^{\pm}(s|\alpha)$), with the proof essentially the same for the varying weights. Denote by $E$ the minimal value of~\eqref{eq:var-problem}, and for any $n\in\NN$, $k=\lfloor cn\rfloor$, $\eta > 0$ set
\begin{equation}
A_{n,\eta} = \left\{ \lambda\; \biggl|\; \sum_{i\neq j} \ln|s_{i}-s_{j}| +\sum_{i=1}^{n} \ln W^{\pm}_{n,\lfloor cn\rfloor}(s_{i}|\alpha) \le n^2 (E + \eta) \biggr.\right\}.
\end{equation}
The following estimate holds for the complement of $A_{n,\eta}$.
\begin{prop}\label{prop:large-dev}
For any number $a>0$ there exists $N\in\NN$, which depends on $\eta$ but not on $a$, such that
\begin{equation}
\mu_{n,\lfloor cn\rfloor}(A_{n,\eta+a}^c|q,q^{\pm 1}) \le e^{-an^2} \quad \text{ for all } n \ge N.
\end{equation}
\end{prop}

Denote by $\mu^{(1)}_n$ the first correlation measure for $\mu_{n,\lfloor cn\rfloor}$, restricted to the sets $A_{n, \eta=\frac1n}$ and normalized, \textit{i.e.}, for any compactly supported bounded function $h$ we have
\begin{equation*}
  \label{eq:first-corr-measure}
  \int h(s) \mathrm{d}\mu^{(1)}_n(s) := \left(n\; \mu_{n,\lfloor cn\rfloor}\left(A_{n,\eta=\frac{1}{n}}|q,q^{\pm 1}\right)\right)^{-1}\cdot\int_{A_{n,\eta=\frac{1}{n}}}  \sum_{i}h(s_{i}) \mathrm{d}\mu_{n,\lfloor cn\rfloor}(\lambda|q,q^{\pm 1}).
\end{equation*}
From the weak convergence of $\mu_{n,[cn]}(\lambda|q,q^{\pm 1})$, combined with the estimate from Proposition~\ref{prop:large-dev}, it follows that the correlation measures $\mu^{(1)}_n$ converge weakly to $\rho(t)dx$.
In addition, from the definition of $A_{n,\eta}$ it follows that $\mu^{(1)}_n$ should converge to the minimizer, hence the minimizer is given by the density $\rho(t)$ due to uniqueness. 

Alternatively, one can use Plemelj formula 
to solve a scalar Riemann--Hilbert problem to obtain the solution for the variational problem \eqref{eq:var-problem}.

As $\rho(t)$ is the solution to the variational problem, by~\cite[Ch.~I Rem.~1.5]{ST97} there exists a constant $\ell$ such that
\begin{subequations}
  \label{eq:variational-solution-ell}
\begin{align}
  &\int \dt \ln\abs{s-t}^{-1} \rho(t) +\ln W^{\pm}(s|\alpha)=\ell, \qquad \text{if } s\in\supp(\rho), \rho(s)<1;\\
  &\int \dt \ln\abs{s-t}^{-1} \rho(t) +\ln W^{\pm}(s|\alpha)\geq\ell, \qquad \text{if } s\not\in\supp(\rho);\\
  &\int \dt \ln\abs{s-t}^{-1} \rho(t) +\ln W^{\pm}(s|\alpha)\leq\ell, \qquad \text{if } \rho(s)=1.
\end{align}
\end{subequations}
We add the third case as there is an additional restriction $\rho(s)\le 1$. 

To prove the uniform convergence we need to demonstrate that the probability of a diagram $\lambda$ with the upper boundary $f_{n}$ can be estimated by the difference of $f_{n}$ and limit shape $\Omega$.

We note that $J[\Omega]$ is non-negative, since we can approximate $\Omega$ by an upper boundary of a diagram with any precision for $n$ large enough, but $J[f_{n}]$ is non-negative as it is equal to minus the logarithm of a probability.

Then we note that the functional $Q[f_{n}]$ is positive-definite on Lipschitz functions with compact support (see~\cite[Prop.~1.15]{romik2015surprising}).
Therefore we can use it to introduce the norm on the compactly supported Lipshitz functions $\Abs{\cdot}_{Q} = Q[\cdot]^{1/2}$.
By~\cite[Lemma~1.21]{romik2015surprising}, this norm can be used to bound the supremum norm for a compactly supported Lipshitz function $h$ from above:
\begin{equation}
  \label{eq:supremum-norm-and-Q-norm}
  \Abs{h}_{\infty}=\sup_{s} \abs{h} \leq C Q[h]^{1/4}.
\end{equation}
Now write $f_{n}$ as a sum of the limit shape $\Omega$ and a compactly
supported Lipshitz function $\delta f_{n}$:
$f_{n}(s)=\Omega(s)+\delta f_{n}(s)$. Then the probability of the
corresponding diagram $\lambda$ is given by
\begin{equation}
  \label{eq:probability-as-functional}
  \mu(\lambda|q,q^{\pm 1})=\frac{1}{Z_{n}}\exp\left(-\frac{n^{2}}{\gamma^{2}}
    \bigl( J[\Omega]+Q[\delta f_{n}]+2B[\Omega,\delta f_{n}]+L^{\pm}[\delta f_{n}] \bigr) \right)
\end{equation}
We estimate the linear term $2B[\Omega,\delta f_{n}]+L^{\pm}[\delta f_{n}]$ in the same way as in the proof of~\cite[Lemma~7]{NNP20} and use~\eqref{eq:variational-solution-ell} to demonstrate that it is non-negative. It is essentially another way to state that $\Omega$ is a minimizer. Therefore for a Young diagram $\lambda$ with the upper boundary $f_{n}$ such that
$Q[\delta f_{n}]=\varepsilon^{2}$ we have
\begin{equation}
\mu(\lambda|q,q^{\pm 1}) \leq C_{1}\exp\left(-\frac{n^{2}}{\gamma^{2}}\varepsilon^{2}+\mathcal{O}(n\ln n)\right).
\end{equation}
The number of Young diagrams inside the $n\times k$ box is equal to $\binom{n+k}{n}\le \exp\left(c_2 n\right)$. Therefore, the probability of large fluctuations from the limit shape goes to zero:
  \begin{equation}
    \label{eq:prob-sup-norm-to-zero}
    \mathbb{P}\left(\Abs{f_{n}-\Omega}_{Q} > \varepsilon \right) < C_{1} \exp\left(-\frac{n^{2}}{\gamma^{2}}\varepsilon^{2}+\mathcal{O}(n\ln n)\right)\exp\left(c_2 n\right) \xrightarrow[n\to\infty]{} 0.
  \end{equation}
Finally, we use~\eqref{eq:supremum-norm-and-Q-norm} to conclude that the convergence to the limit shape is uniform.

\subsection{Constant \texorpdfstring{$q$}{q}}
\label{sec:constant-q}

If $q=\mathrm{const}$, we still have $q$-Krawtchouk ensemble, but the limit shape degenerates. It can be derived by taking the limits in Equations~\eqref{eq:zdz-S-eq-zero-principal} and~\eqref{eq:zdz-S-biprinicipal-roots} as described below.
Assume that  $x_{i}=\alpha q^{i-1}$, $y_{j}=q^{b(j-1)}$, then we have $\gamma=-n\ln q$ and $\delta=-bn\ln q$. Denote the corresponding probability measure on Young diagrams by $\mu_{n,k}^{b}(\lambda|q)$. 
Substituting to \eqref{eq:zdz-S-eq-zero-principal}, we get for real $t$
\begin{equation}
  \label{eq:zdz-S-eq-zero-q-const}
  t=-\frac{1}{n\ln q}(\ln(1-\alpha q^{n}z)-\ln(1-\alpha z))-\frac{1}{nb\ln q}(\ln(z+1)-\ln(z+q^{nbc})).
\end{equation}
It is more convenient to exponentiate this equation:
\begin{equation}
  \label{eq:zdz-S-eq-zero-q-const-exponentiated}
  q^{-n(t+1)}=\frac{q^{-n}-\alpha z}{1- \alpha z}\left(\frac{z+ q^{nbc}}{z+1}\right)^{-\frac{1}{b}}
\end{equation}
Assuming $q>1$, $c>1$, $b<0$ we look for a pair of complex conjugate asymptotic solutions in the form $z_{\pm}=q^{un}e^{\pm\imi\varphi}$ with $u<0$.
Such pair is not real only if the leading term on the right-hand side is of the form  $-z^{1-\frac{1}{b}}=q^{un(1-\frac{1}{b})}\cdot e^{\imi\pi\pm\imi\varphi(1-\frac{1}{b})}$, therefore $\max(-1, bc) \le u \le 0$.
On the left hand side we have $q^{-n(t+1)}$, and comparing these terms we get $u=-\frac{b(t+1)}{b-1}$, $\varphi=\frac{b\pi}{b-1}$.
Thus we get constant density $\rho(t)=\frac{\varphi}{\pi}=\frac{b}{b-1}$ for $-1\le t\le \min(-\frac1b, c-1-bc)$. 

For $q>1$, $b>0$ the same argument shows that $z$ should grow at least as fast as $q^{nbc}$, therefore $u\ge bc>0$.
However in this case the absolute value of the right-hand side tends to $1$ and the limit shape degenerates. 

Another way to derive the support of the limit shape is to do the same substitution in Equation~\eqref{eq:zdz-2-S-eq-zero}, find the roots $z_{\pm}$, substitute to~\eqref{eq:zdz-S-eq-zero-q-const}, and only then take the limit $n\to\infty$, assuming $q>1$.
Thus Equation~\eqref{eq:zdz-2-S-eq-zero} takes the form
\begin{equation}
  \frac{z}{nb\ln q}\left(\frac{\alpha b}{\alpha z-1}-\frac{\alpha b}{\alpha z-q^{-n}}-\frac{1}{z+1}+\frac{1}{z+q^{nbc}}\right)=0.
\end{equation}

Taking the terms to the common denominator we obtain a cubic equation in the numerator with one of the roots $z_{0}=0$. As before, denote the other two roots by $z_\pm$. For the degenerate case $b>0$ the actual values for these roots are not important, as we have the degeneration $t_\pm = 0$ for $q<1$ and $t_\pm = c-1$ for $q>1$ for any finite nonzero values $z_\pm$. Thus the diagram is empty for $q<1, b>0$ and is completely filling the rectangle for $q>1, b>0$. For other cases we should take into account the asymptotic behaviour of the roots. For example, for the case $b<0$, $q>1$, $bc<-1$ we have $z_- = \alpha \frac{1-b}{\alpha+b}$, $z_+ = \bigO(q^{-n})$, and we recover the endpoints $t_-=-1$, $t_+=-\frac{1}{b}$. This allows us to conjecture the following asymptotic behavior of the diagrams.

\begin{conj}
\label{conj:constant_limit}
Consider the specializations $x_i = q^{i-1}$ and $y_j = q^{b(j-1)}$.
\begin{enumerate}
\item For $b < 0$, the limit density of the measure $\mu_{n,k}^b$ is $\rho(t) = b / (b-1)$.
Moreover, the limit shape $\Omega(t)$ is a straight line with slope $\frac{1+b}{1-b}$ with the left (resp.\ right) support endpoint being~$1$ for $q > 1$ (resp.~$c$ for $q < 1$); explicitly $\Omega(t)=\frac{1+b}{1-b}t+\frac{2}{1-b}$ (resp.~$\Omega(t)=\frac{1+b}{1-b}t-\frac{2bc}{1-b}$).

\item For $b > 0$, the limit shape for $\mu_{n,k}^b$ is the empty (resp.\ full) diagram for $q < 1$ (resp. $q > 1$).
\item The correlation kernel for $t$ in the support interval converges to the discrete sine kernel
  \begin{equation*}
  \lim_{n,k\to\infty}\mathcal{K}(nt+l,nt+l') = \begin{cases}
        \dfrac{\sin\bigl( \pi \rho(t) \cdot (l-l') \bigr)}{\pi(l-l')} & \text{if } l \neq l', \\
        \rho(t) & \text{if } l = l'.
      \end{cases}
    \end{equation*}
  \end{enumerate}
\end{conj}

\begin{figure}
\[
\begin{tikzpicture}[scale=1]
\draw[->,thin,gray] (-3,0) -- (3,0);
\draw[->,thin,gray] (0,-2) -- (0,2);
\draw[blue,thick] (0,0) circle (.7);
\draw[red,thick] (0,2) -- (0,-2);
\fill[black] (0,.7) circle (.07) node[anchor=south west] {$z_1$};
\fill[black] (0,-.7) circle (.07) node[anchor=north west] {$z_2$};
\draw[->,>=latex] (.1,-.6) -- (.1,-.15);
\draw[->,>=latex] (.1,.6) -- (.1,.15);
\fill[UQpurple] (-.1,-.1) rectangle (.1,.1);
\fill[HUgreen] (1.6-.1,-.1) rectangle (1.6+.1,.1) node[anchor=north,inner sep=13pt,scale=.8] {$f(1)^{-1}$};
\draw[HUgreen,line width=2pt] (2.7,.1) -- (2.7,-.1) node[anchor=north] {$x_n^{-1}$};
\draw[->,>=latex,HUgreen] (2.5,.1) -- (1.8,.1);
\fill[HUgreen] (-1.6-.1,-.1) rectangle (-1.6+.1,.1) node[anchor=north,inner sep=13pt,scale=.8] {$-g(0)$};
\draw[HUgreen,line width=2pt] (-2.7,.1) -- (-2.7,-.1) node[anchor=north] {$-y_1$};
\draw[->,>=latex,HUgreen] (-2.5,.1) -- (-1.8,.1);
\end{tikzpicture}
\qquad\qquad
\begin{tikzpicture}[scale=1]
\draw[->,thin,gray] (-3,0) -- (3,0);
\draw[->,thin,gray] (0,-2) -- (0,2);
\draw[blue,thick] (0,0) circle (1.5);
\draw[red,thick] (60:2) -- (0,0) -- (-60:2);
\fill[black] (60:1.5) circle (.07) node[anchor=west] {\raisebox{6pt}{$z_1 \sim q^{-n} e^{i\varphi}$}};
\fill[black] (-60:1.5) circle (.07) node[anchor=west] {$z_2 \sim q^{-n} e^{-i\varphi}$};
\draw[->,>=latex] ++(0,.2) + (60:1.2) -- +(60:.15);
\draw[->,>=latex] ++(0,-.2) + (-60:1.2) -- +(-60:.15);
\fill[UQpurple] (-.1,-.1) rectangle (.1,.1);
\draw[HUgreen,line width=2pt] (1.6,.1) -- (1.6,-.1) node[anchor=north west] {$q^{1-n}$};
\draw[->,>=latex,HUgreen] (1.45,.1) -- (.2,.1);
\draw[HUgreen,line width=2pt] (-1.6,.1) -- (-1.6,-.1) node[anchor=north east] {$-q^{-n-b}$};
\draw[->,>=latex,HUgreen] (-1.45,.1) -- (-.2,.1);
\end{tikzpicture}
\]
\caption{Local contour behaviors for asymptotics at the corner: (left) discrete Hermite kernel for $q\to 1$, where $g(0)=1, f(1)=\alpha e^{\gamma}$, (right) constant $q$, $\varphi=\frac{b}{b-1}$.}
\end{figure}
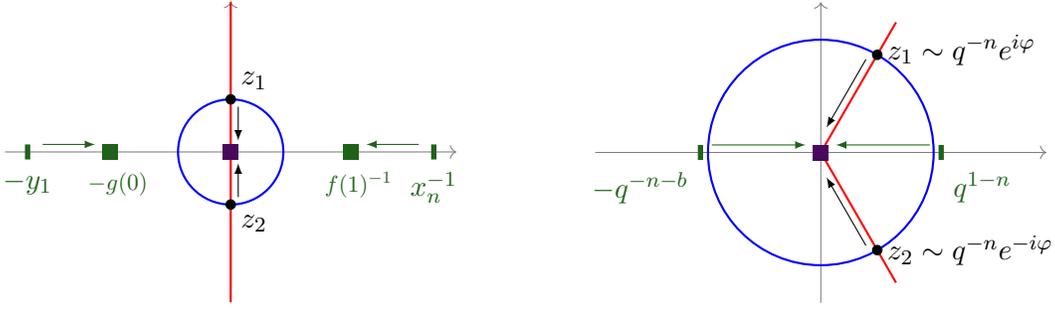

In particular, if $b = -\frac{1}{c}$, then the limit shape for $q > 1$ ends at the right corner. We conjecture that fluctuations in the corner are described by a $q$-analogue of the discrete Hermite kernel. This kernel should be close to the discrete Hermite kernel for $q$ close to $1$. 
In Fig.~\ref{fig:q-const-first-row-distributions}, we present the sample distributions of the first row length for the cases $q=1.2, k>n$ and $q=0.8, n<k$ and we see that they are close to the discrete Hermite distributions. 

\begin{figure}
  \centering
  \includegraphics{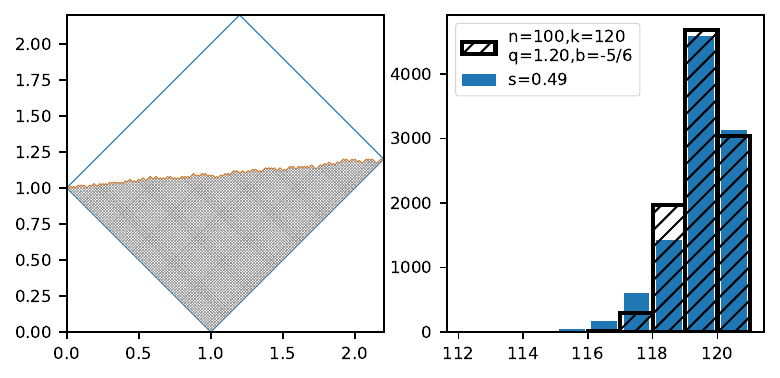}
  \includegraphics{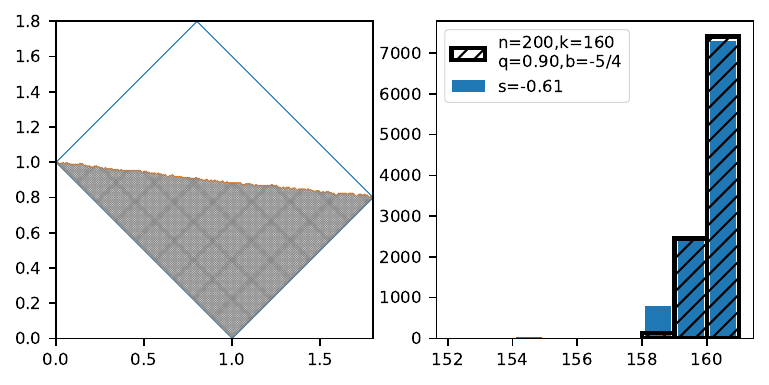}
  \caption{Plots of the limit random Young diagrams, the discrete Hermite kernel (solid blue), and the distribution of the first row (black hatched) for $n=100$, $k=120$, $q=1.2$, $b=-n/k$ and $s=0.49$ (above), and for $n=200$, $k=160$, $q=0.9$, with $s=-0.61$ (below).  }
  \label{fig:q-const-first-row-distributions}
\end{figure}

Similar to the constant case in Section \ref{sec:piec-const-spec}, we can extend the analysis in this section to the case of multiple parameters.
We can take $x_{i}=q^{\alpha_{1}i}$ for $i=1,\dotsc,\lfloor A_{1}n\rfloor$, $x_{i}=q^{\alpha_{2}i}$ for $i=\lceil A_{1}n\rceil,\dotsc,\lfloor (A_{1}+A_{2})n\rfloor$ and so on with constants $\alpha_{1},\dotsc, \alpha_{u}$ and shares $A_{1},\dotsc,A_{u}$, where $\sum_{i=1}^{u}A_{i}=1$.
Similarly denoting the constants for $y_{j}$ by $\beta_{1},\dotsc, \beta_{v}$  and shares by $B_{1},\dotsc,B_{v}$ such that $\sum_{j=1}^{v}B_{j}=c$, we get a higher order polynomial equation instead of \eqref{eq:zdz-S-eq-zero-q-const-exponentiated} that can be solved asymptotically in the same way. We then conjecture the convergence of the correlation kernel to the sine kernels with constant densities on multiple intervals as demonstrated in Fig.~\ref{fig:q-const-several-intervals}.

\begin{figure}
  \centering
  \includegraphics[width=.45\linewidth]{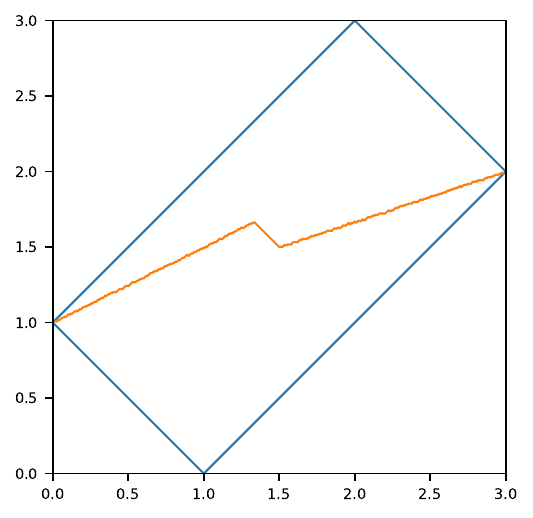}
  \includegraphics[width=.45\linewidth]{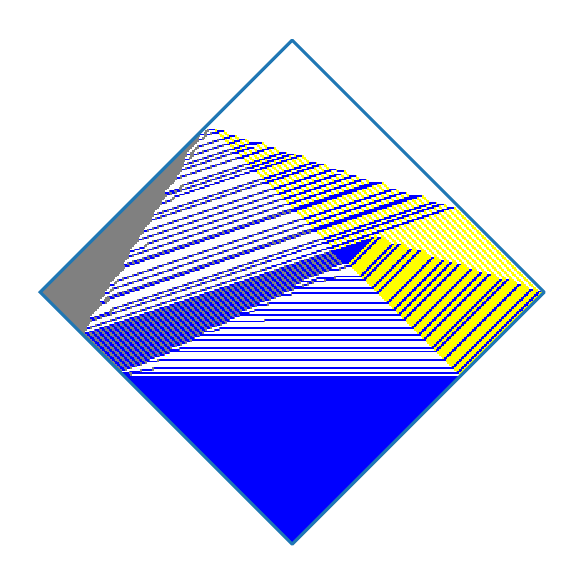}
  \caption{Plot of the random Young diagram for $q=10, n=200, k=400$ with specializations $x_{i}=q^{i}$, $y_{j}=q^{-\frac{1}{3}(j-1)}$ for $j\in[1,\dotsc,k/2]$ and $y_{j}=q^{-\frac{1}{2}(j-1)}$ for $j\in[k/2,\dotsc,k]$ (on the left) and domino tiling of Aztec diamond for $n=50, k=100$ for the same choice of parameters $x_{i},y_{j}$. }
  \label{fig:q-const-several-intervals}
\end{figure}

\section{Conclusion and open problems}
\label{sec:conclusion-outlook}

We have proven the asymptotics of determinantal ensembles in various regimes for general specializations.
In particular, we have considered limit shapes for single-interval and multi-interval support, as well as bulk, edge, Pearcey and corner fluctuations.
Nevertheless some questions remain open, and we list them below. 

\begin{enumerate}
\item In the analysis of bulk asymptotics we have used that
  $f(s)>0, g(s)\geq 0$ to demonstrate convergence to the discrete sine
  kernel. We expect this result to hold for $f(s)\geq 0$. One might
  need to estimate the number of complex roots for transcendental
  equations in this case.
\item We have demonstrated that fluctuations near the corner are
  described by the discrete Hermite kernel. It remains an open problem
  to prove that conditions for the limit shape to end in the corner
  stated in Conjecture~\ref{conj:critical_classification} are
  equivalent.
\item In the multi-interval case it would be interesting to check for
  the existence of the higher-order Airy-like (or Pearcey-like) behavior.
\item The case of principal specialization with $q=\mathrm{const}$
  considered in Section \ref{sec:constant-q} is not included in the general
  setup of the present paper. The use of asymptotic solutions allows us to
  conjecture the limit shape in this case, and we have numerical
  evidence for the fluctuations. Yet the proofs require another
  technique. For $b>0$, the limit shape degenerates to empty or
  fully-filled diagram, but it is possible that there is non-trivial
  behavior in the corner after suitable rescaling.
\item \label{Q:dual_pairs} Similar analysis for other classical dual pairs of Lie groups is
  another promising direction of research. In general, limit shape for
  the measures, given by skew Howe duality for pairs
  $(\SO_{2n+1}, \Pin_{2k})$, $(\Sp_{2n}, \Sp_{2k})$, and
  $(\Or_{2n}, \SO_{k})$ are known only for the case when all
  specialization parameters are equal to $1$ and the probability of a
  diagram is proportional to the dimension of an irreducible component
  \cite{nazarov2021skew}. The fluctuations in this case can be studied
  using semiclassical orthogonal polynomials as demonstrated in
  \cite{NNP20}, but detailed analysis of various asymptotic regimes
  have not been carried out. In Section \ref{sec:gl=2sp} we have
  formulated a hypothesis for the limit shape of symplectic diagrams
  for a general specialization which can be stated for other pairs as
  well. A free fermionic approach might be used to establish this
  result, but it requires additional techniques.
  On the other hand, $(\GL_{n},\Sp_{2k})$ Howe duality can be treated using the free fermion formalism as demonstrated in~\cite{Betea18} or a linear algebraic approach via the Eynard--Mehta theorem (as discussed in~\cite{BR05}) as demonstrated in~\cite{cuenca2024symplecticschurprocess}.
\item Study of transition probabilities for the presented limit shapes
  can be also of interest. For example, for Schur--Weyl duality the
  transition probability converges to Marchenko--Pastur law as
  demonstrated
  in~\cite{biane2001approximate,meliot2010kerovscentrallimittheorem,bufetov2013kerov}.
  We expect to see the Marchenko--Pastur distribution for the
  single-interval support of the limit shape. The multi-interval case
  is an open question.
\item Another reasonable problem is to compute the entropy of the
  measures considered in the present paper, similar to what was done
  for the Plancherel measure~\cite{bufetov2010vershik} and Schur--Weyl
  measure~\cite{mkrtchyan2014entropy}.
\item Moreover, one can study limit shapes of the lozenge or domino
  tilings, presented in Fig.~\ref{fig:lozenge-aztec-tilings}
  and Fig.~\ref{fig:gl-pearcey-lozenge-tiling} by using well-known
  techniques of~\cite{OR03,tracy2006pearcey,okounkov2007random}. In
  these cases we can consider limit shapes to be surfaces in
  three-dimensional space. Note, that rearrangement of the
  specialization parameters $x_{i}$, $y_{j}$ drastically changes the
  picture here, as demonstrated in
  Fig.~\ref{fig:gl-pearcey-lozenge-tiling-different-order}; contrast
  this with Fig.~\ref{fig:gl-pearcey-lozenge-tiling}. In the case of
  $q=\mathrm{const}$, we expect the limit surface to be piecewise
  flat, as can be seen in the right panel of
  Fig.~\ref{fig:q-const-several-intervals}.
\end{enumerate}

\begin{figure}
  \begin{center}
  \includegraphics[width=15cm]{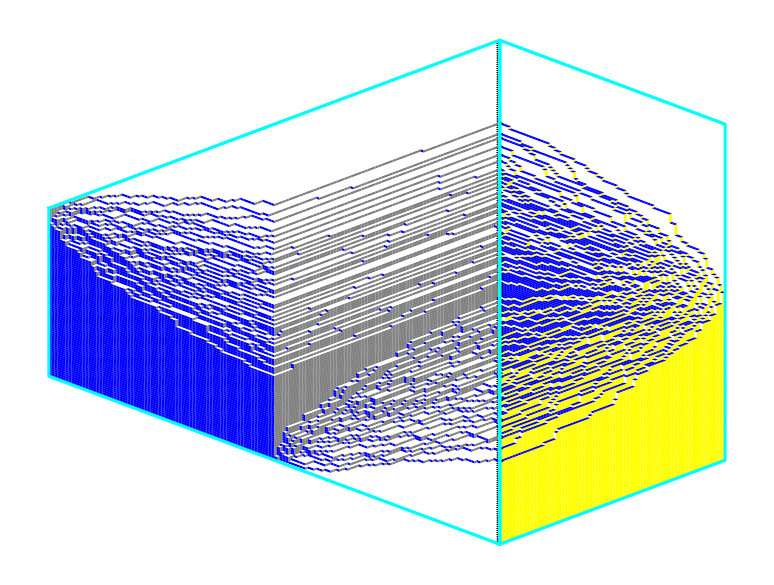}
  \end{center}
  \caption{Lozenge tiling of a skew hexagon corresponding to Pearcey
    regime for $n=40, k=80, c=2, \beta=2c+1+ 2\sqrt{c(c+1)}\approx
    9.9$  for specializations $\{x_{i}\}_{i=1}^{2n}=(1,1\dotsc,1)$ and
    $\{y_{j}\}_{j=1}^{2k}=(\beta^{-1},\dotsc,\beta^{-1}, \beta,\dotsc,\beta)$.}
  \label{fig:gl-pearcey-lozenge-tiling-different-order}
\end{figure}

\appendix
\section{Sampling code}

We provide some code to generate the samples using \textsc{SageMath}~\cite{sage}.
However, we have used specialized code to generate our figures.

\begin{lstlisting}
def sample(n, k, f, g, **kwds):
    M = matrix([[0 if random() <= 1 / (1 + f(i/n) * g(j/k)) else 1
                 for j in range(1,k+1)] for i in range(1,n+1)])
    P,Q = RSK(M, insertion=RSK.rules.dualRSK)
    data = list(P.shape())
    data += [0] * (n - len(data))
    data = (((val-i)/n, (val+i)/n) for i,val in enumerate(data))
    P = polygon2d([(-1,1), (0,0), (k/n,k/n), (k/n-1,k/n+1)],
                  color='black', fill=False)
    P += line(data, thickness=2, **kwds)
    P.set_aspect_ratio(1)
    return P
\end{lstlisting}

We use this as

\begin{lstlisting}
sage: sample(200, 300, lambda x: x^.2, lambda y: y^5).show(figsize=20)
\end{lstlisting}

\bibliographystyle{alpha}
\bibliography{shapes}{}
\end{document}